\documentclass[twoside]{irmaems}
\usepackage[latin1]{inputenc}
\usepackage{amssymb} 

\usepackage{amsmath} 
\usepackage{mathrsfs}
\usepackage{amsbsy} 
\usepackage[dvips]{graphicx}
\usepackage{amscd} 
\usepackage{latexsym}

\renewcommand\theequation{\arabic{section}.\arabic{equation}}

\theoremstyle{plain} 
\newtheorem{theorem}{Theorem}[section] 
\newtheorem{lemma}[theorem]{Lemma} 
\newtheorem{corollary}[theorem]{Corollary} 
\newtheorem{proposition}[theorem]{Proposition} 
\newtheorem{proposition-definition}[theorem]{Proposition-Definition}

\newtheorem{problem}{Problem}

\theoremstyle{definition} 
\newtheorem{definition}{Definition}[section] 

\theoremstyle{remark} 
\newtheorem{remark}[theorem]{Remark} 
\newtheorem{example}{Example}[section]

\newcommand{\Z}{{\mathbb{Z}}} 
\newcommand{\C}{{\mathbb{C}}} 
\newcommand{\Q}{{\mathbb{Q}}} 

\newcommand{\R}{{\mathbb{R}}} 
\newcommand{\N}{{\mathbb{N}}} 

\newcommand{\si}{{\mathscr{S}}_{0,\infty}} 
\newcommand{\M}{{{\cal M}}} 
\newcommand{\B}{{{\cal B}}}

\newcommand{\s}{{\sigma}}

\def\pa{{\partial}}

\newcommand{\al}{\widehat{\alpha}} 
\newcommand{\be}{\widehat{\beta}} 
\newcommand{\T}{T^{*}}

\newcommand{\aps}{{{\alpha^{*}}}} 
 
\newcommand{\bps}{{{\beta^{*}}}} 
 \newcommand{\als}{{\overline{\alpha}}}    
 \newcommand{\bls}{{\overline{\beta}}}

\begin{document}

\title{Asymptotically rigid mapping class groups and Thompson's groups}

\author{Louis Funar, Christophe Kapoudjian and Vlad Sergiescu 
}
\begin{small}\address{
Institut Fourier BP 74, UMR 5582 \\      
University of Grenoble I \\      
38402 Saint-Martin-d'H\`eres cedex, France  \\      
e-mail: \,{\tt \{funar,sergiesc\}@fourier.ujf-grenoble.fr}  \\
[4pt]
Laboratoire Emile Picard, UMR 5580\\      
University of Toulouse   III\\      
31062 Toulouse cedex 4, France\\      
e-mail:\, {\tt ckapoudj@picard.ups-tlse.fr}\\
%[4pt]
%Institut Fourier BP 74, UMR 5582 \\      
%University of Grenoble I \\      
%38402 Saint-Martin-d'H\`eres cedex, France  \\      
%e-mail: \,{\tt sergiesc@fourier.ujf-grenoble.fr}  \\
} 
\end{small}

\maketitle

\begin{abstract}
We consider Thompson's groups from the perspective of 
mapping class groups of surfaces of infinite type. This point of view leads 
us to the braided Thompson groups, which are extensions 
of Thompson's groups by infinite (spherical) braid groups. 
We will outline the main features of these groups and some applications to 
the quantization of Teichm\"uller spaces. 
The chapter provides an introduction to the subject 
with an emphasis on some of the authors results.

\vspace{0.1cm}
\noindent 2000 MSC Classification: 57 M 07, 20 F 36, 20 F 38, 57 N 05.  
 
\noindent Keywords: mapping class group, Thompson group, Ptolemy groupoid, infinite braid group, 
quantization, Teichm\"uller space, braided Thompson group, Euler class, 
discrete Godbillon-Vey class, Hatcher-Thurston complex, combable group, 
finitely presented group, central extension, Grothendieck-Teichm\"uller group. 
\end{abstract}

\begin{small}
\tableofcontents
 \end{small}

\section{Introduction}
The purpose of this chapter is to present 
the recently developed interaction between
mapping class groups of surfaces, including braid groups, and
Richard J. Thompson's groups $F,T$ and $V$. 
We follow here the present authors'
geometrical approach, while giving some hints to the algebraic 
developments of Brin
and Dehornoy and the quasi-conformal approach of de Faria, Gardiner and Harvey.

\vspace{0.1cm}\noindent 
When compared to mapping-class groups, already thoroughly studied by Dehn
and Nielsen, Thompson's groups appear quite recent. Introduced by Richard J. 
Thompson in the middle of the 1960s, they originally developed from
algebraic logic; however, a PL representation of them was immediately obtained.

\vspace{0.1cm}\noindent 
Recall that Thompson's group $F$ is the group of PL homeomorphisms of $[0,1]$
which locally are of the form $2^n+\frac{p}{2^q}$, with breaks in 
$\Z\left[\frac{1}{2}\right]$.
The group $T$ acts in a similar way on the unit circle $S^1$. 
The group $V$ acts by left continuous bijections on $[0,1]$ as 
a group of affine interval exchanges. This 
action may be lifted to a continous one on the triadic
Cantor set.

\vspace{0.1cm}\noindent 
By conjugating these groups via the Farey map sending the rationals to the
dyadics, one obtains a similar definition as 
groups of piecewise ${\rm PSL}(2,\Z)$
maps with rational breakpoints; this definition already has a certain
2-dimensional flavour.
Observe that Thompson's groups act near the boundary of the hyperbolic
disk and thus near the boundary of the infinite binary tree. 
This observation played a basic
role in the beginning of the material discussed here.
From this point of view Thompson's group $T$ is a piecewise generalisation of
${\rm SL}(2,\Z)$; the mapping class group is a multi-handle generalisation of
${\rm SL}(2, \Z)$. 
In the same vein ${\rm SL}(n,\Z)$ is an arithmetic generalization 
and ${\rm Aut}(F_n)$ is a non-commutative one.
We also note that, following Thurston, the mapping class group 
$\Gamma_g$ acts on the
boundary of the Teichm\"uller space and preserves its 
piecewise projective integral structure.
  
\vspace{0.1cm}\noindent 
Another way to encode these groups is to consider pairs of binary trees
which represent dyadic subdivisions. Dually, this data gives a simplicial
bijection of the complementary forests, 
called partial automorphism of the infinite
binary tree $\tau_2$. Of course this does not extend as an automorphism of
$\tau_2$. However, one
observes the following simple but essential fact: if one thickens the
infinite tree $\tau_2$ to a surface $\si$, 
then the corresponding partial
homeomorphisms extend to the entire surface $\si$. 
Thus, two objects
appear here: the surface $\si$  and the mapping class group 
which lifts the elements of a Thompson group.
 
 \vspace{0.1cm}\noindent     
To make definitions precise, we are forced to endow the surface 
$\si$ with a
rigid structure which encodes its tree-like aspect. 
The homeomorphisms we
consider are asymptotically rigid, i.e. they preserve the rigid structure
outside a compact sub-surface. These homeomorphisms give rise to the
asymptotically rigid mapping class groups.

\vspace{0.1cm}\noindent 
We now give some details on the structure of this chapter.
We present in Section 1 various constructions of groups and spaces and
explain how the group $T$ itself is a mapping class group of 
$\si$. Next, we
introduce the (historically) first relation between Thompson groups and braid
groups, namely the extension:
\[ 1\to B_{\infty}\to A_T \to T \to 1 \]
In order to avoid working with non-finitely supported braids, the authors
chose to build $A_T$ from a convenient geometric homomorphism
\[  T\to {\rm Out}(B_{\infty})\]
However, retrospectively, while having definite advantages, this choice may
not have been the best.
The main theorem of \cite{gr-se} 
says that the group $A_T$ is almost acyclic -- the
corresponding group $A_{F'}$ being acyclic. The proofs of these 
theorems are quite
involved and far from the geometric-combinatorial topics  
discussed in the rest of
this chapter; this is why we shall present them rather sketchyly. However, we do
describe the group $A_T$ as a mapping class group. It is actually 
while trying to extend the Burau representation from 
$B_{\infty}$ to $A_T$ that the notion of
asymptotically rigid mapping class group was formulated.

\vspace{0.1cm}\noindent 
The next two sections, 2 and 3, are of central importance. 
We show that a group $\B$
which is an asymptotically rigid 
mapping class group of $\si$ and surjects onto $V$ is finitely presented.
While the acyclicity theorem mentioned above was formulated 
on the basis of homotopy-theoretic evidence, 
the group $\B$ and its finite presentability came largely
from conformal field theory  evidence. The Moore-Seiberg 
duality groupoid is finitely presented, a fact
mathematically established in \cite{ba-ki1,ba-ki2,fu-ge}.

\vspace{0.1cm}\noindent 
We begin by introducing Penner's Ptolemy groupoid, partly issued 
from the conformal field theory  work
of Friedan and Shenker. Its objects are ideal tesselations and 
its morphisms are compositions of flips. 
We then explain how Thompson's groups fit into this setting. 
A basic observation here is that the Ptolemy groupoid is isomorphic to a
sub-groupoid of the Moore-Seiberg stable duality groupoid. 
This duality groupoid is related in turn to a Hatcher-Thurston type 
complex for the surface $\si$. One main
result is that this complex is simply connected.

\vspace{0.1cm}\noindent 
Section 3 applies all this to the asymptotically rigid 
mapping class group $\B$ of $\si$.

\vspace{0.1cm}\noindent 
Let us emphasize here that our notion of asymptotically 
rigid mapping class group
is different from the asymptotic mapping class group considered recently by
various authors (see the chapter written by Matsuzaki in volume IV of this handbook).

\vspace{0.1cm}\noindent 
The kernel of the morphism from $\B$ onto $V$ is the compactly 
supported mapping class group of $\si$.  
Let us note that the group  $\B$  contains all genus zero 
mapping class groups as well as the braid groups. 
The main theorem states that $\B$  is
a finitely presented group. 
A quite compact symmetric set of relations is produced as well.

\vspace{0.1cm}\noindent
Section 4 is dedicated to the braided Ptolemy-Thompson group $T^*$. 
This is an extension of $T$ by the braid group $B_{\infty}$. 
It is an asymptotically rigid mapping class group of
$\si$ of a special kind. It is a simpler group than $A_T$ and 
will be used in Sections 5 and 6.
We prove that $T^*$, like $\B$, is a finitely presented group. 
We note that so far, $A_T$ is only known to be finitely generated.

\vspace{0.1cm}\noindent
In Section 5, we consider a relative abelianisation of $T^*$:
\[
1\to \Z=B_{\infty}/[B_{\infty}, B_{\infty}]\to T^*/[B_{\infty}, B_{\infty}]\to T\to 1
\]
We prove that this central extension is classified by a multiple of the
Euler class of $T$ that we detect to be $12\chi$, where $\chi$ 
is the Euler class pulled-back to $T$.
This fact eventually allows us to classify the 
dilogarithmic projective extension
of $T$ which arises in the quantization of the 
Teichm\"uller theory, as we explain as well.

\vspace{0.1cm}\noindent
In Section 6 we discuss an infinite genus mapping class group that maps onto $V$
which is proved to be (at least) finitely generated. It also has the property
of being homologically equivalent to the stable mapping class group. As already
mentioned, the proofs involve as a key ingredient the group $T^*$.

\vspace{0.1cm}\noindent
In Section 7 we introduce a simplicial unified approach to the various 
extensions of the group $V$. 
This includes the extension $BV$ of Matt Brin and Patrick Dehornoy 
coming from categories with multiplication and from the geometry of
algebraic laws, respectively.
Moreover, one can approach in this way the action of the 
Grothendieck-Teichm\"uller group on a $V$-completion $\widehat{\B}$ 
of $\B$, thus getting a quite neat presentation of the
entire setting.

\vspace{0.1cm}\noindent
A sample of open questions is contained in the final section.

\vspace{0.1cm}\noindent
We would like to dedicate these notes to the memory of Peter Greenberg and of
Alexander Reznikov. Their work is inspiring us forever.

\vspace{0.2cm}\noindent 
{\bf Acknowledgments.} The authors are indebted to L. Bartholdi,  
M. Bridson, M. Brin, J. Burillo, D. Calegari, F. Cohen, 
P. Dehornoy, D. Epstein, V. Fock, 
R. Geogheghan, E. Ghys, S. Goncharov, F. Gonz\'alez-Acu\~na,  
V. Guba,  P. Haissinsky, B. Harvey, V. Jones,  R. Kashaev, 
F. Labourie, P. Lochak, J. Morava,  H. Moriyoshi,   P. Pansu, A. Papadopolous, B. Penner, 
C. Pittet, M. Sapir, L. Schneps and H. Short for useful discussions and 
suggestions concerning this subject during the last few years.

\section{From Thompson's groups to mapping class groups of surfaces}

\subsection{Three equivalent definitions of the Thompson groups}

\subsubsection*{Groups of piecewise affine bijections}

{\it Thompson's group  $F$} \index{Thompson group $F$}\index{group! Thompson $F$} 
is the group of continuous and nondecreasing bijections of the interval $[0,1]$ which are piecewise dyadic affine. In other words, for each $f\in F$, there exist two subdivisions of $[0,1]$, $a_0=0<a_1<\ldots<a_n=1$ and $b_0=0<b_1\ldots<b_n$, with $n\in \N^*$, such that :
\begin{enumerate}
\item $a_{i+1}-a_i$ and $b_{i+1}-b_i$ belong to $\{\frac{1}{2^k},\; k\in \N\}$;
\item  the restriction of $f$ to $[a_i,a_{i+1}]$ is the unique nondecreasing affine map  onto  $[b_i, b_{i+1}]$.
\end{enumerate}

\vspace{0.1cm}\noindent 
Therefore, an element of $F$ is completely determined by the data of two dyadic subdivisions of  $[0,1]$ having the same cardinality.

\vspace{0.1cm}\noindent 
Let us identify the circle to the quotient space $[0,1]/0\sim 1$. {\it Thompson's group $T$} \index{Thompson group $T$}\index{group! Thompson $T$}  is the group of continuous and nondecreasing bijections of the circle which are piecewise dyadic affine.
In other words, for each $g\in T$, there exist two  subdivisions of $[0,1]$, $a_0=0<a_1<\ldots<a_n=1$ and $b_0=0<b_1\ldots<b_n$, with  $n\in \N^*$, such that :
\begin{enumerate}
\item $a_{i+1}-a_i$ and $b_{i+1}-b_i$ belong to $\{\frac{1}{2^k},\; k\in \N\}$.  
\item There exists $i_0\in \{1,\ldots, n\}$, such that, pour each $i\in \{0,\ldots, n-1\}$, the restriction of $g$ to $[a_i,a_{i+1}]$ is the unique nondecreasing map onto $[b_{i+i_0}, b_{i+i_0+1}]$. The indices must be understood  modulo $n$.
\end{enumerate}

\vspace{0.1cm}\noindent 
Therefore, an element of $T$ is completely determined by the data of two dyadic subdivisions of $[0,1]$ having the same cardinality, 
say $n\in \N^*$, plus an integer  $i_{0}$ mod $n$.

\vspace{0.1cm}\noindent 
Finally, {\it Thompson's group  $V$} \index{Thompson group $V$}\index{group! Thompson $V$}  is the group of bijections of $[0,1[$, which are right-continuous at each point, piecewise nondecreasing and dyadic affine. In other words, for each $h\in V$, there exist two subdivisions of $[0,1]$, $a_0=0<a_1<\ldots<a_n=1$ and $b_0=0<b_1\ldots<b_n$, with $n\in \N^*$, such that :
\begin{enumerate}
\item  $a_{i+1}-a_i$ and $b_{i+1}-b_i$ belong to $\{\frac{1}{2^k},\; k\in \N\}$;
\item there exists a permutation $\sigma\in \mathfrak{S}_n$, such that, for each  $i\in\{1,\ldots,n\}$, the restriction of  $h$ to $[a_{i-1},a_i[$ is the unique  nondecreasing affine map onto $[b_{\sigma(i)-1}, b_{\sigma(i)}[$.
\end{enumerate}
It follows that an element $h$ of  $V$ is completely determined by the data of two dyadic subdivisions of $[0,1]$ having the same cardinality, say $n\in \N^*$, plus a permutation $\sigma\in \mathfrak{S}_n$. Denoting $I_{i}=[a_{i-1},a_i]$ and $J_i= [b_{i-1}, b_i]$, these data can be summarized into a triple $((J_i)_{1\leq i\leq n},(I_i)_{1\leq i\leq n},\sigma\in \mathfrak{S}_n)$.

 \vspace{0.1cm}\noindent 
Such a triple is not uniquely determined by the element $h$. Indeed, a refinement of the subdivisions gives rise to a new triple defining the same $h$. This remark also applies to elements of $F$ and $T$.

\vspace{0.1cm}\noindent 
The inclusion  $F\subset T$ is obvious. The identification of the integer $i_{0}$ mod $n$ to the cyclic permutation $\sigma: k\mapsto k+i_{0}$ yields the inclusion $T\subset V$.

\vspace{0.1cm}\noindent 
R. Thompson proved that $F,T$ and $V$ are finitely presented groups and that $T$ and $V$ are simple \index{simple group} (cf. \cite{ca-fl-pa}). The group $F$ is not perfect ($F/[F,F]$ is isomorphic to $\Z^ 2$), but $F'=[F,F]$  is simple. However, $F'$ is not finitely generated (this is related to the fact that an element $f$ of  $F$ lies in $F'$ if and only if its support is included  in $]0,1[$).

\vspace{0.1cm}\noindent 
Historically, Thompson's groups $T$ and $V$ are the first examples of infinite simple and finitely presented groups. Unlike $F$, they are not torsion-free.

\begin{figure} 
\begin{center}
\includegraphics{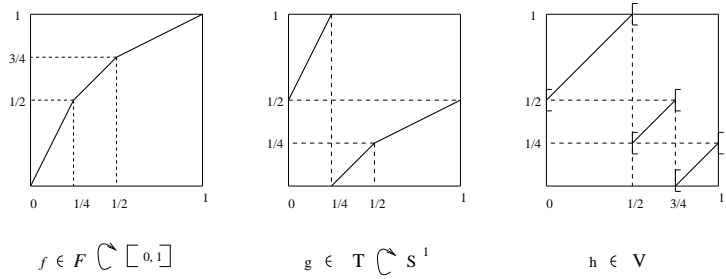}
\caption{Piecewise dyadic affine bijections representing elements of Thompson's groups}
\end{center}
\end{figure}

\subsubsection*{Groups of diagrams of finite binary trees}

A {\it finite binary rooted planar tree} \index{tree} is a finite planar tree having a unique 2-valent vertex, called the {\it root}, a set of monovalent vertices called the {\it leaves}, and whose other vertices are 3-valent. The planarity of the tree provides a canonical labelling of its leaves, in the following way. Assuming that the plane is oriented, the leaves are labelled from  1 to $n$, from left to right, the root being at the top and the leaves at the bottom.

\vspace{0.1cm}\noindent 
There exists a bijection between the set of dyadic subdivisions of $[0,1]$ and the set of  finite binary rooted planar trees. Indeed,  given such a tree, one may label its  vertices by dyadic intervals in the following way. First, the root is labelled by $[0,1]$. Suppose that a vertex is labelled by $I=[\frac{k}{2^n}, \frac{k+1}{2^n}]$, then its two descendant vertices 
are labelled by the two halves  $I$: $[\frac{k}{2^n}, \frac{2k+1}{2^{n+1}}]$ for the left 
one and  $[\frac{2k+1}{2^{n+1}},\frac{k+1}{2^n}]$ for the right one. 
Finally, the dyadic subdivision associated to the tree is 
the sequence of intervals which label its  leaves.

\vspace{0.1cm}\noindent 
As we have just seen, an element of Thompson's group $V$ is defined by 
the data of two dyadic subdivisions of $[0,1]$, with the same cardinality $n$, plus a permutation
 $\sigma\in \mathfrak{S}_n$. This amounts to encoding it by a pair of finite binary rooted trees with the same number of leaves $n\in \N^*$, plus a  permutation $\sigma\in \mathfrak{S}_n$. 

\vspace{0.1cm}\noindent 
Thus, an element $h$ of $V$ is represented by a triple $(\tau_1,\tau_0,\sigma)$, where $\tau_0$ and $\tau_1$ have the same number of leaves $n\in \N^*$, and  $\sigma$ belongs to 
the symmetric group $\mathfrak{S}_n$. Such a triple will be called a {\it symbol} for $h$.
It is convenient to interpret the permutation $\sigma$ as the bijection $\varphi_{\sigma}$ which
maps the $i$-th leaf of the source tree  $\tau_0$ to the $\sigma(i)$-th 
leaf of the target tree $\tau_1$.
When  $h$ belongs to  $F$, the permutation  $\sigma$, which is the identity, is not 
represented, and the symbol reduces to a pair of trees $(\tau_{1},\tau_{0})$. 
When $h$ belongs to  $T$, the cyclic permutation  is graphically materialized 
by a small circle surrounding the leaf number $\sigma(1)$ of $\tau_1$.

\vspace{0.1cm}\noindent 
One introduces the following equivalence relation on the set of symbols : two symbols are equivalent if they represent the same 
element of $V$. One denotes by $[\tau_1,\tau_0,\sigma]$ the equivalence class of the symbol. Therefore,
 $V$ is (in bijection with) the set of equivalence classes of symbols.  
The composition law of piecewise dyadic 
affine bijections is pushed out on the set of equivalence classes of symbols in the following way. In order to define 
 $[\tau_1',\tau_0',\sigma']\cdot[\tau_1,\tau_0,\sigma]$, one may suppose, at the price of refining both symbols, that
  the tree   $\tau_1$ coincides with the tree $\tau_0'$. The the product of the two symbols is 
 
\[ [\tau_1',\tau_1,\sigma']\cdot[\tau_1,\tau_0,\sigma]= [\tau_1',\tau_0,\sigma'\circ\sigma]. \]
   The neutral element is represented by any symbol  $(\tau, \tau, 1)$, for any finite binary rooted planar tree 
   $\tau$. The inverse of $[\tau_1,\tau_0,\sigma]$ is  $[\tau_0,\tau_1,\sigma^{-1}]$.
   
\vspace{0.1cm}\noindent 
   It follows that $V$ is isomorphic to the group of  equivalence classes of symbols endowed with this internal law.

\begin{figure}
\begin{center}
\includegraphics[scale=0.9]{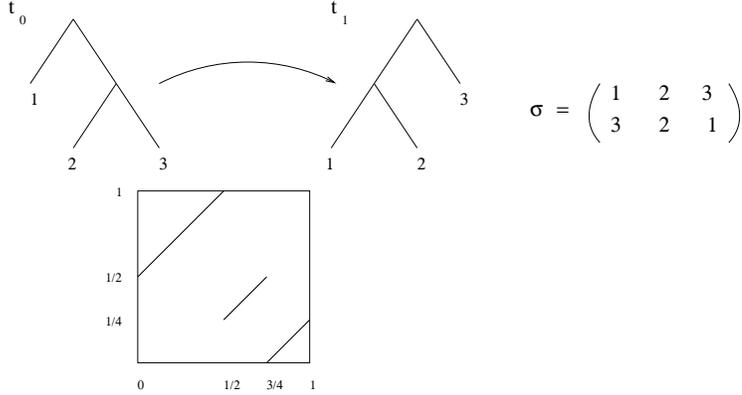}
\caption{Symbolic representation of an element of $V$, with its corresponding representation as a piecewise dyadic affine bijection}
\end{center}
\end{figure}

 \subsubsection*{Partial automorphisms of trees (\cite{ka-se})}

The beginning of the article \cite{ka-se} formalizes a change of point of view, consisting in considering, not the finite binary trees, 
but their complements in the infinite binary tree.

\vspace{0.1cm}\noindent 
Let ${\cal T}_2$ be the infinite binary rooted planar tree (all its vertices other than the root are 
3-valent). Each finite binary rooted planar tree $\tau$ can be embedded in a unique way into ${\cal T}_2$, assuming that the embedding maps the root of $\tau$ onto the root of ${\cal T}_2$, and respects the orientation.  Therefore, $\tau$ may be identified with 
a subtree  ${\cal T}_2$, whose root coincides with that of  ${\cal T}_2$. 

\begin{definition}[ cf. \cite{ka-se}]
A {\em partial isomorphism} \index{partial isomorphism} of ${\cal T}_2$ consists of the data of two finite binary rooted subtrees $\tau_0$ and $\tau_1$ of ${\cal T}_{2}$ having the same number of 
leaves $n\in\N^*$, and an isomorphism $q: {\cal T}_2\setminus \tau_0\rightarrow {\cal T}_2\setminus \tau_1$. The complements 
of  $\tau_0$ and $\tau_1$ have $n$ components, each one isomorphic to  ${\cal T}_2$, 
which are enumerated from 1 to $n$ according to the labelling of the leaves 
of the trees  $\tau_0$ and $\tau_1$. Thus, 
$ {\cal T}_2\setminus \tau_0=T^1_0\cup\ldots\cup T^n_0$ and $ {\cal T}_2\setminus \tau_1=T^1_1\cup\ldots\cup T^n_1$ where the $T^i_j$'s are the 
connected components. 
Equivalently, the partial isomorphism of ${\cal T}_2$
is given by a permutation $\sigma\in \mathfrak{S}_n$ and,  for $i=1,\ldots,n$, an  isomorphism
 $q_i: T^i_0\rightarrow T^{\sigma(i)}_1$. 

\vspace{0.1cm}\noindent 
Two partial automorphisms $q$ and $r$ can be composed if and only if the target of $r$ coincides with the source of $r$. One gets the partial automorphism $q\circ r$.
The composition provides a structure of inverse monoid on the  set of partial automorphisms, which is denoted ${\rm Fred}({\cal T}_2)$.

\begin{center}
\includegraphics[scale=0.9]{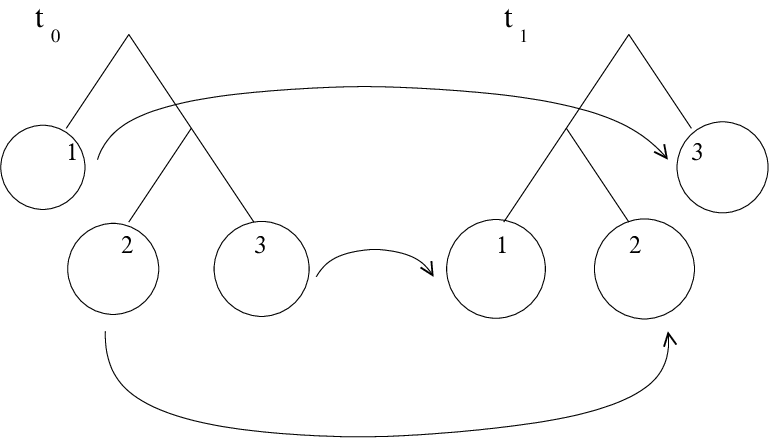}
%\caption{A partial automorphism of ${\cal T}_2$}
\end{center}

\end{definition}

\vspace{0.1cm}\noindent 
One may construct a group from  ${\rm Fred}({\cal T}_2)$. Let $\partial {\cal T}_2$ be the boundary of  ${\cal T}_2$ (also called the set of  ``ends'' of ${\cal T}_2$) 
endowed with its usual topology, for which it is a Cantor set.

\vspace{0.1cm}\noindent 
The point is that a partial automorphism does not act (globally) on the tree, but does  act on its boundary.
 One has therefore a morphism  ${\rm Fred}({\cal T}_2)\rightarrow {\rm Homeo}(\partial{\cal T}_2)$, 
 whose image  $N$ is the {\it spheromorphism group of  Neretin}. \index{spheromorphism group of Neretin}

\vspace{0.1cm}\noindent 
 Let now  ${\rm Fred}^+({\cal T}_2)$ be the sub-monoid of  ${\rm Fred}({\cal T}_2)$, whose elements are the partial automorphisms which respect the local orientation of the edges.
Thompson's group  $V$ can be viewed as the subgroup of  $N$ which is the image of 
${\rm Fred}^+({\cal T}_2)$ by the above morphism.

\begin{remark}
There exists a Neretin group  $N_{p}$ for each integer $p\geq 2$, as introduced in \cite{ne} (with different notation). 
They are constructed in a similar way as $N$, by replacing 
the dyadic complete (rooted or unrooted) tree by the p-adic complete (rooted or unrooted) tree.
They are proposed as combinatorial  or $p$-adic analogues of the diffeomorphism group of the circle.
Some aspects of this analogy have been studied in \cite{ka0}.
\end{remark}

\subsection{Some properties of Thompson's groups}
Most readers of this section are probably more comfortable with the
mapping class group than with Thompson's groups. 
Therefore, we think that it will be useful to
gather here some of the classical and less classical properties 
of Thompson's groups. 
There is a fair amount of randomness in our choices  
and the only thing we would really like  to emphasize  is their ubiquity.
Thompson's groups became known in algebra because $T$ and $V$ 
were the first  infinite finitely presented simple groups. 
They were preceded by Higman's example
of an infinite finitely generated simple group in 1951. 
More recently, Burger and Mozes (see \cite{burg-mozes}) 
constructed an example which is also without torsion.

\vspace{0.1cm}\noindent 
Thompson used $F$ and $V$ to give new examples of groups with
an unsolvable word problem and also in his algebraic characterisation of 
groups with a solvable word problem (see \cite{T}) as  being those which  
embed in a finitely generated simple subgroup of a finitely presented group. 
The group $F$ was rediscovered in homotopy theory, as a universal conjugacy
idempotent, and later in universal algebra.
We refer to \cite{ca-fl-pa} for an  introduction from scratch to several
aspects of Thompson's groups, including their presentations, and also their
piecewise linear and projective representations. 
One can find as well an introduction to
the amenability problem for $F$, including a proof of the
Brin-Squier-Thompson theorem that $F$ does not contain a 
free group of rank 2. Last but not
least, one can find a list of the merely 25 notations in the literature for
$F$, $T$ and $V$. Fortunately, after \cite{ca-fl-pa} appeared, 
the notation has almost
stabilized.

\vspace{0.1cm}\noindent 
We also mention the survey \cite{ser} for various other aspects and  \cite{Gh} and 
\cite{navas} for the general topic of homeomorphisms of the circle.

\vspace{0.1cm}\noindent 
The groups $F$, $T$ and $V$ are actually ${\rm FP}_{\infty}$, 
i.e. they have classifying spaces with finite skeleton 
in each dimension; this was first proved by Brown  
and Geoghegan (see \cite{BrGe,br1}). 
Let us mention what is the rational cohomology of these
groups, computed by Ghys and Sergiescu in \cite{gh-se} and Brown in 
\cite{br2}. First, $H^*(F;\Q)$ 
is the product between the divided powers algebra
on one generator of degree 2 and the cohomology algebra of the
2-torus.

\vspace{0.1cm}\noindent 
The cohomology of $T$ is the quotient $\Q(\chi,\alpha)/\chi\cdot \alpha$, 
where $\chi$ is the Euler class and $\alpha$ a (discrete) Godbillon-Vey class. 
In what concerns the group  $V$ its rational cohomology vanishes
in each dimension. See \cite{ser} for more results with  either 
$\Z$ or with twisted coefficients.

\vspace{0.1cm}\noindent 
Here are other properties of these groups involving cohomology.
Using a smoothening of Thompson's group it is proved in \cite{gh-se} that
there is
a representation  $\pi_1(\Sigma_{12})\to {\rm Diff}(S^1)$ 
having Euler number $1$  and 
an invariant Cantor set.

\vspace{0.1cm}\noindent 
Reznikov showed that the group $T$ does not have Kazhdan's property $T$  
(see \cite{reznikov}),
and later Farley \cite{farley} proved that it has Haagerup property AT 
(also called a-T-menability). Therefore it
verifies the Baum-Connes conjecture (see also \cite{Fa}).
Napier and Ramachandran proved that $F$ is not a K\"ahler group \cite{rama}.
Cyclic cocycles on $T$ were introduced in \cite{oyko}. The group $T$ 
in relation with the symplectic automorphisms of $\mathbb C\mathbb P^2$ 
was considered by Usnich in \cite{usnich}. 

\vspace{0.1cm}\noindent 
A theorem of Brin \cite{brin-ihes} 
states that the group of outer automorphisms of $T$ is 
$\Z/2\Z$. Furthermore, in \cite{BCR} the authors computed the 
abstract commensurator of $F$. 
Using the above mentioned smoothening, it is proved in \cite{gh-se} that all
rotation numbers of elements in $T$ are rational. 
New direct proofs were given by
Calegari (\cite{calegari}), Liousse (\cite{liousse}) and 
Kleptsyn (unpublished).

\vspace{0.1cm}\noindent 
For the connection of $F$ and $T$ with the piecewise projective $C^1$-homeomorphisms, see for instance \cite{Gr,gr} and \cite{ma}.
The group $F$ is naturally connected to associativity in various frameworks
\cite{de2,GeGu,FL}. See also \cite{bri1,bri2} for the group $V$.

\vspace{0.1cm}\noindent 
Brin proved that the rank 2 free group is a limit of Thompson's group $F$
 (\cite{brin3}). Complexity
aspects were considered in \cite{birget}. Guba (\cite{Gu2}) 
showed that the Dehn function
for $F$ is quadratic. 
The group $F$ was studied in cryptography in \cite{RST,matucci,BT}.
Thompson's groups were studied from the viewpoint of $\C^*$-algebras and von
Neumann algebras; see for instance Jolissaint (\cite{Joli}) and 
Haagerup-Picioroaga \cite{HaagPi}.

\vspace{0.1cm}\noindent 
On the edge of logic and group theory, the interpretation of arithmetic 
in Thompson's groups 
was investigated by Bardakov-Tolstykh (\cite{BT}) and Altinel-Muranov (\cite{AM}).
Let us finally mention the work of Guba and Sapir on Thompson's groups
as diagram groups; see for instance \cite{GuS}.

\vspace{0.1cm}\noindent 
Let us emphasize here that we avoided to speak on
generalisations of Thompson's groups: 
this topic is pretty large and we think it would
not  be at its place here.
Let us close this section by mentioning again that our
choice was just to mention some developments related 
to Thompson's groups from the unique angle of ubiquity.

\subsection{Thompson's group  $T$ as a mapping class group 
of a surface}\label{thommcg}

The article \cite{ka-se} is partly devoted to developing the notion of an asymptotically rigid homeomorphism.

\begin{definition}[following \cite{ka-se}]
\begin{enumerate}
\item Let  $\mathscr{S}_{0,\infty}$ be the oriented surface of genus zero, 
which is the following  inductive limit of compact oriented genus zero surfaces with boundary 
$\mathscr{S}_{n}$ : Starting with a cylinder $\mathscr{S}_{1}$, one gets $\mathscr{S}_{n+1}$ from 
 $\mathscr{S}_{n}$ by gluing a pair of pants  (i.e. a three-holed sphere) along 
 each boundary circle of $\mathscr{S}_{n}$. This construction yields, for each $n\geq 1$, an embedding 
  $\mathscr{S}_{n}\hookrightarrow \mathscr{S}_{n+1}$, 
  with an orientation on $\mathscr{S}_{n+1}$ compatible with that of  $\mathscr{S}_{n}$. 
  The resulting inductive limit (in the topological category) of the $\mathscr{S}_{n}$'s is the surface
   $\mathscr{S}_{0,\infty}$: $$\mathscr{S}_{0,\infty}={\displaystyle \lim_{\stackrel{\rightarrow}{n}} \mathscr{S}_{n}}$$
  
\item By the above construction, the surface $\mathscr{S}_{0,\infty}$ is the union of a cylinder and of countably many 
pairs of pants.  This topological decomposition of  $\mathscr{S}_{0,\infty}$
will be called the {\em canonical pair of pants decomposition}. \index{canonical pants decomposition}
\end{enumerate}

\end{definition}

\vspace{0.1cm}\noindent 
The set of isotopy classes of  orientation-preserving  homeomorphisms  of $\mathscr{S}_{0,\infty}$ 
is an uncountable group. The group operation is map composition. 
By restricting to a certain type of homeomorphisms (called asymptotically rigid),
we shall obtain countable subgroups. We first need to complete the canonical decomposition to a richer structure.

\vspace{0.1cm}\noindent 
Let us choose an involutive homeomorphism $j$ of $\mathscr{S}_{0,\infty}$ which reverses the 
orientation, stabilizes 
each pair of pants of its canonical decomposition, and has fixed points along lines which decompose the pairs of pants into hexagons.
The surface $\mathscr{S}_{0,\infty}$ can be disconnected along those lines into two planar surfaces with boundary, one of which is called the {\em visible side} of 
$\mathscr{S}_{0,\infty}$, while the other is the {\em hidden side} of $\mathscr{S}_{0,\infty}$. The involution $j$ maps the visible side of $\mathscr{S}_{0,\infty}$ 
onto the hidden side, and vice versa.

\vspace{0.1cm}\noindent 
From now on, we assume that such an involution $j$ is chosen, hence a decomposition of the surface into a ``visible" and a ``hidden'' side.

\begin{definition}

  The data consisting of the canonical pants  decomposition of $\mathscr{S}_{0,\infty}$ together with the above decomposition into a visible and a hidden side is called
the {\em canonical rigid structure} \index{canonical rigid structure} of  $\mathscr{S}_{0,\infty}$.

\end{definition}

\vspace{0.1cm}\noindent  
The tree ${\cal T}_2$ may be embedded into the visible side of $\mathscr{S}_{0,\infty}$, as the dual 
tree to the pants decomposition.
This set of data is represented in Figure \ref{surfinfcyl}.

\begin{figure}

\begin{center}
\includegraphics{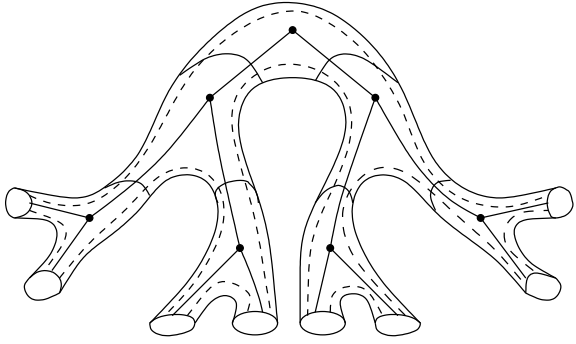}
\caption{Surface $\mathscr{S}_{0,\infty}$ with its canonical rigid structure}\label{surfinfcyl} 
\end{center}

\end{figure}

\vspace{0.1cm}\noindent 
The surface $\mathscr{S}_{0,\infty}$ appears already in \cite{fu-ka1}, 
endowed with a pants decomposition (with no cylinder), dual to the regular unrooted dyadic tree.

\vspace{0.1cm}\noindent 
In \cite{ka-se}, the notion of {\it asymptotically rigid homeomorphism} is defined. 
It plays a key role in  \cite{fu-ka1}, \cite{fu-ka2} and \cite{fu-ka3}.

\vspace{0.1cm}\noindent 
Let us introduce some more terminology.  
Any connected and compact subsurface of $\mathscr{S}_{0,\infty}$ 
which is the union of the cylinder and finitely many pairs of pants of the canonical decomposition will be called an {\it  admissible subsurface}
\index{admissible subsurface} of $\mathscr{S}_{0,\infty}$.
The {\it type} of such a subsurface $S$ is the number of connected components in its boundary.
The {\it tree of $S$} is the trace of  ${\cal T}_2$ on $S$. Clearly, the type of $S$ is equal 
to the number of leaves of its tree.

\begin{definition}[following  \cite{ka-se} and \cite{fu-ka1}]\label{asy}
A homeomorphism $\varphi$ of $\mathscr{S}_{0,\infty}$ is {\em asymptotically rigid} \index{asymptotically rigid homeomorphism} if there exist two 
admissible  subsurfaces $S_0$ and $S_1$ having the same type, such that 
$\varphi(S_{0})=S_{1}$ and whose restriction 
$\mathscr{S}_{0,\infty}\setminus{S_0}\rightarrow \mathscr{S}_{0,\infty}\setminus{S_1}$ is rigid, 
meaning that 
it maps each pants (of the canonical pants decomposition) onto a pants. 

\vspace{0.1cm}\noindent 
The {\em asymptotically rigid mapping class group} 
\index{asymptotically rigid mapping class group} of $\mathscr{S}_{0,\infty}$ 
is the group of isotopy classes of asymptotically rigid homeomorphisms. 
\end{definition}

\vspace{0.1cm}\noindent 
Though the proof of the following theorem is easy, this theorem is seminal as it is the starting point of a deeper study of the links between Thompson's groups 
and mapping class groups.

\begin{theorem}[\cite{ka-se}, Theorem 3.3]\label{mcg}
Thompson's group  $T$ can be embedded into the group of isotopy classes of orientation-preserving homeomorphisms of
 $\mathscr{S}_{0,\infty}$. An isotopy class belongs to the image of the embedding 
 if it may be represented by 
 an asymptotically rigid homeomorphism of
  $\mathscr{S}_{0,\infty}$ which globally preserves 
  the decomposition into visible/hidden sides.
  \end{theorem}

\vspace{0.1cm}\noindent 
We denote by $\mathscr{S}_{0,\infty}^+$ the visible side of $\mathscr{S}_{0,\infty}$.  This is a planar surface \index{planar surface} which inherits from the canonical decomposition of $\mathscr{S}_{0,\infty}$ a decomposition into hexagons (and one rectangle, corresponding to the visible side of the cylinder).
We could restate the above definitions by replacing pairs of pants by hexagons 
and the surface $\mathscr{S}_{0,\infty}$ by its visible side 
$\mathscr{S}_{0,\infty}^+$. 
Then Theorem \ref{mcg} states that $T$ can be embedded into the 
mapping class group of the {\it planar} surface $\mathscr{S}_{0,\infty}^+$.  
In fact $T$ is the {\em asymptotically rigid mapping class group} of 
$\mathscr{S}_{0,\infty}^+$, namely the group of mapping classes of 
those homeomorphisms of $\mathscr{S}_{0,\infty}^+$ 
which map  all but finitely many hexagons onto hexagons.

\subsection{Braid groups and Thompson groups}

A seminal result, which is the starting point of the article \cite{ka-se}, is a theorem of P. Greenberg and the third author (\cite{gr-se}). It states that there exists 
an extension of the derived subgroup $F'$ of Thompson's group $F$ by the stable braid group  $B_{\infty}$ (i.e. the braid group on a countable set of strands) 
$$ 1\rightarrow B_{\infty}\longrightarrow A\longrightarrow F'\rightarrow 1\;\;\; (Gr-Se),$$
where the group  $A$ is acyclic,\index{acyclic group} i.e. its integral homology vanishes. 
The existence of such a relation 
between Thompson's group and the braid group was conjectured by comparing their homology types.
On the one hand, it is proved in \cite{gh-se} that $F'$ has the homology of $\Omega S^3$, the space of based loops \index{loop space of $S^3$}
on the three-dimensional sphere. More precisely, the +-construction of the classifying space $BF'$ is homotopically 
equivalent to $\Omega S^3$. On the other hand,  F. Cohen proved that $B_{\infty}$ has the homology 
of  $\Omega^2 S^3$, the double loop space of $S^3$.
It turns out that both spaces ($\Omega S^3$ and $\Omega^2 S^3$)  are related by the path fibration 

$$\Omega^2 S^3\hookrightarrow  P(\Omega S^3)\rightarrow  \Omega S^3,$$   
where $P(\Omega S^3)$ denotes the space of based paths on $\Omega S^3$. The total space of this fibration,  $P(\Omega S^3)$, is contractible.  
Therefore, the existence of this natural fibration has led the authors of \cite{gr-se} to conjecture the existence of the short exact sequence
 $(Gr-Se)$.

\vspace{0.1cm}\noindent
The construction of $A$ amounts to giving 
a morphism $F'\to {\rm Out}(B_{\infty})$. 
In \cite{gr-se} one is lead to consider an extended binary tree and the braid group relative to its vertices. The group $F'$ 
acts on that tree by partial automorphisms and therefore induces the 
desired morphism.

\vspace{0.1cm}\noindent
Let us give a hint on how the acyclicity of $A$ is proved in \cite{gr-se}. 
Via direct computations, one shows that $H_1(A)=0$. One then proves that 
the fibration 
\[ {BB_{\infty}}_+ \to BA_+ \to BF'_+\]
can be delooped to a fibration 
\[ \Omega S^3 \to E \to S^3\]
Using the fact that $A$ is perfect one concludes that the space $E$ is contractible and so $A$ is acyclic.

\vspace{0.1cm}\noindent 
As a matter of fact, it is also proved in \cite{gr-se} that the short exact sequence $(Gr-Se)$ extends to the Thompson group $T$.
Indeed, there exists a short exact sequence  $$ 1\rightarrow B_{\infty}\longrightarrow A_T\longrightarrow T\rightarrow 1$$ whose pull-back via the embedding $F'\hookrightarrow T$ is $(Gr-Se)$. At the homology level, it corresponds 
to a fibration 
\[ \Omega^2 S^3 \to S^3\times \C P^{\infty} \to {\mathcal L}S^3\]
where ${\mathcal L}S^3$ denotes the free non-parametrized loop space of $S^3$.
This fact should not to be considered as anecdotical for the following reason. Let us divide
the groups $B_{\infty}$ and $A_T$ by the derived subgroup of  $B_{\infty}$. One obtains
a central extension of $T$ by $\Z=H_1(B_{\infty})$, which may be identified in the second cohomology group 
 $H^2(T,\Z)$ to the {\it discrete Godbillon-Vey class} of Thompson's group $T$.

\vspace{0.1cm}\noindent 
Let us emphasize that a simpler version of $A_T$, namely the 
braided Ptolemy-Thompson group $T^*$, will be presented later. Retrospectively, 
$A_T$ could be called then the marked braided Ptolemy-Thompson group. 

\vspace{0.1cm}\noindent 
One of the motivations of \cite{ka-se} is to pursue the investigations 
about the analogies between the diffeomorphism group of the circle ${\rm Diff}(S^1)$ and Thompson's group $T$. 
A remarkable aspect of this analogy concerns the Bott-Virasoro-Godbillon-Vey class. The latter 
is a differentiable cohomology class of degree 2. Recall that the Lie algebra of the group ${\rm Diff}(S^1)$ 
is the algebra  ${\rm Vect}(S^1)$ of vector fields on the circle. 
There is a map  $H^*({\rm Diff}(S^1),\R)\rightarrow H^*({\rm Vect}(S^1),\R)$, where the right hand-side denotes the Gelfand-Fuchs cohomology of ${\rm Vect}(S^1)$, which  is simply induced by the differentiation of cocycles. The image of the Bott-Virasoro-Godbillon-Vey class \index{Bott-Virasoro-Godbillon-Vey class} is a generator of $H^2({\rm Vect}(S^1),\R)$ corresponding to the universal central extension 
of ${\rm Vect}(S^1)$, known by the physicists as the {\it Virasoro Algebra}.
\index{Virasoro algebra}

\vspace{0.1cm}\noindent
 Let us explain the analogies 
 between the cohomologies of $T$ and ${\rm Diff}(S^1)$. By the cohomology of $T$ we mean Eilenberg-McLane cohomology, while the cohomology under consideration on ${\rm Diff}(S^1)$ is the differentiable one (as very little is known about its Eilenberg-McLane cohomology).
The striking result is the following: 
the ring of cohomology of $T$ (with real coefficients) and the ring of differentiable cohomology of ${\rm Diff}(S^1)$ are isomorphic. Both are generated 
by two classes of degree 2: the Euler class (coming from the action on the circle), and 
the Bott-Virasoro-Godbillon-Vey class. 
In the cohomology ring of $T$, the Bott-Virasoro-Godbillon-Vey class is called the {\it discrete Godbillon-Vey class}.  \index{discrete Godbillon-Vey class}
The isomorphism between the two cohomology rings does not seem to be induced by 
known embeddings of $T$ into ${\rm Diff}(S^1)$ (such embeddings have been constructed in \cite{gh-se}).

\vspace{0.1cm}\noindent 
A fundamental aspect of the Godbillon-Vey class concerns its relations with 
the projective representations  of ${\rm Diff}(S^1)$, especially those 
which may be derived into highest weight 
modules of the Virasoro Algebra. Pressley and Segal (\cite{pr-se}) introduced some representations $\rho$ of
 ${\rm Diff}(S^1)$ in the {\it restricted} linear group ${\rm GL}_{res}$ of the Hilbert space $L^2(S^1)$. 
 Pulling back by $\rho$ a certain cohomology class  (which we refer to as 
the {\it Pressley-Segal} class) of 
 ${\rm GL}_{res}$, one obtains on  ${\rm Diff}(S^1)$ some multiples of the Godbillon-Vey class (cf. \cite{ka-se}, \S 4.1.3 
 for a precise statement). 

\subsection{Extending the Burau representation}

In  \cite{ka-se}, we show that an analogous scenario exists for the discrete Godbillon-Vey class $\bar{gv}$ of $T$. 
We first remark that the Pressley-Segal extension of ${\rm GL}_{res}$ is itself a pull-back of 
$$1\rightarrow \C^*\longrightarrow \frac{{\rm GL}(\mathfrak{H})}{\mathfrak{T}_1}\longrightarrow \frac{{\rm GL}(\mathfrak{H})}{\mathfrak{T}}\rightarrow 1$$
where ${\rm GL}(\mathfrak{H})$ denotes the group of bounded invertible operators of the Hilbert space 
$\mathfrak{H}$, $\mathfrak{T}$ the group of operators having a determinant,
and  $\mathfrak{T}_1$ the subgroup of operators having determinant 1.

\vspace{0.1cm}\noindent 
The first step is to reconstruct the group $A_T$, \index{group $A_T$} not in a combinatorial way as in \cite{gr-se}, but as a mapping class group of a surface  $\mathscr{S}^t_{0,\infty}$. 
The latter is obtained from $\mathscr{S}_{0,\infty}$, by gluing, on each 
pair of pants of its canonical decomposition, 
an infinite cylinder or ``tube'', marked with countably many punctures (cf. Figure \ref{pantsurf2}). 
The precise definition of the group $A_{T}$ being rather technical, 
we refer for that the reader to  \cite{ka-se}.

\begin{figure}
\begin{center}
\begin{picture}(0,0)%
\includegraphics{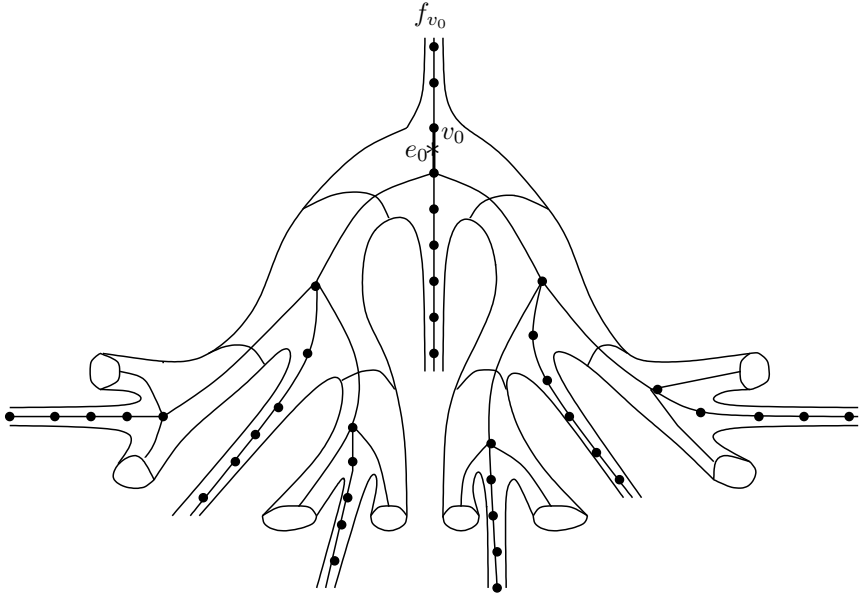}%
\end{picture}%
\setlength{\unitlength}{4972sp}%
\begingroup\makeatletter\ifx\SetFigFont\undefined%
\gdef\SetFigFont#1#2#3#4#5{%
  \reset@font\fontsize{#1}{#2pt}%
  \fontfamily{#3}\fontseries{#4}\fontshape{#5}%
  \selectfont}%
\fi\endgroup%
\begin{picture}(4270,2967)(2988,-2473)
\put(5176,-196){\makebox(0,0)[lb]{\smash{\SetFigFont{10}{12.0}{\rmdefault}{\mddefault}{\updefault}$v_0$}}}
\put(5041,389){\makebox(0,0)[lb]{\smash{\SetFigFont{10}{12.0}{\rmdefault}{\mddefault}{\updefault}$f_{v_0}$}}}
\put(4996,-286){\makebox(0,0)[lb]{\smash{\SetFigFont{10}{12.0}{\rmdefault}{\mddefault}{\updefault}$e_0$}}}
\put(5096,-286){\makebox(0,0)[lb]{\smash{\SetFigFont{12}{14.4}{\rmdefault}{\mddefault}{\updefault}$*$}}}
\end{picture}

\caption{Decomposition of ${\mathscr{S}}^t_{0,\infty}$ into  pants with tubes}\label{pantsurf2}
\end{center}
\end{figure}

\vspace{0.1cm}\noindent 
This new approach provides a setting that is convenient for an easy extension 
of the Burau representation 
of the braid group to $A_{T}$. We proceed as follows. The group  $A_T$ acts on the fundamental group 
of the punctured surface ${\mathscr{S}}^t_{0,\infty}$, which is a free group  of infinite countable rank. Moreover, 
the action is index-preserving, i.e. it induces the identity on $H_{1}(F_{\infty})$. Let  ${\rm Aut}^{ind}(F_{\infty})$ 
be the group of automorphisms of $F_{\infty}$ which are index-preserving.
The Magnus representation \index{Magnus representation} of ${\rm Aut}^{ind}(F_{n})$ extends to an infinite dimensional representation of 
${\rm Aut}^{ind}(F_{\infty})$ in the Hilbert space $\ell_2$ on the set of punctures 
of ${\mathscr{S}}^t_{0,\infty}$. Composing with the  map 
$A_{T}\rightarrow {\rm Aut}^{ind}(F_{\infty})$, one obtains 
a representation  $\rho_{\infty}^{\bf t}: A_T
 \rightarrow {\rm GL}(\ell^2)$ which extends the classical Burau representation of the braid group $B_{n}$. The scalar $t\in \C^*$ parameterizes a family of such representations.

\begin{theorem}[\cite{ka-se}, Theorem 4.7]
For each ${\bf t}\in\C^*$, the Burau representation \index{Burau representation} $\rho^{\bf t}_{\infty}:
B_{\infty}\rightarrow {\mathfrak T}$ extends to a
representation $\rho^{\bf t}_{\infty}$ of the mapping class group $A_T$
in the Hilbert space $\ell^2$ on the set of punctures  ${\mathscr{S}}^t_{0,\infty}$.  There exists a morphism 
of extensions \\

\setlength{\unitlength}{0.9cm}

\begin{picture}(10,2) 
\multiput(4,2)(1.5,0){2}{\vector(1,0){0.5}}   %les vecteurs horizontaux en haut
\put(3.5,1.9){1} \put(4.7,1.9){$B_{\infty}$} \put(6.3,1.9){$ A_T$}
\multiput(7,2)(2,0){2}{\vector(1,0){1}}   %les vecteurs horizontaux en haut

\put(8.4,1.9){$T$} 
\put(10.2,1.9){1}

\put(6.5,1.4){\vector(0,-1){.6}}  %le vecteur vertical 
\multiput(4,0.4)(1.2,0){2}{\vector(1,0){0.5}}   %les vecteurs horizontaux en bas
\put(3.5,0.3){1} \put(4.7,0.3){${\mathfrak T}$} 
\put(6,0.3){${\rm GL}({\ell^2})$}
\multiput(7.4,0.4)(2.2,0){2}{\vector(1,0){0.4}}   %les vecteurs horizontaux en bas
\put(8.1,0.3){$\frac{{\rm GL}({\ell^2})}{{\mathfrak T}}$}
\put(10.2,0.3){1}

%\put(8.5,3.4){\line(0,-1){.7}}\put(8.6,3.4){\line(0,-1){.7}}
\put(4.9,1.4){\vector(0,-1){.6}}

%\put(10.2,4){\vector(1,0){1}} \put(10.2,2.1){\vector(1,0){1}}
\put(8.5,1.4){\vector(0,-1){.6}}
\end{picture}

\noindent which induces a morphism of central extensions \\

\setlength{\unitlength}{0.9cm}

\begin{picture}(10,2) 
\multiput(3.8,2)(2,0){2}{\vector(1,0){0.4}}   %les vecteurs horizontaux en haut
\put(3.5,1.9){1} \put(4.3,1.9){$H_1(B_{\infty})$} \put(6.3,1.9){$\frac{A_T}{[B_{\infty},B_{\infty}]}$}
\multiput(7.8,2)(1.5,0){2}{\vector(1,0){0.9}}   %les vecteurs horizontaux en haut

\put(8.9,1.9){$T$}
\put(10.3,1.9){1}

\put(6.8,1.4){\vector(0,-1){.6}}  %le vecteur vertical 
\multiput(3.9,0.4)(1.4,0){2}{\vector(1,0){0.5}}   %les vecteurs horizontaux en bas
\put(3.5,0.3){1} \put(4.7,0.3){$\C^*$} 
\put(6.1,0.3){$\frac{{\rm GL}({\ell^2})}{{\mathfrak T}_1}$}
\multiput(7.4,0.4)(2.4,0){2}{\vector(1,0){0.4}}   %les vecteurs horizontaux en bas
\put(8.3,0.3){$\frac{{\rm GL}({\ell^2})}{{\mathfrak T}}$}
\put(10.3,0.3){1}

%\put(8.5,3.4){\line(0,-1){.7}}\put(8.6,3.4){\line(0,-1){.7}}
\put(4.9,1.4){\vector(0,-1){.6}}

%\put(10.2,4){\vector(1,0){1}} \put(10.2,2.1){\vector(1,0){1}}
\put(9,1.4){\vector(0,-1){.6}}
\end{picture}

\noindent 
The vertical arrows are injective if  ${\bf t}\in\C^*$ is not a root of unity.

\end{theorem}

\section{From the Ptolemy  groupoid to the Hatcher-Thurston complex}

\subsection{Universal Teichm\"uller theory according to Penner}

In \cite{pe0} (see also \cite{pe}), R. Penner introduced his version of a universal Teichm\"uller space, together with an associated universal group. 
Unexpectedly, this group happens to be 
isomorphic to the Thompson group $T$. This connection  between Thompson groups and Teichm\" uller 
theory plays a key role in \cite{fu-ka1}, \cite{fu-ka2} and \cite{fu-ka3}. It is therefore appropriate to 
give some insight into Penner's approach.

\vspace{0.1cm}\noindent 
The universal Teichm\"uller space \index{universal Teichm\"uller space} according to Penner is a set ${\cal T}ess$ of ideal tessellations of the Poincar\'e disk, modulo the action of ${\rm PSL}(2,\R)$ (cf. Definition \ref{mos} below). 
The space ${\cal T}ess$ is homogeneous under the action of the group ${\rm Homeo}^+(S^1)$
of orientation-preserving homeomorphisms of the circle:  
$${\cal T}ess={\rm Homeo}^+(S^1)/{\rm PSL}(2,\R).$$
Denoting by ${\rm Diff}^+(S^1)$ the diffeomorphism group of $S^1$ and ${\rm Homeo}_{qs}(S^1)$ 
the group of quasi-symmetric homeomorphisms of $S^1$
 (a quasi-symmetric homeomorphism of the circle is induced by 
 a quasi-conformal homeomorphism of 
 the disk) one has the following inclusions

$${\rm Diff}^+(S^1)/{\rm PSL}(2,\R) \hookrightarrow {\rm Homeo}_{qs}(S^1)/{\rm PSL}(2,\R)\hookrightarrow {\rm Homeo}^+(S^1)/{\rm PSL}(2,\R),$$

\noindent which justify that ${\cal T}ess$ is a generalization of 
the ``well known'' universal Teichm\"uller spaces, namely Bers' space 
${\rm Homeo}_{qs}(S^1)/{\rm PSL}(2,\R)$, and the physicists' space ${\rm Diff}^+(S^1)/{\rm PSL}(2,\R)$.

\vspace{0.1cm}\noindent 
\noindent Moreover, Penner introduced some coordinates on
 ${\cal T}ess$, as well as  a ``formal'' symplectic form, whose pull-back 
 on ${\rm Diff}^+(S^1)/{\rm PSL}(2,\R)$ is the Kostant-Kirillov-Souriau form.

\begin{definition}[following \cite{pe0} and \cite{pe}]\label{mos}
Let $\mathbb D$ be the Poincar\'e disk. A {\em tessellation} 
\index{tessellation} of $\mathbb D$ 
is a locally finite and countable set of complete geodesics on $\mathbb D$ 
whose endpoints 
lie on the boundary circle  $S^1_{\infty}=\partial \mathbb D$ and are called vertices. The geodesics are  
called {\em arcs} or {\em edges}, forming a triangulation of $\mathbb D$.
A {\em marked tessellation} of $\mathbb D$ is a pair made of a 
tessellation plus a distinguished oriented edge 
(abbreviated d.o.e.) $\vec{a}$. One denotes by  ${\cal T}ess'$ 
the set of marked tessellations.
\end{definition}

\vspace{0.2cm}\noindent 
Consider the  basic ideal triangle having vertices 
at $1,-1, \sqrt{-1}\in S^1_{\infty}$ in the 
unit disk model  $\mathbb D$. The orbits of its sides 
by the group ${\rm PSL}(2,\Z)$ is the so-called {\em Farey tessellation} 
\index{Farey tessellation}\index{tessellation! Farey} $\tau_0$, 
as drawn in Figure \ref{farey}. 
Its ideal vertices are 
the rational points of $\partial \mathbb D$.
The marked Farey tessellation has its distinguished oriented edge 
$\vec{a_0}$ joining 
-1 to 1.  
\\

\begin{figure}
\begin{center}
\includegraphics[scale=0.8]{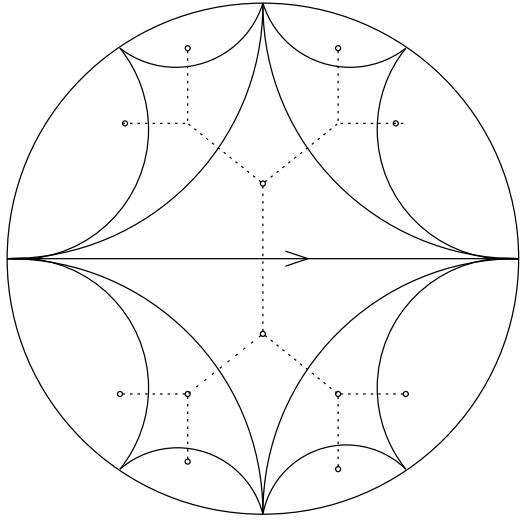} 
\caption{Farey tessellation and its dual tree}\label{farey}
\end{center}
\end{figure}

\vspace{0.1cm}\noindent 
The group ${\rm Homeo}^+(S^1)$ acts on the left on ${\cal T}ess'$ in the following way. Let $\gamma$ be an arc 
of a marked tessellation $\tau$, with endpoints $x$ and $y$, and $f$ be an element of ${\rm Homeo}^+(S^1)$; then 
 $f(\gamma)$ is defined as the geodesic with endpoints  $f(x)$ and $f(y)$. 
 If $\gamma$ is oriented from $x$ to $y$, then $f(\gamma)$ is oriented from $f(x)$ to $f(y)$. Finally, 
 $f(\tau)$ is the marked tessellation $\{f(\gamma),\gamma \in \tau\}$. 
 Viewing ${\rm PSL}(2,\R)$ as a subgroup of ${\rm Homeo}^+(S^1)$, one defines 
  ${\cal T}ess$ as the quotient space 
  ${\cal T}ess'/{\rm PSL}(2,\R)$.

\vspace{0.1cm}\noindent 
For any $\tau\in {\cal T}ess'$, let us denote by $\tau^0$ its set of ideal vertices. It is a countable and dense subset of the boundary circle, so that it may be proved
 that there exists 
a unique $f\in {\rm Homeo}^+(S^1)$ such that  $f(\tau_0)=\tau$. One denotes this homeomorphism 
 $f_{\tau}$. The resulting map

$${\cal T}ess'\longrightarrow  {\rm Homeo}^+(S^1),\;\;\, \tau\mapsto f_{\tau}$$

\noindent is a bijection. It follows that  ${\cal T}ess={\rm Homeo}^+(S^1)/{\rm PSL}(2,\R)$.

\vspace{0.2cm}
\noindent 
Since the action of ${\rm PSL}(2,\R)$ is 3-transitive, each element of ${\cal T}ess$ 
can be uniquely represented by its {\em normalized marked triangulation} 
containing the basic ideal triangle and whose d.o.e. is $\vec{a_0}$. 

\vspace{0.2cm}
\noindent 
The marked tessellation is of Farey-type if its canonical marked triangulation 
has the same vertices  and all but finitely many triangles (or sides) as the Farey triangulation. 
Unless explicitly stated otherwise all tessellations considered 
in the sequel will be Farey-type tessellations.  In particular, 
the ideal triangulations have the same vertices as $\tau_0$ and coincide 
with $\tau_0$ for all but finitely many ideal triangles.

\subsection{The isomorphism between Ptolemy and Thompson groups}\label{IKS}

\begin{definition}[Ptolemy groupoid]
The objects of the {\em (universal) Ptolemy groupoid}
\index{Ptolemy groupoid} $Pt$ are the 
marked tessellations of Farey-type. The morphisms 
are ordered pairs of marked tessellations modulo the common ${\rm PSL}(2,\R)$ 
action. 
\end{definition}

\vspace{0.2cm}
\noindent 
We now define  particular elements of $Pt$ called flips. \index{flip} 
Let $e$ be an edge of the marked tessellation represented by 
the normalized marked triangulation  $(\tau, \vec{a})$.  
The result of the flip $F_e$ on $\tau$ is the triangulation
$F_e(\tau)$ obtained  from $\tau$ by changing only 
the two neighboring triangles containing the edge $e$, according 
to the picture below: 

\vspace{0.2cm}
\begin{center}
\includegraphics{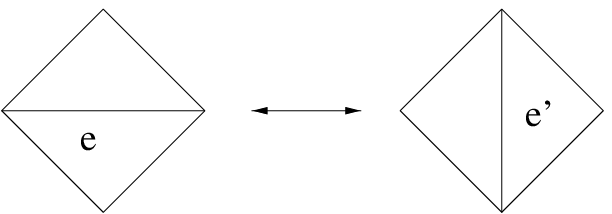}
\end{center}

\vspace{0.2cm}
\noindent 
This means that we remove $e$ from $\tau$ and then add the 
new edge $e'$ in order to get $F_e(\tau)$.  
In particular there is a natural correspondence $\phi:\tau\to F_e(\tau)$ 
sending $e$ to $e'$ and being the identity on all other edges. The result of a flip 
is the new triangulation together with this edge correspondence.

\vspace{0.2cm}
\noindent 
If $e$ is  not the d.o.e. of $\tau$ then $F_e(\vec{a})=\vec{a}$. 
If $e$ is the d.o.e. of $\tau$ then $F_e(\vec{a})=\vec{e'}$, where 
the orientation of $\vec{e'}$ is chosen so that 
the frame $(\vec{e}, \vec{e'})$ is positively oriented.

\vspace{0.2cm}
\noindent We define now the flipped tessellation $F_e((\tau, \vec{a}))$ 
to be the tessellation $(F_e(\tau), F_e(\vec{a}))$. 
It is proved in \cite{pe0} that flips generate the Ptolemy groupoid i.e. 
any element of $Pt$ is a composition of flips.

\vspace{0.2cm}
\noindent 
There is also a slightly different version of the Ptolemy groupoid 
which is quite useful in the case where we consider Teichm\"uller 
theory for surfaces of finite type. Specifically, we should assume that the 
tessellations are {\em labelled}, namely that their edges are indexed by 
natural numbers.

\begin{definition}[Labelled Ptolemy groupoid]
The objects of the {\em  labelled (universal) Ptolemy groupoid} $\widetilde{Pt}$ are the labelled  
marked tessellations. The morphisms  between  two objects 
$(\tau_1,\vec{a_1})$ and $(\tau_2,\vec{a_2})$ are {\em eventually trivial} 
permutation maps (at the labels level) $\phi:\tau_1\to \tau_2$ such that 
$\phi(\vec{a_1})=\vec{a_2}$. When marked tessellations are represented 
by their normalized tessellations, the latter coincide for all 
but finitely many triangles. Recall that $\phi$ is said to be 
eventually trivial 
if the induced correspondence at the level of the labelled tessellations
is the identity for all but finitely many edges. 
\end{definition}

\vspace{0.2cm}
\noindent 
Now flips make sense as elements of the labelled Ptolemy groupoid 
$\widetilde{Pt}$. Indeed the flip $F_e$ is endowed with 
the natural eventually trivial permutation  $\phi:\tau\to F_e(\tau)$ 
sending $e$ to $e'$ and being the identity for all other edges.

\vspace{0.2cm}
\noindent 
There is a standard procedure for converting a groupoid into a group, by 
using an a priori identification of all objects of the category. 
Here is how this goes  in the case of the Ptolemy groupoid. 
For any marked tessellation $(\tau, \vec{a})$ there is a 
characteristic map $Q_{\tau}:{\mathbb Q}-\{-1,1\}\to \tau$. 
Assume that $\tau$ is the canonical triangulation representing this 
tessellation. 
We first label by $\mathbb Q\cup{\infty}$ 
the vertices of $\tau$, by induction:
\begin{enumerate}
\item $-1$ is labelled by $0/1$, $1$ is labelled by $\infty=1/0$ and 
$\sqrt{-1}$ is labelled by $-1/1$. 
\item If we have a triangle in $\tau$ having two vertices already 
labelled by $a/b$ and $c/d$  then its third vertex is labelled $(a+c)/(b+d)$. 
Notice that vertices in the upper half-plane are labelled by negative 
rationals and those from the lower half-plane by positive rationals. 
\end{enumerate}
\vspace{0.2cm}
\noindent 
As it is well-known this labeling produces a bijection between 
the set of vertices of $\tau$ and $\mathbb Q\cup{\infty}$. 

\vspace{0.2cm}
\noindent 
Let now  $e$ be an edge of $\tau$, which is different from $\vec{a}$.
Let $v(e)$ be the vertex  opposite to $e$ of the triangle $\Delta$ of $\tau$ 
containing $e$ in its frontier and lying in the component of 
${\mathbb D}-e$ which does not contain 
$\vec{a}$. We then associate to $e$ the label of $v(e)$. 
We also give  $\vec{a}$ the label $0\in {\mathbb Q}$.  
In this way one obtains a bijection  
$Q_{\tau}:{\mathbb Q}-\{-1,1\}\to \tau$. 

\vspace{0.2cm}
\noindent 
Remark that if $(\tau_1,\vec{a_1})$ and $(\tau_2,\vec{a_2})$ are 
marked tessellations then there exists a unique map $f$ between their vertices 
sending triangles to triangles and marking on  marking. 
Then $f\circ Q_{\tau_1}=Q_{\tau_2}$. 

\vspace{0.2cm}
\noindent 
The role played by $Q_{\tau}$ is to allow flips to be indexed by 
the rationals and not by the edges of $\tau$. 

\begin{definition}[Ptolemy group \cite{pe0}]
Let ${\mathcal T}$ be the set of marked tessellations of Farey-type. 
Define the action of the free monoid  $M$ generated by ${\mathbb Q}-\{-1,1\}$
on ${\mathcal T}$ by means of:  
\[ q \cdot (\tau, \vec{a}) = F_{Q_{\tau}(q)}(\tau, \vec{a}),\,  \mbox{ for } \, 
q\in {\mathbb Q}-\{-1,1\}, (\tau, \vec{a})\in {\rm FT}.\]
We set  $f\sim f'$ on $M$  if the  two actions of $f$ and $f'$ on  
${\mathcal T}$ coincide. Then the induced composition law 
on  $M/\sim$ is a monoid structure for which each element has an inverse. This  
makes  $M/\sim$ a group, which is called the Ptolemy group $T$ \index{Ptolemy group} (see \cite{pe0} for more details).  
\end{definition}

\vspace{0.2cm}
\noindent 
In particular it makes sense to speak of flips in the present case. 
It is clear that flips generate the Ptolemy group. 

\vspace{0.2cm}
\noindent 
The notation $T$ for the Ptolemy group is not misleading because 
this group is isomorphic to the Thompson group $T$ and for this reason, 
we preferred to call it the Ptolemy-Thompson group. \index{Ptolemy-Thompson group $T$}

\vspace{0.2cm}
\noindent 
Given two marked tessellations $(\tau_1,\vec{a_1})$ and $(\tau_2,\vec{a_2})$ 
the above combinatorial isomorphism $f:\tau_1\to \tau_2$   
provides a map between the vertices of the tessellations, 
which are identified with $P^1(\mathbb Q)\subset S^1_{\infty}$. 
This map extends continuously to a homeomorphism of  $S^1_{\infty}$, 
which is piecewise-${\rm PSL}(2,\Z)$. 
This establishes an isomorphism between the Ptolemy group and 
the group of piecewise-${\rm PSL}(2,\Z)$ homeomorphisms of the circle. 

\vspace{0.2cm}
\noindent 
An explicit isomorphism with the group $T$ in the form introduced above 
was provided by Lochak and Schneps (see \cite{lo-sc}). 
In order to understand this isomorphism we will need another 
characterization of the Ptolemy groupoid, as follows.

\begin{definition}[Ptolemy groupoid second definition \cite{pe0,pe}]
The universal Ptolemy groupoid \index{Ptolemy groupoid}  $Pt'$ is the category 
whose objects are the marked tessellations. As for the morphisms, 
they are composed of morphisms of two types, called {\em elementary moves}:
\begin{enumerate} 

\item A-move: it is the data of a pair of marked tessellations $(\tau_1, \tau_2)$, where $\tau_1$ and
 $\tau_2$ only differ by the  d.o.e. The d.o.e. $\vec{a_1}$ of $\tau_1$ 
 is one of the two diagonals of a quadrilateral whose 4 sides belong to  $\tau_1$. Let us assume 
 that the vertices of this quadrilateral 
 are 
 enumerated in the cyclic direct order by $x,y,z,t$, in such a way that  $\vec{a_1}$ 
 is the edge oriented from $z$ to $x$.  Let $\vec{a_2}$ be the other diagonal, oriented 
 from $t$ to $y$. Then, 
$\tau_2$ is defined as the marked tessellation 
  $\tau_1\setminus\{\vec{a_1}\}\cup \{\vec{a_2}\}$, with oriented edge
$\vec{a_{2}}$.

\item B-move: it is the data of a pair of marked tessellations 
$(\tau_1, \tau_2)$, where $\tau_1$ and $\tau_2$ have the same edges, but only differ 
by the choice of the d.o.e. The marked edge $\vec{a_1}$  
is the side of the unique triangle of the tessellation  $\tau_1$ with ideal vertices $x,y,z$, enumerated in the direct order, in such a way that $\vec{a_1}$ 
is the edge from $x$ to $y$. 
Let $\vec{a_2}$ be the  edge   
oriented from  $y$ to $z$. Then, $\vec{a_2}$ is the d.o.e. of $\tau_2$.
\end{enumerate}

\noindent 
Relations between morphisms: if $\tau_1$ and $\tau_2$ are two marked 
tessellations such that there exist two sequences of elementary moves 
$(M_1,\ldots, M_k)$ and  $(M'_1,\ldots, M'_{k'})$  connecting  $\tau_1$ to $\tau_2$, then 
the morphisms $M_k \circ \ldots \circ M_1$ and $M'_{k'}\circ \ldots \circ M'_1$ are equal. 

\end{definition}

\begin{remark}
Given two marked tessellations $\tau_{1}$ and $\tau_{2}$ with the same sets of endpoints, 
there is a (non-unique) finite sequence of elementary moves connecting $\tau_{1}$ to $\tau_{2}$ if and only if $\tau_{1}$ and $\tau_{2}$ only differ by a finite number of edges.
\end{remark}

\noindent 
From the above remark, it follows that $Pt'$ is not a connected groupoid. 
Let $Pt=Pt'_{\Q}$ be the connected component of the Farey tessellation. 
It is the full sub-groupoid of $Pt'$ obtained by restricting to 
the tessellations whose set of 
ideal vertices are the rationals of the boundary circle $\partial \mathbb D$, 
and which differ from the Farey tessellation 
by only finitely many edges, namely 
the Farey-type tessellations. 
Then it is not difficult to prove that the two definitions of $Pt$ are actually equivalent. However the second definition makes the Lochak-Schneps isomorphism 
more transparent. \\
    
\noindent{\it Construction of the universal Ptolemy group}\\
    
\noindent    
Let $W$ be a symbol  $A$ or $B$. For any $\tau\in Ob(Pt')$, let us define the object $W(\tau)$, which is the target 
of the morphism of type $W$, whose source is $\tau$. For any sequence
 $W_1,\ldots, W_k$ of symbols $A$ or $B$, let us use the notation  $W_k\cdots W_2W_1 (\tau)$ for
  $W_k(...W_2( W_1(\tau))...)$.    
Let  $M$ be the free group on  $\{A,B\}$. Let us fix a tessellation  $\tau$ (the construction will not depend on this choice). Let $K$ be the subgroup of $M$ 
made of the elements $W_k\cdots W_2 W_1$ such that $W_k\cdots W_2 W_1(\tau)=\tau$ 
(it can be easily checked that this implies $W_k\cdots W_2 W_1(\tau')=\tau'$ for any $\tau\in Ob(Pt')$, and that $K$ is a normal subgroup of $M$).

\begin{definition}[\cite{pe0}, \cite{pe}] 
The group  $G= M/K$ is called the {\em universal Ptolemy group}. \index{Ptolemy group}
\end{definition}

\begin{theorem}[Imbert-Kontsevich-Sergiescu, \cite{im}]
The universal Ptolemy group  $G$ is anti-isomorphic to the  
Thompson group  $T$, 
which will be henceforth also called the Ptolemy-Thompson group in order 
to emphasize this double origin.  
\end{theorem}

\noindent 
Let us indicate a proof that relies on the definition of $T$ 
as a group of bijections of the boundary of the dyadic tree.  
Let  $\tau\in Pt$, and let  ${\rm T}_{\tau}$ be the regular 
(unrooted) dyadic tree which is dual to the tessellation $\tau$. 

\vspace{0.1cm}
\noindent 
Let $e_{\tau}$ be the edge of ${\rm T}_{\tau}$ which is transverse 
to the oriented edge 
 $\vec{a}_{\tau}$ of $\tau$. The edge $e_{\tau}$ is oriented in such a way that  
 $(\vec{a}_{\tau},\vec{e}_{\tau})$ is directly oriented in the disk. 
For each pair ($\tau$, $\tau'$) of marked tessellations of $Pt$, let
 $\varphi_{\tau,\tau'}\in Isom({\rm T}_{\tau},{\rm T}_{\tau'})$
be the unique isomorphism of planar oriented trees which maps the oriented edge  $\vec{e}_{\tau}$ onto 
the oriented edge $\vec{e}_{\tau'}$.  
As a matter of fact, the planar trees ${\rm T}_{\tau}$ and ${\rm T}_{\tau'}$  coincide outside two 
finite subtrees $t_{\tau}$ and $t_{\tau'}$ respectively, so that their boundaries  
$\partial {\rm T}_{\tau}$ and $\partial {\rm T}_{\tau'}$ may be canonically identified. Therefore,
$\varphi_{\tau,\tau'}$ induces a homeomorphism of $\partial {\rm T}_{\tau^*}$, 
denoted  $\partial \varphi_{\tau,\tau'}$. Clearly, $\partial \varphi_{\tau,\tau'}$ belongs to $T$, as it is induced on the boundary of the dyadic planar tree by a partial 
isomorphism which respects the local orientation of the edges.

\vspace{0.1cm}\noindent 
The map $g\in G\mapsto \partial\varphi_{\tau_{*}, g(\tau_{*})}\in {\rm Homeo} ( \partial {\rm T}_{\tau^*} )$ has $T$ as image, 
and is an anti-isomorphism onto  $T$.

\vspace{0.1cm}\noindent 
An explanation for the anti-isomorphy is the following. One has 
$\varphi_{\tau_{*}, gh(\tau_{*})}= \varphi_{h(\tau_{*}), g(h(\tau_{*}))}\varphi_{\tau_{*}, h(\tau_{*})}$. Now $\varphi_{h(\tau_{*}), g(h(\tau_{*}))}$ 
is the conjugate of 
$\varphi_{\tau_{*}, g(\tau_{*})}$ by $\varphi_{\tau_{*}, h(\tau_{*})}$, 
hence $\varphi_{\tau_{*}, gh(\tau_{*})}=
 \varphi_{\tau_{*}, h(\tau_{*})} \varphi_{\tau_{*}, g(\tau_{*})}$.

 \vspace{0.1cm}\noindent 
Following \cite{im}, it is also possible to construct an anti-isomorphism between $G$ and $T$, when the latter is realized as a subgroup of
  ${\rm Homeo}^+(S^1)$, viewing the circle as the boundary of the Poincar\'e disk.

\vspace{0.1cm}\noindent 
 For each $g\in G$, there exists a unique $f\in {\rm Homeo}^+(S^1)$ such that $f(\tau_{*})=g(\tau_{*})$. 
It is denoted by $f_{g}$. This provides a map  $f: G\rightarrow {\rm Homeo}^+(S^1)$, 
$g\mapsto f_{g}$, which is an anti-isomorphism. Indeed, for all $h$ and $g$ in $G$, the effect of
 $h$ on $\tau=g(\tau_{*})$ is the same as the effect of the conjugate 
 $f_{g}\circ f_{h}\circ f_{g}^{-1}$, so that
$(hg)(\tau_{*})=f_{g}\circ f_{h}\circ f_{g}^{-1}(\tau)=(f_{g}\circ f_{h})(\tau_{*})$. 
The morphism is injective, since $f_{g}=id$ implies that $g(\tau_{*})=\tau_{*}$, hence $g=1$.

\vspace{0.1cm}\noindent 
It is worth mentioning that a new presentation of $T$ has been obtained in \cite{lo-sc}, derived from the anti-isomorphism of $G$ and $T$.
It uses only two generators $\alpha$ and $\beta$, defined as follows. Let $\alpha\in T$ 
be the element induced by  $\varphi_{\tau_0, A.\tau_0}$, and 
 $\beta\in T$ induced by $\varphi_{\tau_0, B.\tau_0}$. 

\begin{theorem}[\cite{lo-sc}]
The Ptolemy-Thompson group \index{Ptolemy-Thompson group presentation}
\index{presentation! Thompson group}
$T$ is generated by two elements $\alpha$ and $\beta$, with relations:
$${\alpha}^4=1, {\beta}^3=1, (\beta\alpha)^5=1,$$        
$$[\beta\alpha\beta,\alpha^2\beta\alpha\beta\alpha^2]=1, [\beta\alpha\beta,\alpha^2\beta^2\alpha^2\beta\alpha\beta\alpha^2\beta\alpha^2]=1.$$ 

\end{theorem}

\vspace{0.1cm}\noindent 
Let us make explicit the relation between the Cayley graph of $T$, for the above presentation, and the nerve of the category $Pt$.

\begin{definition}
Let $Gr(Pt)$ be the graph whose vertices are the objects of $Pt$, and 
whose edges correspond to the elementary moves of type $A$ and $B$.
\end{definition}

\vspace{0.1cm}\noindent 
From the anti-isomorphism between $G=M/K$ and $T$, it follows easily that $Gr(Pt)$ is 
precisely the Cayley graph of Thompson's group $T$, for its presentation on the generators $\alpha$ and $\beta$.

\vspace{0.1cm}\noindent 
We can use the same method to derive a labelled Ptolemy group $\widetilde{T}$ out of the labelled Ptolemy groupoid $\widetilde{Pt}$.  It is not 
difficult to obtain therefore the following: 

\begin{proposition}\label{permut}
We have an exact sequence 
\[ 1\to S_{\infty} \to \widetilde{T}\to T\to 1\]
where $S_{\infty}$ is the group of  eventually trivial 
permutations of the labels. Moreover, the group $\widetilde{T}$ is 
generated by the obvious lifts $\widetilde{\alpha}$ and $\widetilde{\beta}$ 
of the generators $\alpha,\beta$ of $T$. The pentagon relation now reads 
$(\widetilde{\beta}\widetilde{\alpha})^5=\sigma_{12}$, where 
$\sigma_{12}$ is the transposition exchanging the labels of the diagonals 
of the pentagon.
\end{proposition}

\begin{remark} 
 Let us mention that the image of $G$ in ${\rm Homeo}^+(S^1)$ by the anti-isomorphism
  $f:g\mapsto f_{g}$ does not correspond to the piecewise dyadic affine version of $T$, as recalled in the preliminaries. 
  Let us view here the circle $S^1$ as the real projective line, and not as the quotient space $[0,1]/0\sim 1$.
 Under this identification, $f(G)$ is the group $P{\rm PSL}(2,\Z)$ 
 of orientation preserving homeomorphisms of the projective line, 
 which are piecewise ${\rm PSL}(2,\Z)$, with rational breakpoints.  This version of $T$ 
 is the starting point of a detailed study of the piecewise projective geometry of Thompson's group $T$, 
led in \cite{ma} and \cite{ma1}.
\end{remark}

 \subsection{A remarkable link between the Ptolemy  groupoid and the 
 Hatcher-Thurston complex of  ${\mathscr{S}}_{0,\infty}$, following \cite{fu-ka1}}

In  \cite{fu-ka1}, we give a generalization of the 
Ptolemy  groupoid which uses pairs of pants decompositions of the surface  ${\mathscr{S}}_{0,\infty}$.

\vspace{0.1cm}\noindent 
The surface $\mathscr{S}_{0,\infty}$ appears in \cite{ka-se} with its ``canonical rigid structure'' (see also section \ref{thommcg}). The constructions involved in \cite{fu-ka1} require 
to handle not only the canonical rigid structure of $\mathscr{S}_{0,\infty}$, but also a set of rigid structures.

\begin{figure}
\begin{center}
\includegraphics{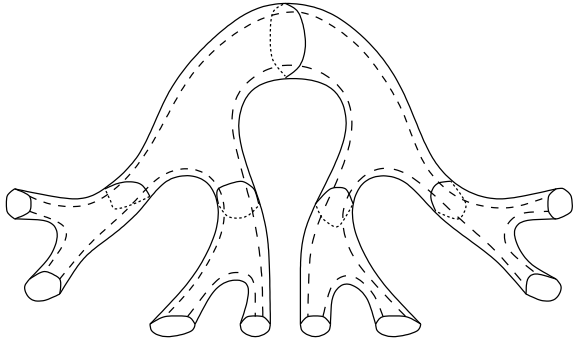}

\caption{Surface $\mathscr{S}_{0,\infty}$ with its canonical rigid structure}\label{surfinf0}
\end{center}
\end{figure}

\begin{definition}
A {\em rigid structure} \index{rigid structure} on $\mathscr{S}_{0,\infty}$ consists of the data of 
a pants decomposition of  $\mathscr{S}_{0,\infty}$ together with a 
decomposition of  $\mathscr{S}_{0,\infty}$ into two connected components, called the visible and the hidden side, which are compatible in the following sense. 
The intersection of each pair of pants with the visible or hidden sides of the surface is a hexagon.

\vspace{0.1cm}\noindent 
The choice of a reference rigid structure defines the {\em canonical rigid structure} \index{canonical rigid structure} (cf. Figure \ref{surfinf0}).  
The dyadic regular (unrooted) tree ${\cal T}_{*}$ is embedded onto the visible side of $\mathscr{S}_{0,\infty}$,
 as the dual tree to the canonical decomposition (into hexagons).

\vspace{0.1cm}\noindent 
A rigid structure is {\em marked} when one of the circles of the decomposition 
is endowed with an orientation. The choice of a circle of the canonical 
decomposition and of an orientation of this circle defines the 
canonical marked rigid structure.

\vspace{0.1cm}\noindent 
A rigid structure is {\em asymptotically trivial} \index{asymptotically trivial rigid structure} if it coincides 
with the canonical rigid structure outside a compact subsurface of $\mathscr{S}_{0,\infty}$.

\vspace{0.1cm}\noindent 
The set of isotopy classes of (resp. marked)
asymptotically trivial rigid structures is denoted $Rig(\mathscr{S}_{0,\infty})$ (resp.  $Rig'(\mathscr{S}_{0,\infty})$).
\end{definition}

In  \cite{fu-ka1}, we define the {\it stable groupoid of duality} ${\mathscr{D}}_0^s$, \index{stable groupoid of duality} which generalizes $Pt$, since it
contains a full sub-groupoid isomorphic to $Pt$. We first recall the definition of this sub-groupoid, 
which will be denoted ${{\mathscr{D}}_0^s}_{\Q}$. 

\begin{definition}
The objects of the groupoid ${{\mathscr{D}}_0^s}_{\Q}$ are 
the asymptotically rigid marked structures 
of $\mathscr{S}_{0,\infty}$
whose underlying decomposition into visible and hidden sides is the canonical one.

\vspace{0.1cm}\noindent 
The morphisms are composed of {\em elementary morphisms}, called {\em moves}, of two types, $A$ and $B$. 
\begin{enumerate}
\item A-move: Let $r_{1}$ be an object of ${{\mathscr{D}}_0^s}_{\Q}$. The distinguished oriented circle $\gamma$ 
separates two adjacent pairs of pants, whose union is a 4-holed sphere $\Sigma_{0,4}$. Up to isotopy, there exists a unique circle contained 
in $\Sigma_{0,4}$, whose geometric intersection number 
with $\gamma$ is equal to  2, and which is invariant by the involution $j$ interchanging the visible and hidden sides. 
Otherwise stated, the circle $\gamma'$ is the image of $\gamma$ by the rotation of angle
 $+\frac{\pi}{2}$ described in  Figure \ref{pto} which stabilizes both sides of $\mathscr{S}_{0,\infty}$ 
 and $\Sigma_{0,4}$.  
Let $r_{2}=r_{1}\setminus\{\gamma\}\cup \{\gamma'\}$. By definition, the pair $(r_{1},r_{2})$ is the $A$-move 
on the rigid marked structure $r_{1}$. Its source is $r_{1}$ while $r_{2}$ is its target.
\item $B$-move: Let $r_{1}$ be an object of ${{\mathscr{D}}_0^s}_{\Q}$. 
Let $P$ be the pair of pants of $r_{1}$  bounded by $\gamma$, which is on the left when one 
moves along $\gamma$ following its orientation. Let $\gamma''$ be the oriented circle of the boundary of $P$, which is 
the image of the oriented circle $\gamma$
by the rotation of order 3 and angle  $+\frac{2\pi}{3}$ described on  Figure \ref{pto} (it stabilizes both sides of
 $\mathscr{S}_{0,\infty}$ and $P$). Let  $r_{2}$ be the pants decomposition whose circles are the same as 
 those of $r_{1}$, but whose distinguished oriented circles is $\gamma''$. 
By definition, the pair $(r_{1},r_{2})$ is the $B$-move on $r_{1}$. Its source is $r_{1}$ while its target is $r_{2}$.

\end{enumerate}
 Relations among morphisms:  if $r_1$ and  $r_2$ are two objects of ${{\mathscr{D}}_0^s}_{\Q}$ such that there exist two 
 sequences of moves $(M_1,\ldots, M_k)$ and $(M'_1,\ldots, M'_{k'})$ transforming $r_1$ into $r_2$, 
 then  $M_k \circ \ldots \circ M_1= M'_{k'}\circ \ldots \circ M'_1$. 

\end{definition}

\begin{figure}        
\begin{center} 
\includegraphics[scale=0.8]{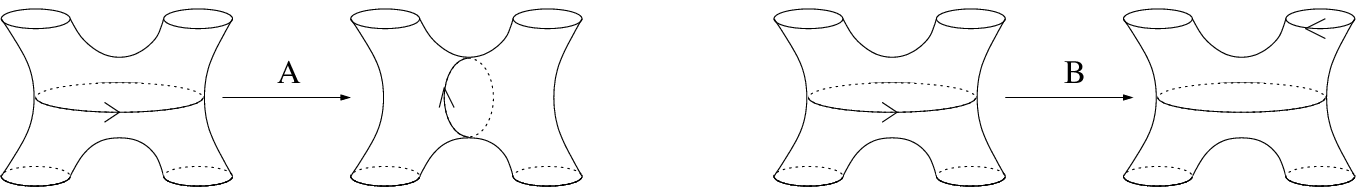}  
\end{center}
\caption{Moves in the groupoid ${\mathscr{D}}_{0,\Q}^s$ }\label{pto}          
\end{figure} 

\begin{remark}
There is a bijection between the set of objects of $Pt$ and  the set of objects of ${\mathscr{D}}_{0,\Q}^s$, 
which maps the {\it marked Farey tessellation}
onto the {\it canonical marked rigid structure} of $\mathscr{S}_{0,\infty}$. 
This bijection extends to a groupoid isomorphism $Pt\rightarrow {{\mathscr{D}}_0^s}_{\Q}$.
\end{remark}

\vspace{0.1cm}\noindent 
Via this isomorphism, the generators 
$\alpha$ and $\beta$ may be viewed as isotopy classes 
of asymptotically rigid homeomorphisms 
(which preserve the visible/hidden sides decomposition) of $\mathscr{S}_{0,\infty}$. 
The generator $\alpha$ corresponds to the mapping class such that $\alpha(r_*)=A(r_*)$, and $\beta$ 
to the mapping class such that  $\beta(r_*)=B(r_*)$. This gives a new proof of the 
existence of an embedding of $T$ into the mapping class group of 
$\mathscr{S}_{0,\infty}$, obtained in \cite{ka-se}.

\subsection{The Hatcher-Thurston complex of $\mathscr{S}_{0,\infty}$}

The Hatcher-Thurston complex of pants decompositions is first mentioned 
in the appendix of \cite{ha-th}. It is defined again in \cite{ha-lo-sc}, for any compact 
oriented surface, possibly with boundary, where it is proved that it is simply connected. We extend its definition to the non-compact 
surface $\mathscr{S}_{0,\infty}$.

\begin{definition}[\cite{fu-ka1}]
The Hatcher-Thurston complex \index{Hatcher-Thurston complex} ${\cal HT}(\mathscr{S}_{0,\infty})$ is a cell  2-complex.
\begin{enumerate}
\item Its vertices are the asymptotically trivial 
pants decompositions of $\mathscr{S}_{0,\infty}$.
\item Its edges correspond to pairs of decompositions $(p,p')$ such that  $p'$ is obtained from 
 $p$ by a local $A$-move, i.e. by replacing a circle $\gamma$ of $p$ by any circle $\gamma'$ whose geometric intersection number with $\gamma$ is equal to 2 (and does not intersect the other circles of $p$).
\item Its 2-cells fill in the cycles of moves 
of the following types: 
triangular cycles, pentagonal cycles (cf. Figure \ref{tr-pe}), and square cycles corresponding to the commutation 
of two $A$-moves with disjoint supports. 
\end{enumerate} 
\end{definition}    

\vspace{0.1cm}\noindent 
The Hatcher-Thurston complex ${\cal HT}(\mathscr{S}_{0,\infty})$ is an inductive limit of Hatcher-Thurston complexes 
of compact subsurfaces of $\mathscr{S}_{0,\infty}$. It is therefore simply connected.
 
\begin{figure}  
\begin{center}        
\includegraphics[scale=0.8]{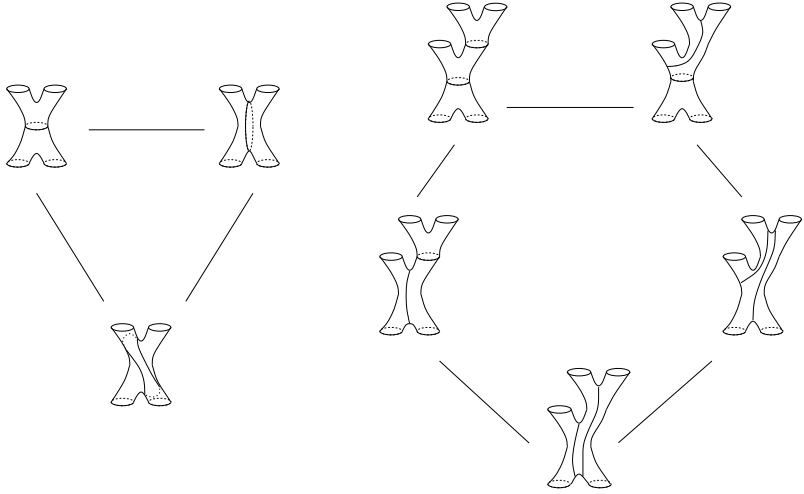}  
\caption{Triangular cycle and pentagonal cycle in ${\cal HT}(\mathscr{S}_{0,\infty})$}\label{tr-pe}\end{center}          
\end{figure}  

\vspace{0.1cm}\noindent 
The following proposition establishes a fundamental relation between 
the Cayley graph of Thompson's group $T$ (generated by $\alpha$ and $\beta$) and the 
Hatcher-Thurston complex of $\mathscr{S}_{0,\infty}$. The presentation of $T$ will be exploited to prove some 
useful properties of the Hatcher-Thurston complex.

\begin{proposition}[following \cite{fu-ka1}]
The forgetful map $Ob({\mathscr{D}}_{0,\Q}^s )\rightarrow {\cal HT}(\mathscr{S}_{0,\infty})$,
 which maps an asymptotically rigid marked structure onto 
 the underlying pants decomposition, extends to a cellular map 
 $\nu:Gr({\mathscr{D}}_{0,\Q}^s )\rightarrow {\cal HT}(\mathscr{S}_{0,\infty})$ 
 from the graph of the groupoid onto the 1-skeleton of the Hatcher-Thurston complex. 
 It maps an edge corresponding to an $A$-move onto an edge of type $A$ of ${\cal HT}(\mathscr{S}_{0,\infty})$, and collapses an edge corresponding to a $B$-move onto a vertex.
\end{proposition}

\vspace{0.1cm}\noindent 
Under the isomorphisms $Gr({\mathscr{D}}_{0,\Q}^s)\approx  Gr(Pt)\approx Cayl(T)$, where 
$Cayl(T)$ is the Cayley graph of $T$ with generators  $\alpha$ and $\beta$, 
$\nu$ may be identified with a morphism from $Cayl(T)$ to ${\cal HT}(\mathscr{S}_{0,\infty})$.

\vspace{0.1cm}\noindent 
One can easily check that:  
\begin{enumerate}
\item the image by $\nu$ of the cycle of 10 moves associated to the relation  
$(\alpha\beta)^5=1$ is a pentagonal cycle of the Hatcher-Thurston complex;
\item the image by $\nu$ of the cycle associated to the relation $[\beta\alpha\beta,\alpha^2\beta\alpha\beta\alpha^2]=1$ is a 
square cycle  $(DC1)$,
 corresponding to the commutation of two $A$-moves 
 supported by two adjacent 4-holed spheres; 
 \item the image by $\nu$ of the cycle associated to the relation  
\[ [\beta\alpha\beta,\alpha^2\beta^2\alpha^2\beta\alpha\beta\alpha^2\beta\alpha^2]=1 \]
is a square cycle $(DC_2)$, corresponding to the commutation of two $A$-moves 
 supported by two  4-holed spheres separated by a pair of pants.
 \end{enumerate}

\begin{definition}[Reduced Hatcher-Thurston complex]\index{reduced Hatcher-Thurston complex}\index{Hatcher-Thurston complex! reduced}
Let ${\cal HT}_{red}(\mathscr{S}_{0,\infty})$ be the subcomplex of
  ${\cal HT}(\mathscr{S}_{0,\infty})$, which differs from the latter by the set of square 2-cells: a square 2-cell of ${\cal HT}(\mathscr{S}_{0,\infty})$ belongs to 
  ${\cal HT}_{red}(\mathscr{S}_{0,\infty})$ if and only if it is of type $(DC_1)$ 
  (corresponding to the commutation of $A$-moves supported by two adjacent 4-holed spheres), or of type
   $(DC_2)$ (corresponding to the commutation of $A$-moves 
   supported by two 4-holed spheres 
   separated by a pair of pants).
\end{definition}

\begin{proposition}[\cite{fu-ka1}, Proposition 5.5]
The subcomplex ${\cal HT}_{red}(\mathscr{S}_{0,\infty})$ is simply connected.
\end{proposition}

\vspace{0.1cm}\noindent 
We refer to \cite{fu-ka1} for the proof. It is based on the existence of the morphism of complexes $\nu$, and consists in proving, using the presentation of 
Thompson's group $T$, that any square cycle of ${\cal HT}(\mathscr{S}_{0,\infty})$ may be expressed as a product of conjugates of 
at most three types of cycles: the squares of types $(DC1)$ and $(DC2)$, and the pentagonal cycles.  

\section{The universal mapping class group in genus zero}

\subsection{Definition of the group $\B$} 

We have seen that $T$ is isomorphic to 
the group of mapping classes of asymptotically rigid homeomorphisms of $\mathscr{S}_{0,\infty}$ which globally 
preserve the decomposition of the surface into visible/hidden sides. It turns out that if one forgets the last condition, one obtains an interesting larger 
group, which is the main object of the article \cite{fu-ka1}.

\begin{definition}[\cite{fu-ka1}]
The universal mapping class group in genus zero \index{universal mapping vclass group in genus zero}
${\cal B}$ is the group of isotopy classes 
of (orientation-preserving) homeomorphisms 
of $\mathscr{S}_{0,\infty}$ which are asymptotically rigid, namely the 
asymptotically rigid mapping class group \index{asymptotically rigid mapping class group}
of $\mathscr{S}_{0,\infty}$ 
(see also Definition \ref{asy}).
\end{definition}

\vspace{0.1cm}\noindent 
From what precedes, $T$ imbeds into  $\B$. As a matter of fact, $\B$ is an extension of Thompson's group $V$.
 
\begin{proposition} [\cite{fu-ka1}, Proposition 2.4]
Let $K_{\infty}$ be the pure mapping class group of the surface $\mathscr{S}_{0,\infty}$, i.e. the group of mapping classes of  homeomorphisms 
which are compactly supported in $\mathscr{S}_{0,\infty}$. There exists a short exact sequence 
$$1\rightarrow K^*_{\infty}\longrightarrow \B \longrightarrow V\rightarrow 1$$
Moreover, the extension splits over  $T\subset V$.
 \end{proposition}

\begin{proof} For the comfort of the reader, we recall the proof given in \cite{fu-ka1}.  Let us define the projection $\B \to V$.          
Consider $\varphi\in \B$ and let $\Sigma$ be a support for         
$\varphi$. We introduce the         
  symbol         
  $(T_{\varphi(\Sigma)},T_{\Sigma},\sigma(\varphi))$, where  
$T_{\Sigma}$ (resp. $T_{\varphi(\Sigma)}$) denotes the minimal finite binary subtree of ${\cal T}$ which contains ${q}(\Sigma)$ (resp. ${q}(\varphi(\Sigma))$), and $\sigma(\varphi)$ is the bijection induced by $\varphi$ between the set of leaves of both trees. The image of $\varphi$ in $V$ is the          
class of this triple, and it is easy to check that this correspondence induces         
a well-defined and surjective morphism $\B\rightarrow V$. The kernel is the subgroup of         
isotopy classes of homeomorphisms inducing the identity outside  some  
compact set, 
and hence is the direct limit of the pure mapping class groups. 
 
\vspace{0.1cm}  
\noindent        
Denote by ${\bf T}$ the subgroup of $\B$ consisting of mapping classes         
represented by asymptotically rigid homeomorphisms preserving the whole visible         
side  of $\s$. The image of ${\bf T}$ in $V$ is the subgroup of elements         
represented by symbols $(T_1,T_0,\sigma)$, where $\sigma$ is a bijection preserving the cyclic order of the labeling of the leaves of the 
trees. Thus, the image of ${\bf T}$ is Ptolemy-Thompson's group  
$T\subset V$. Finally, the kernel of the epimorphism ${\bf T}\to T$ is         
trivial. In the following, we shall identify $T$ with ${\bf T}$.      
\end{proof}

\vspace{0.1cm}\noindent 
As the kernel of this extension is not finitely generated, there is no evidence that ${\cal B}$ should be finitely generated. 
The main theorem of \cite{fu-ka1} asserts a stronger result.

\subsection{$\B$ is finitely presented}

\begin{theorem}[\cite{fu-ka1}, Th. 3.1]
The group $\B$ is finitely presented.
\end{theorem}

\vspace{0.1cm}\noindent 
The proof is geometric, and inspired by the method of  Hatcher and Thurston for the presentation 
of mapping class groups of compact surfaces. It relies on the Bass-Serre theory, as generalized by K. Brown in 
\cite{br1}, which asserts the following. Let a group $G$ act on a simply connected 2-dimensional complex $X$, whose stabilizers of vertices are finitely presented, 
and whose stabilizers of edges are finitely generated. If the set of $G$-orbits of cells is finite (otherwise stated, the action is cocompact), then $G$ is finitely presented. 

\vspace{0.1cm}\noindent 
Clearly, the group ${\cal B}$ acts cellularly on the Hatcher-Thurston complex of  $\mathscr{S}_{0,\infty}$. However, the idea consisting in 
exploiting this action must be considerably improved if one wishes to prove the above theorem. Indeed, the complex ${\cal HT}(\mathscr{S}_{0,\infty})$ is simply connected, but it has infinitely many orbits of $\B$-cells. 
This is due to the existence of the square cycles, corresponding to the commutation of $A$-moves 
on disjoint supports. Let $\sigma$ be a 2-cell filling in such a square cycle; the $A$-moves which commute are supported on two 4-holed spheres, separated by a certain number of pairs of pants $n_{\sigma}$. Clearly, this integer is an invariant of the $\B$-orbit of $\sigma$, which can be arbitrarily large.\\
The interest for the reduced Hatcher-Thurston  ${\cal HT}_{red}(\mathscr{S}_{0,\infty})$ appears now clearly: it is both simply connected and finite modulo $\B$. 
Unfortunately, the stabilizers of the vertices or edges of ${\cal HT}_{red}(\mathscr{S}_{0,\infty})$ (which are the same as those of ${\cal HT}(\mathscr{S}_{0,\infty})$) 
under the action of ${\cal B}$ are not finitely generated. The idea, in order to overcome this difficulty, is to ``rigidify'' 
the pants decompositions so that the size of their stabilizers become more reasonable.
This leads us to introduce a complex ${\cal DP}(\mathscr{S}_{0,\infty})$, whose definition is rather technical (cf. \cite{fu-ka1}, \S5), 
which is a sort of mixing of the Hatcher-Thurston complex, and a certain $V$-complex, called the ``Brown-Stein complex'',  defined in \cite{br2}. 
The latter has been used in \cite{br2} to prove that $V$ has the $FP_{\infty}$ property.

\vspace{0.1cm}\noindent 
Therefore, our $\B$-complex  ${\cal DP}(\mathscr{S}_{0,\infty})$ encodes simultaneously some finiteness properties of 
the mapping class groups ${\cal M}(0,n)$ as well as of the Thompson group $V$.

\vspace{0.1cm}\noindent 
With the right complex in hand it is not difficult to find  
the explicit presentation for 
$\B$, by following the method described in \cite{br1}.

\section{The braided Ptolemy-Thompson group}

\subsection{Finite presentation}\label{ptdef}

In the continuity of our investigations  on the relations between 
Thompson groups and mapping class groups of surfaces  
we introduced and studied a group (in fact two groups which are quite similar) called the braided Ptolemy-Thompson group (\cite{fu-ka2}) $T^*$, which might appear 
as a simplified version of the group $A_{T}$ of \cite{gr-se}, and studied 
from a different point of view in \cite{ka-se}. Indeed, 
$T^*$, like $A_{T}$, is an extension of $T$ by the stable braid group $B_{\infty}$. Its definition is simpler than 
that of $A_{T}$, and is essentially topological.

\begin{definition}[from \cite{fu-ka2}]

\begin{enumerate}

\item Let $D$ be the planar surface with boundary obtained by thickening 
the dyadic complete (unrooted) planar tree. The decomposition into hexagons of $D$, which is dual to 
the tree, is called the {\em canonical decomposition}. \index{canonical decomposition}
By a {\em separating arc} of the decomposition 
we mean a connected component of the boundary of a hexagon which is not included in the boundary of $D$.
\item Let $D^{\sharp}$ be the surface $D$ with punctures corresponding to the vertices of the tree, and 
$D^*$ the surface $D$ whose punctures are the middles of 
the separating arcs of the canonical decomposition  (cf. Figure \ref{D}).
A connected subsurface of $D^{\sharp}$ or $D^*$ is {\em admissible} 
\index{admissibile subsurface}  
if it is the union of finitely many hexagons of the canonical decomposition.

\item Let $D^{\diamond}$ denote $D^{\sharp}$ or $D^*$. An orientation-preserving  homeomorphism $g$ of
 $D^{\diamond}$ is {\em asymptotically  rigid} \index{asymptotically rigid 
homeomorphism}
if it preserves globally the set of punctures, 
 and if there exist 
 two admissible subsurfaces $S_0$ and $S_1$ such that $g$ induces by restriction 
 a ``rigid'' homeomorphism from $D^{\diamond}\setminus S_0$ onto $D^{\diamond}\setminus S_1$, i.e. 
 a homeomorphism that respects the canonical decomposition and the punctures.
\end{enumerate}

\end{definition}

 \begin{figure} 
  \begin{center}   
\includegraphics[scale=0.8]{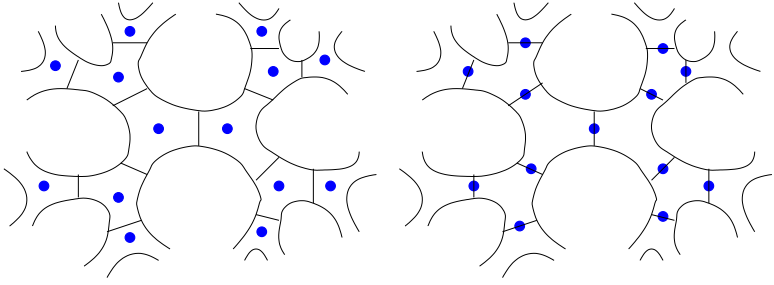}      
\caption{$D^{\sharp}$ and $D^{*}$ with their canonical rigid structures}\label{D}    
\end{center}  
\end{figure}

\vspace{0.1cm}\noindent 
Note that $D$ may be identified with the visible side of the surface $\mathscr{S}_{0,\infty}$ of  \cite{ka-se} and \cite{fu-ka1}. Its
canonical decomposition into hexagons is the trace on the visible side of $\mathscr{S}_{0,\infty}$ of the canonical pants decomposition of the latter.

\begin{definition}
The braided Ptolemy-Thompson group $T^{\diamond}$ \index{braided Ptolemy-Thompson group}
(where the symbol $\diamond$ may denote either $*$ or $\sharp$) is the 
group of isotopy classes of asymptotically rigid 
homeomorphisms of $D^{\diamond}$.
\end{definition}

\vspace{0.1cm}\noindent 
It is not difficult to see that there exists a short exact sequence 
$$1\rightarrow B_{\infty}\longrightarrow T^{\diamond}\longrightarrow T\rightarrow 1.$$ 
Unlike the extension of $T$ by $B_{\infty}$ which defines $A_{T}$, the above extensions, producing 
respectively the groups $T^{\sharp}$ and $T^*$, are not related to the discrete Godbillon-Vey class.

\vspace{0.1cm}\noindent 
The main result of \cite{fu-ka2} is a theorem concerning 
the group presentations.

 \begin{theorem}[\cite{fu-ka2}, Th. 4.5]
 The groups $T^{\sharp}$ and $T^*$ are finitely presented.
 \end{theorem}

\vspace{0.1cm}\noindent 
Moreover, an explicit presentation for $T^{\sharp}$ is given, with 3 generators. We show that $T^*$ is generated 
by  2 elements. By comparing their associated abelianized groups, one proves that 
$T^{\sharp}$ and $T^*$, though quite similar, are not isomorphic.

\vspace{0.1cm}\noindent 
As in \cite{fu-ka1}, we prove the above theorem by making  $T^{\sharp}$ and $T^*$ act on convenient simply-connected 2-complexes, 
The results of \S 4 are used once again, especially 
the reduced Hatcher-Thurston complex, by introducing braided 
versions of the Hatcher-Thurston complex of the surface $D^{\sharp}$ and $D^*$ (the pairs of pants being replaced 
by hexagons).

\vspace{0.1cm}\noindent 
In short, a vertex of these two complexes is a decomposition into hexagons which coincides with the canonical decomposition 
outside a compact subsurface $D$, such that:
\begin{enumerate} 
\item  in the $T^{\sharp}$-complex 
 each hexagon contains a puncture of $D^{\sharp}$ in its interior; 
\item in the $T^*$-complex  each separating arc 
passes through a puncture of $D^*$.
\end{enumerate}
 There are two types of edges: an $A$-move of Hatcher-Thurston, and a braiding move $B$ (cf. \cite{fu-ka2}, \S3). 
 Forgetting the punctures, one obtains fibrations from the complexes onto the Hatcher-Thurston complex of $D$, whose fibers 
 over the vertices are isomorphic to the Cayley complex of the stable braid group $B_{\infty}$. 
The presentation of $B_{\infty}$ that is convenient exploits the distribution of the punctures on a tree or a graph. It is given 
by a more general theorem of the third author (cf. \cite{se}).

\vspace{0.1cm}\noindent  
The groups $T^{\sharp}$ and $T^*$ share a number of properties 
which makes them  quite different from $\B$.  
For instance the cyclic orderability of $T$ together with the 
left orderability of $B_{\infty}$ leads to a  cyclic order on $T^*$. 
Using a result from \cite{ca} we obtain:

 \begin{proposition}[\cite{fu-ka2}, Prop. 2.13]
  The group $T^*$ can be embedded into the group of orientation-preserving homeomorphisms of the circle.
 \end{proposition}

\vspace{0.1cm}\noindent 
 Adapting one of the Artin  solutions of  
 the word problem in the braid group, we also prove 
 
 \begin{proposition}[\cite{fu-ka2}, Prop. 2.16]
 
 The word problem for the group  $T^*$ is solvable.
  \end{proposition}

 \vspace{0.1cm}\noindent 
The group  $T^*$ is also used in the study of an asymptotically rigid mapping class group of infinite genus, whose rational homology is isomorphic to the ``stable homology of the mapping class group''.

\subsection{Asynchronous combability}

\vspace{0.2cm}
\noindent 
The aim of this section is to show that $T^{\star}$ has strong finiteness 
properties. Although it was known  that one can generate 
Thompson groups  using automata (\cite{GN}), very little was known about the 
geometry of their Cayley graphs. Recently, D. Farley proved (\cite{Fa}) 
that Thompson groups (and more generally picture groups, see \cite{GuS}) 
act properly by isometries on CAT(0) 
cubical complexes  (and hence are a-T-menable), and V.Guba (see \cite{Gu,Gu2}) 
obtained that  
the smallest Thompson group $F$ has quadratic Dehn function while $T$ and $V$ 
have polynomial Dehn functions.  It is known that automatic groups have quadratic Dehn functions 
on one side and  Niblo and Reeves (\cite{NR}) proved that any group acting properly 
discontinuously and cocompactly on a CAT(0) cubical complex is automatic. 
One might  therefore wonder whether Thompson groups are automatic.

\vspace{0.2cm}
\noindent 
We approach this problem from the 
perspective of mapping class groups, since one can view $T$ and $T^*$ as 
mapping class groups of a surface of infinite type. One of the far reaching 
results in this respect is the Lee Mosher theorem  (\cite{Mo}) stating that 
mapping class  groups of surfaces of finite type are automatic. 
Our main result in 
\cite{fu-ka4} shows that, when shifting to surfaces of infinite type, a slightly 
weaker result still holds.

\vspace{0.2cm}
\noindent 
We will follow below the terminology introduced by Bridson in \cite{A,Br,Br2},
in particular we allow very general combings. We refer  the reader 
to \cite{ECHLPT} for a thorough introduction to the subject. 

\vspace{0.2cm}
\noindent 
Let $G$ be a finitely generated group with a finite generating set 
$S$, such that $S$ is closed with respect to taking inverses,  
and  $C(G,S)$ be the corresponding Cayley graph. This graph is endowed with 
the word metric in which the distance $d(g,g')$ between the vertices 
associated 
to the elements $g$ and $g'$ of $G$ is the minimal length of a word 
in the generators $S$  representing  the 
element $g^{-1}g'$ of $G$. 

\vspace{0.2cm}
\noindent 
%\begin{definition}
A {\em combing} \index{combing} of the group $G$ with generating set $S$ 
is a map which associates to any element $g\in G$ a 
path $\sigma_{g}$ in the Cayley graph associated to $S$
from $1$ to $g$. In other words $\sigma_{g}$ is a word in the 
free group generated by $S$ that represents  the 
element $g$ in $G$.  We can also represent $\sigma_g(t)$ as an infinite 
edge path in $C(G,S)$ (called  
combing path)  that joins the identity element to $g$, moving 
at each step to a neighboring vertex and which becomes eventually 
stationary at $g$. Denote by $|\sigma_g|$ the length of the path $\sigma_g$ 
i.e. the smallest $t$ for which $\sigma_g(t)$ becomes stationary.  
%\end{definition}

\begin{definition}
The combing $\sigma$ of the group $G$ is {\em synchronously bounded}
\index{synchronosuly bounded combing}
\index{combing! synchronosuly bounded} if it 
satisfies the synchronous fellow traveler property defined as follows. 
There exists $K$ such that the combing paths $\sigma_g$ and 
$\sigma_{g'}$ of  any two elements $g$, $g'$ at distance $d(g,g')=1$ 
are at most distance $K$ apart at each step i.e. 
\[ d(\sigma_g(t),\sigma_{g'}(t))\leq K, \; {\rm for \; any }\; t\in \R_+.\]
A group  $G$ having a synchronously bounded combing is  called synchronously
combable. 
\end{definition}

\vspace{0.2cm}
\noindent In particular, combings furnish normal forms for 
group elements. 
The existence of combings with special properties (like 
the fellow traveler property) has important consequences for the 
geometry of the group (see \cite{A,Br}).

\vspace{0.2cm}
\noindent 
We will  also introduce a slightly weaker condition (after 
Bridson and Gersten) as follows:

\begin{definition}
The combing $\sigma$ of the group $G$ is {\em asynchronously bounded} 
\index{asynchronously bounded combing}\index{combing! asynchronously bounded} 
if there exists $K$ such that  for any 
two elements $g$, $g'$ at distance $d(g,g')=1$ there exist ways to 
travel along   the combing paths $\sigma_g$ and 
$\sigma_{g'}$ at possibly different speeds so that corresponding  
points are at most distance $K$ apart.  Thus, there exists 
continuous increasing functions $\varphi(t)$ and $\varphi'(t)$ 
going from zero to infinity such that   
\[ d(\sigma_g(\varphi(t)),\sigma_{g'}(\varphi'(t)))\leq K, \; {\rm for \; any }\; t\in \R_+.\]
A group  $G$ having an asynchronously bounded combing is  called asynchronously
combable. 

\vspace{0.1cm}
\noindent
The asynchronously bounded combing $\sigma$ has {\em a departure function} 
\index{combing! departure function}
$D:\R_+\to \R_+$ if for all $r >0$, $g\in G$ and 
$0\leq s,t\leq |\sigma_g|$, the assumption  $|s-t| > D(r)$ implies that 
$d(\sigma_g(s),\sigma_g(t)) > r$. 
\end{definition}

\vspace{0.1cm}
\noindent
The main result of \cite{fu-ka4} can be stated as follows:

\begin{theorem}[\cite{fu-ka4}]
The group $T^{\star}$ is asynchronously combable. 
\end{theorem}

\noindent  In particular, in the course of the proof we also prove that:

\begin{corollary}
The Thompson group $T$ is asynchronously combable. 
\end{corollary}

\noindent The proof is largely inspired by the methods of L. Mosher. 
The mapping class group is embedded into the Ptolemy groupoid of some 
triangulation of the surface, as defined by L. Mosher and R. Penner.  
It suffices then to provide combings for the latter.

\begin{remark}
There are known examples of asynchronously combable groups with 
a departure function: 
asynchronously automatic groups (see \cite{ECHLPT}), 
the fundamental group of a Haken 3-manifold (\cite{Br}), 
or of a geometric 3-manifold (\cite{Br2}),
semi-direct products of $\Z^n$ by $\Z$ (\cite{Br}). 
Gersten (\cite{Ger}) proved that 
asynchronously combable groups with 
a departure function are of type ${\rm FP}_3$ and 
announced that they should actually be ${\rm FP}_{\infty}$. 
Recall that a group $G$ is ${\rm FP}_{n}$ if there is a projective 
$\Z[G]$-resolution of $\Z$ which is finitely generated in dimensions at most $n$ 
(see \cite{Geo}, Chapter 8 for a thorough discussion on this topic).  
Notice that there exist asynchronously combable groups  
(with departure function) which are 
not asynchronously automatic, for instance the {\sf Sol} and 
{\sf Nil} geometry groups  of closed 3-manifolds 
(see \cite{Bra});  in particular, they are not automatic. 
\end{remark}

\section{Central extensions of $T$ and quantization}

\subsection{Quantum universal Teichm\"uller space}
The goal of the quantization is, roughly speaking, 
to obtain {\em non-commutative deformations} of the action 
of the mapping class group on the Teichm\"uller space.  
It appears that the Teichm\"uller space of a surface 
has a particularly nice  {\em global} system of coordinate charts 
whenever the surface has at least one puncture, 
the so-called shearing coordinates introduced by Thurston 
(see \cite{Fun} for a survey). Each coordinate chart corresponds to 
fixing the isotopy class of a triangulation of the surface with 
vertices at the puncture. The mapping class group embeds into the 
{\em labelled} Ptolemy groupoid of the surface and there is a natural 
extension of the mapping class group action to 
an action of this groupoid on the set of coordinate charts. 
The necessity of considering labelled triangulations comes from the 
existence of triangulations with non-trivial automorphism groups. 
This theory extends naturally to the universal setting of  
Farey-type tessellations of the Poincar\'e disk ${\mathbb D}$, which behaves 
naturally as an infinitely punctured surface. Since there are no 
automorphisms of the binary tree which  induce eventually trivial 
permutations it follows that we do not need labelled tessellations.  
The analogue of the mapping class group is therefore  
the Ptolemy-Thompson group $T$. We will explain below (see Section \ref{dilog}) how one obtains by quantization a projective representation of 
$T$, namely a representation into the linear group modulo scalars, which is 
called the dilogarithmic representation.  
One of the main results of \cite{fu-se} (see also Sections \ref{comput} 
and \ref{ident}) is the fact that 
the dilogarithmic representation comes from a central extension 
of $T$ whose class is 12 times the Euler class generator. 
This result is very similar to the case of a finite type surface where 
the dilogarithmic representations come from a central extension of the 
mapping class group of a punctured surface having extension class 
12 times the Euler class plus the puncture classes (see \cite{fu-kas} 
for details).

\vspace{0.1cm}\noindent 
Here and henceforth, for the sake of brevity, 
we will use the term tessellation instead of marked tessellation. 
For each tessellation $\tau$ let $E(\tau)$ be the set of its edges. 
We associate further a skew-symmetric 
matrix  $\varepsilon(\tau)$ with entries $\varepsilon_{ef}$,  for all $e,f\in E(\tau)$, 
as follows. If $e$ and $f$ do not belong to the same triangle of $\tau$ or $e=f$  then 
$\varepsilon_{ef}=0$. Otherwise, $e$ and $f$  are distinct edges belonging 
to the same triangle of $\tau$ and thus have a common vertex. 
We obtain $f$ by rotating $e$  in the plane along that vertex 
such that the moving edge  is sweeping out the  respective triangle of $\tau$. 
If we rotate  clockwisely then $\varepsilon_{ef}=+1$ and otherwise  
$\varepsilon_{ef}=-1$. 

\vspace{0.2cm}
\noindent 
The pair $(E(\tau), \varepsilon(\tau))$ is called a {\em seed} in \cite{FG}. 
Observe, that in this particular case seeds are completely determined by tessellations.

 \vspace{0.2cm}
\noindent 
Let $(\tau,\tau')$ be a flip $F_e$ in the edge $e\in E(\tau)$. Then the associated seeds 
$(E(\tau), \varepsilon(\tau))$  and  $(E(\tau'), \varepsilon(\tau'))$  
are obtained one from another by a {\em mutation} in the direction $e$.  
Specifically, this means that   
there is an isomorphism $\mu_e:E(\tau)\to E(\tau')$ such that 
\[ \varepsilon(\tau')_{\mu_e(s)\mu_e(t)}=\left\{
\begin{array}{ll}
-\varepsilon_{st}, & \mbox{ \rm if } e=s \mbox{ \rm or } e=t, \\
\varepsilon_{st}, & \mbox{ \rm if } \varepsilon_{se}\varepsilon_{et} \leq 0, \\
\varepsilon_{st} + |\varepsilon_{se}|\varepsilon_{et},   
& \mbox{ \rm if } \varepsilon_{se}\varepsilon_{et} >0 
\end{array}
\right. 
\]
The map $\mu_e$ comes from the natural identification of the edges  of the two 
respective tessellations out of $e$ and $F_e(e)$. 

\vspace{0.2cm}
\noindent 
This algebraic setting appears in the description of the universal Teichm\"uller space 
${\mathcal T}$.  Its formal definition (see \cite{FG1,FG2})  is the  set of positive real 
points  of the cluster ${\mathcal X}$-space related to the set of seeds above. However, 
we can give a more intuitive description of it, following \cite{pe0}.  
Specifically,  this ${\mathcal T}$ is the space of {\em all} marked 
Farey-type tessellations from Section \ref{IKS}. 
Each tessellation $\tau$ gives rise to a coordinate system 
$\beta_{\tau}:{\mathcal T}\to \R^{E(\tau)}$.  The real number  
$x_e=\beta_{\tau}(e)\in \R$ specifies the amount of translation along the geodesic associated 
to the edge $e$  which is required when gluing together the two ideal triangles sharing that geodesic 
to obtain a given quadrilateral in the hyperbolic plane.  
These are called the shearing coordinates (introduced by Thurston and then 
considered by Bonahon, Fock and Penner) on the universal Teichm\"uller space and they provide a 
homeomorphism $\beta_{\tau}:{\mathcal T}\to \R^{E(\tau)}$. 
There is an explicit geometric formula (see also \cite{Fo,Fun}) for the shearing coordinates, as follows. 
Assume that the union of the two ideal triangles in ${\mathbb H}^2$ is the ideal quadrilateral 
of vertices $pp_0p_{-1}p_{\infty}$ and the common geodesic is $p_{\infty}p_0$. 
Then the respective shearing coordinate is the cross-ratio 
\[ x_e=[p,p_0,p_{-1},p_{\infty}]=\log\frac{(p_0-p)(p_{-1}-p_{\infty})}
{(p_{\infty}-p)(p_{-1}-p_0)}.\]

\vspace{0.2cm}
\noindent 
Let $\tau'$ be obtained from $\tau$ by a flip $F_e$ and set $\{x'_f\}$ for the 
coordinates associated to $\tau'$.   The map  
$\beta_{\tau,\tau'}:R^{E(\tau')}\to \R^{E(\tau)}$ given  by 
\[ \beta_{\tau,\tau'}(x'_s)=\left\{\begin{array}{ll}
x_s-\varepsilon(\tau)_{se}\log(1+\exp(-{\rm sgn}(\varepsilon_{se})x_e)), & \mbox{\rm if } s\neq e,\\ 
-x_e, &  \mbox{\rm if } s= e\\ 
\end{array}
\right. 
\]
relates the two coordinate systems, namely $\beta_{\tau,\tau'}\circ \beta_{\tau'}=\beta_{\tau}$.  

\vspace{0.2cm}
\noindent 
These coordinate systems provide a contravariant functor $\beta: Pt\to {\rm Comm}$ 
from the Ptolemy groupoid $Pt$ to the category ${\rm Comm}$ of commutative topological 
$*$-algebras over $\C$. We associate to a tessellation $\tau$ the 
algebra $B(\tau)=C^{\infty}(\R^{E(\tau)}, \C)$ of smooth  complex valued 
functions on $\R^{E(\tau)}$, with the $*$-structure given by $*f=\overline{f}$. 
Furthermore to any flip $(\tau,\tau')\in Pt$ one associates the map 
$\beta_{\tau,\tau'}:B(\tau')\to B(\tau)$. 

\vspace{0.2cm}
\noindent 
The matrices $\varepsilon(\tau)$ have a deep geometric meaning. 
In fact the  bi-vector  field 
\[ P_{\tau}=\sum_{e,f} \varepsilon(\tau)_{ef} \,\frac{\partial}{\partial x_e}\wedge 
\frac{\partial}{\partial x_f} \]
written here in the coordinates $\{x_e\}$ associated to $\tau$  
defines a Poisson structure on ${\mathcal T}$ which is invariant by the action 
of the Ptolemy groupoid. The associated Poisson bracket is then given by the formula 
\[ \{x_e,x_f\}=\varepsilon(\tau)_{ef}. \]

\vspace{0.2cm}
\noindent 
Kontsevich proved that  there is a canonical  formal quantization 
of a (finite dimensional) Poisson manifold. The universal Teichm\"uler space 
is not only a Poisson manifold but also endowed with a group action and 
our aim will be an equivariant quantization.  Chekhov, Fock  and Kashaev 
(see \cite{CF,CP,K,K2})  constructed an equivariant  quantization by means of 
explicit formulas. There are two ingredients in their approach. First, 
the Poisson bracket is given by constant coefficients, in any 
coordinate charts and  second,  the quantum (di)logarithm.

\vspace{0.2cm}
\noindent To any category $C$ whose morphisms are $\C$-vector spaces one 
associates its projectivisation $PC$ having the same objects and new 
morphisms given by ${\rm Hom}_{PC}(C_1,C_2)={\rm Hom}_C(C_1,C_2)/U(1)$, for any two objects 
$C_1,C_2$ of $C$. Here $U(1)\subset \C$ acts by scalar multiplication. 
A projective functor into $C$ is actually a functor into $PC$. 

\vspace{0.2cm}
\noindent 
Now let ${\rm A}^*$ be the category of topological $*$-algebras. 
Two functors $F_1, F_2:C\to {\rm A}^*$ essentially coincide if there exists a third functor $F$ and 
natural transformations $F_1\to F$, $F_2\to F$ providing dense inclusions 
$F_1(O)\hookrightarrow F(O)$ and $F_2(O)\hookrightarrow F(O)$, for any object $O$ of $C$.

\begin{definition}\label{quant} \index{quantization}
A {\em quantization}  ${\mathcal T}^h$ of the universal Teichm\"uller space 
is a family of contravariant projective functors 
$\beta^h:Pt\to {\rm A}^*$ depending smoothly on the real parameter $h$ such that: 
\begin{enumerate}
\item The limit $\lim_{h\to 0} \beta^h= \beta^0$ exists and essentially 
coincides with the functor 
$\beta$. 
\item The limit $\lim_{h\to 0}[f_1,f_2]/h$ is defined and coincides with 
the Poisson bracket on ${\mathcal T}$. Alternatively, for each $\tau$ we have a 
$\C(h)$-linear (non-commutative) product structure $\star$ 
on the vector space $C^{\infty}(\R^{E(\tau)},\C(h))$ such that 
\[ f\star g= fg + h \{f,g\} + {o}(h)\]
where $\{f,g\}$ is the Poisson bracket on functions on ${\mathcal T}$ and 
$\C(h)$  denotes the algebra of smooth $\C$-valued functions on the real parameter $h$. 
\end{enumerate}
\end{definition}

\vspace{0.2cm}
\noindent 
We associate to each tessellation $\tau$ the {\em Heisenberg algebra}
\index{Heisenberg algebra} $H_{\tau}^h$ which is 
the topological $*$-algebra over $\C$ generated by the elements $x_e, e\in E(\tau)$ 
subject to the relations
\[ [x_e,x_f]=2\pi i h \varepsilon(\tau)_{ef}, \;\,\; x_e^*=x_e. \]
We then define  $\beta^h(\tau)=H_{\tau}^h$. 

\vspace{0.2cm}
\noindent 
Quantization should associate a homomorphism 
$\beta^h((\tau,\tau')): H_{\tau'}^h\to H_{\tau}^h$ to each element 
$(\tau,\tau')\in Pt$. It actually suffices to consider the case where  
$(\tau,\tau')$ is the flip $F_e$ in the edge $e\in E(\tau)$. 
Let $\{x'_s\}$, $s\in E(\tau')$ be the generators of $H_{\tau'}^h$. 
We then set 
\[ \beta^h((\tau,\tau'))(x'_s)=\left\{\begin{array}{ll}
x_s -\varepsilon(\tau)_{se}\; \phi^h(-{\rm sgn}\left(\varepsilon(\tau)_{se})x_e\right), & \mbox{\rm if } 
s\neq e, \\
-x_s, & \mbox{\rm if } s= e \\
\end{array}
\right. 
\] 
Here $\phi^h$ is the {\em quantum logarithm} \index{quantum logarithm} function, namely 
\[ \phi^h(z) = -\frac{\pi h}{2}\int_{\Omega}
\frac{\exp(-it z)}{{\rm sh}(\pi t) \, {\rm sh }(\pi h t)}  dt \]
where the contour $\Omega$ goes along the real axes from $-\infty$ 
to $\infty$ bypassing the origin from above. 

\vspace{0.2cm}
\noindent Some properties of the quantum logarithm are collected below: 
\[ \lim_{h\to 0}\phi^h(z)=\log\left(1+\exp(z)\right), \;\;  \phi^h(z)-\phi^h(-z)=z, \]
\[ \overline{\phi^h(z)}=\phi^h\left(\overline{z}\right), \;\; 
 \frac{\phi^h(z)}{h}=\phi^{1/h}\left(\frac{z}{h}\right)\]

\vspace{0.2cm}
\noindent 
A convenient way to represent this transformation graphically is to associate to a 
tessellation its dual binary tree embedded in ${\mathbb H}^2$ and to assign to each edge 
$e$ the respective generator $x_e$. Then the action of a flip reads as follows: 

\vspace{0.2cm}
\begin{center}
\includegraphics{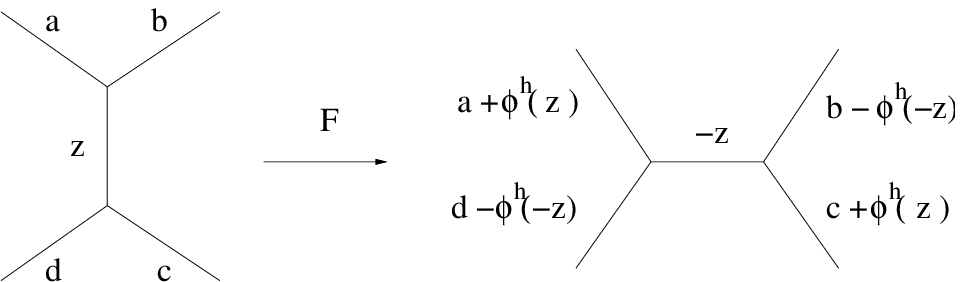}
\end{center}

\vspace{0.2cm}
\noindent We then have: 
\begin{proposition}[\cite{CF,FG3}]
The projective functor $\beta^h$ is well-defined and it is a quantization of the 
universal Teichm\"uller space $\mathcal T$. 
\end{proposition}

\vspace{0.2cm}
\noindent 
One proves that $\beta^h((\tau,\tau'))$ is independent of the decomposition of 
the element $(\tau,\tau')$ as a product of flips. 
In the  classical limit $h\to 0$ the quantum flip tends to the usual formula 
of the coordinate  change induced by a flip. 
Thus the first requirement in Definition \ref{quant} is fulfilled, and 
the second one  is obvious, from the defining relations in the Heisenberg algebra $H_{\tau}^h$.

\subsection{The dilogarithmic representation of $T$}\label{dilog}
The subject of this section is to give a somewhat 
self-contained definition of the 
dilogarithmic representation of the group $T$. The case of general  cluster modular groupoids 
is developed in full detail in \cite{FG,FG3} and the group $T$ as a cluster modular groupoid 
is explained in \cite{FG2}.  

\vspace{0.2cm}
\noindent 
The quantization of a physical system in quantum mechanics should provide a Hilbert space 
and the usual procedure is to consider a Hilbert space representation of the algebra 
from Definition \ref{quant}. This is formalized in the notion of 
representation of a quantum space. 

\vspace{0.2cm}
\noindent 
\begin{definition}\label{repre} 
A projective $*$-representation \index{projective $*$-representation}
of the quantized universal Teichm\"uller space ${\mathcal T}^h$,  
specified by the functor $\beta^h:Pt\to {\rm A}^*$,  
consists of the following data: 
\begin{enumerate}
\item A projective functor $Pt\to {\rm Hilb}$ to the category of Hilbert spaces. In particular, 
one associates a Hilbert space ${\mathcal L}_{\tau}$ to each tessellation $\tau$ and 
a unitary operator $K_{(\tau,\tau')}:{\mathcal L}_{\tau}\to {\mathcal L}_{\tau'}$, defined 
up to a scalar of absolute value 1.  
\item A $*$-representation $\rho_{\tau}$ of the Heisenberg algebra $H_{\tau}^h$ 
in the Hilbert space ${\mathcal L}_{\tau}$, such that the operators ${\mathbf K}_{(\tau,\tau')}$ 
intertwine the representations  $\rho_{\tau}$ and  $\rho_{\tau'}$ i.e.
\[   \rho_{\tau}(w)= {\mathbf K}_{(\tau,\tau')}^{-1} \rho_{\tau'}\left(\beta^h((\tau,\tau'))(w)\right)
{\mathbf K}_{(\tau,\tau')}, \,\,\, w\in H_{\tau}^h.\]
\end{enumerate}
\end{definition}

\vspace{0.2cm}
\noindent 
{\em The classical Heisenberg $*$-algebra} \index{classical Heisenberg algebra} $H$ is generated by $2n$ elements 
$x_s,y_s$, $1\leq s\leq n$ and relations 
\[ [x_s, y_s]=2\pi i \, h, \;\; [x_s,y_t]=0, \mbox{\rm if } s\neq t, \;\; [x_s, x_t]=[y_s,y_t]=0, \mbox{\rm for all } s, t\] 
with the obvious $*$-structure. 
The single irreducible integrable $*$-representation 
$\rho$ of $H$ makes it act   
 on the Hilbert space $L^2(\R^n)$ by means of the operators:  
\[ \rho(x_s)f(z_1,\ldots,z_n)= z_sf(z_1,\ldots,z_n) , \;\;\, \rho(y_s)=-2\pi i \, h \frac{\partial f}{\partial z_s}. \]
{\em The Heisenberg algebras} $H_{\tau}^h$  are defined by commutation relations with 
constant coefficients and hence their representations can be constructed 
by selecting a Lagrangian subspace in the generators $x_s$   --  called a polarization  -- 
and letting the generators act as linear combinations in the above 
operators $\rho(x_s)$ and $\rho(y_s)$. 

\vspace{0.2cm}
\noindent 
The Stone-von Neumann theorem holds then for these algebras. 
Specifically, there exists a unique unitary irreducible Hilbert space representation 
of given central character that is integrable i.e. which can be integrated to the 
corresponding Lie group. Notice that  there exist in general also 
non-integrable unitary representations. 

\vspace{0.2cm}
\noindent 
In particular we obtain representations of $H_{\tau}^h$ and  $H_{\tau'}^h$. 
The uniqueness of the representation yields the existence of  an intertwiner ${\mathbf K}_{(\tau,\tau')}$
(defined up to a scalar) between the two representations. 
However,  neither the Hilbert spaces nor the representations $\rho_{\tau}$ are 
canonical, 
as they depend on the choice of the polarization.

\vspace{0.2cm}
\noindent 
We will give below {\em the construction of a canonical representation} when
the quantized Teichm\"uller space  is replaced by its double. 
We need first to switch to another system of coordinates, 
coming from the cluster  ${\mathcal A}$-varieties. 
Define, after Penner (see \cite{pe0}),  
the  {\em universal decorated Teichm\"uller space} ${\mathcal A}$ to be the space 
of all marked tessellations endowed with one horocycle for each vertex (decoration).   
Alternatively (see \cite{FG1}), 
${\mathcal A}$ is the set of positive real points of the 
cluster ${\mathcal A}$-space related to the previous set of seeds. 

\vspace{0.2cm}
\noindent 
Each tessellation $\tau$  yields a coordinate system 
$\alpha_{\tau}:{\mathcal A}\to \R^{E(\tau)}$ which associates to the edge 
$e$ of $\tau$ the coordinate $a_e=\alpha_{\tau}(e)\in \R$. 
The number $\alpha_{\tau}(e)$  is the algebraic distance between 
the two horocycles on ${\mathbb H}^2$ centered at vertices of $e$, 
measured along the geodesic associated to $e$.  These are the so-called {\em lambda-length} \index{lambda-length coordinates of Penner}
coordinates of Penner. 

\vspace{0.2cm}
\noindent 
There is a canonical  map $p:{\mathcal A}\to {\mathcal T}$  (see \cite{pe0}, Proposition 3.7 and 
\cite{FG1}) 
such that, in the coordinate systems induced by a tessellation $\tau$, the corresponding map 
$p_{\tau} :\R^{(E(\tau)}\to \R^{E(\tau)}$ is given by 
\[ p_{\tau}\left(\sum_{t\in E(\tau)} \varepsilon(\tau)_{st}a_t\right)=x_s.\]

\vspace{0.2cm}
\noindent 
Let  $(\tau,\tau')$ be the flip on the edge $e$ and set $a'_s$ be the coordinates system 
associated to $\tau'$. Then the flip induces the  following 
change of coordinates:  
\[ \alpha_{\tau,\tau'}(a_s)=
a_s, \: \mbox{\rm if }\: s\neq e \]
\[ \alpha_{\tau,\tau'}(a_e)=
-a_e + \log\left(\exp\left(\sum_{t;\varepsilon(\tau)_{et}>0}
\varepsilon(\tau)_{et} a_t\right)+ \exp\left(-\sum_{t;\varepsilon(\tau)_{et}<0}
\varepsilon(\tau)_{et} a_t\right)\right).
\]
It can be verified that $p_{\tau}$ are compatible with the action of the 
Ptolemy groupoid on the respective coordinates.

\vspace{0.2cm}
\noindent 
{\em The vector space} ${\mathcal L}_{\tau}$ is defined as the space of 
square integrable functions  with finite dimensional support 
on ${\mathcal A}$ with respect to the $\alpha_{\tau}$ coordinates 
i.e. the functions $f:\R^{E(\tau)}\to \C$,  with support contained in 
some $R^F\times \{0\}\subset R^{E(\tau)}$, for some finite subset 
$F\subset E(\tau)$. The coordinates on $\R^{E(\tau)}$ are the $a_e, e\in E(\tau)$. 
The function $f$ is square integrable if 
\[ \int_{\R^F} |f|^2 \bigwedge _{e\in F}da_e  < \infty \]
for any such $F$ as above. Let $f, g \in {\mathcal L}_{\tau}$. Then let 
 $\R^F\times \{0\}$ contain the intersection of their supports. 
Choose $F$ minimal with this property. Then the scalar product 
\[ \langle f, g \rangle= \int_{\R^F} f(a)\overline{g(a)} \bigwedge _{e\in F}da_e \]
makes ${\mathcal L}_{\tau}$ a Hilbert space.

\vspace{0.2cm}
\noindent 
To define the intertwining operator ${\mathbf K}$ we set now: 
\begin{align*} 
G_e((a_s)_{s\in F}) &=
\int \exp\left(\int_{\Omega}
\frac{\exp(it \sum_{s\in F}\varepsilon(\tau)_{es}a_s) \sin(tc)}
{2i {\rm sh}(\pi t){\rm sh}(\pi h t)} \frac{dt}{t} + \right. \\
& 
\left.\frac{c}{\pi i h}\left(\sum_{s; \varepsilon(\tau)_{es}<0}\varepsilon(\tau)_{es}a_s 
+a_e\right)\right) dc.  
\end{align*}
The key ingredient in the construction of this function is the 
{\em quantum dilogarithm} \index{quantum dilogarithm}
(going back to  Barnes (\cite{Bar}) and 
used by Baxter (\cite{Bax}) and Faddeev (\cite{Fad})): 
\[ \Phi^h(z) = \exp\left(-\frac{1}{4}\int_{\Omega}
\frac{\exp(-it z)}{{\rm sh}(\pi t) \, {\rm sh }(\pi h t)}  \frac{dt}{t} \right)\]
where the contour $\Omega$ goes along the real axes from $-\infty$ 
to $\infty$ bypassing the origin from above. 

\vspace{0.2cm}
\noindent Some properties of the quantum dilogarithm are collected below: 
\[ 2\pi i h d \log \Phi^h(z)=\phi^h(z), \; \;  \lim_{\Re z\to -\infty}\Phi^h(z)=1\]
\[ \lim_{h\to 0}\Phi^h(z)/\exp(-{\rm Li}_2(-\exp(z)))=2\pi i h, \, \, \mbox{\rm where } 
{\rm Li}_2(z)=\int_0^z\log(1-t)dt \]
\[ \Phi^h(z)\Phi^h(-z)=\exp\left(\frac{z^2}{4\pi i h}\right) \exp\left(-\frac{\pi i}{12}(h+h^{-1})\right)\]
\[ 
 \overline{\Phi^h(z))}=\left(\Phi^h\left(\overline{z}\right)\right)^{-1}, \;\; \Phi^h(z)=\Phi^{1/h}\left(\frac{z}{h}\right).\]

\vspace{0.2cm}
\noindent 
Now let  $f\in {\mathcal L}_{\tau}$, namely $f:\R^F\times \{0\}\to \C$. 
Let $(\tau,\tau')$ be the flip $F_e$ on the edge $e$. 
Let $a_s, s\in F$ be the coordinates in $\R^F$. If $e\not\in F$ then we set 
\[ {\bf K}_{(\tau,\tau')}=1.\]
 If $e\in F$ then the coordinates associated to $\tau'$ are $a_s, s\neq e$ and $a'_e$. 
Set then 
\[ ({\bf K}_{(\tau,\tau')} f)(a_s, _{s\in F, f\neq e}, a'_e)= 
\int G_e((a_s)_{s\in F, s\neq e}, a_e+a'_e) f((a_s)_{s\in F}) da_s. \]

\vspace{0.2cm}
\noindent 
The last piece of data is {\em the representation of the Heisenberg algebra} 
\index{representation of Heisenberg algebra} 
$H_{\tau}^h$ in the Hilbert space ${\mathcal L}_{\tau}$. We can 
actually  do better, namely enhance the space with a bimodule structure. 
Set 
\[ \rho^-_{\tau} (x_s) =-\pi i h \frac{\partial}{\partial a_s} + \sum_{t}\varepsilon(\tau)_{st} a_t \]
and 
\[ \rho^+_{\tau} (x_s)=\pi i h \frac{\partial}{\partial a_s} + \sum_{t}\varepsilon(\tau)_{st} a_t. \]
Then $\rho^-_{\tau} $ gives a left module and $\rho^+_{\tau} $ a  right module structure on 
${\mathcal L}_{\tau}$ and the two actions commute. 
%Denote by $H_{\tau}^{h,\rm opp}$ 
%the algebra $H_{\tau}^h$ with the new product $x\cdot y=yx$. 
%Then $\rho_{\tau} =\rho^-_{\tau} \otimes \rho^+_{\tau} $ is a representation of 
%$H_{\tau}^h\otimes H_{\tau}^{h,\rm opp}$ in the Hilbert space ${\mathcal L}_{\tau}$. 
%We keep denoting by $\beta^h$ the action of the Ptolemy groupoid on 
%$H_{\tau}^h\otimes H_{\tau}^{h,\rm opp}$, which is a quantization of 
%${\mathcal T}\times {\mathcal T}$ (the double of the quantum Teichm\"uller space). 
We have then: 

\begin{proposition}[\cite{CF,FG3,FG}]
The data $({\mathcal L}_{\tau}, \rho^{\pm}_{\tau}, {\mathbf K}_{(\tau,\tau')})$ 
is a projective $*$-representation of the  quantized universal Teichm\"uller space.   \index{projective $*$-representation of the  quantized Teichm\"uller space}
\end{proposition}

\vspace{0.2cm}
\noindent
The data $({\mathcal L}_{\tau}, \rho^{\pm}_{\tau}, {\mathbf K}_{(\tau,\tau')})$ is called  the dilogarithmic representation of the Ptolemy groupoid.\index{dilogarithmic representation of Ptolemy groupoid}  
The proof of this result is given in \cite{FG} and  a particular case is explained 
with lots of details  in \cite{Go}. 

\vspace{0.2cm}
\noindent
The last step in our construction is to observe that a 
representation of the Ptolemy groupoid 
$Pt$ {\em induces a representation of the Ptolemy-Thompson group} 
\index{dilogarithmic representation Ptolemy-Thompson group} $T$ by means of 
an identification of the Hilbert spaces ${\mathcal L}_{\tau}$ for all $\tau$.

\vspace{0.2cm}
\noindent
Projective representations are equivalent to representations of 
central extensions by  means of the following well-known procedure.
To a general group $G$,  Hilbert space $V$ and homomorphism 
$A:G\to {\rm PGL}(V)$ we can associate a central extension $\widetilde{G}$ 
of $G$ by $\C^*$ which resolves the projective representation $A$ 
to a linear representation $\widetilde{A}:\widetilde{G}\to {\rm GL}(V)$. 
The extension $\widetilde{G}$ is the pull-back on $G$ of the 
canonical central $\C^*$-extension ${\rm GL}(V)\to {\rm PGL}(V)$. 

\vspace{0.2cm}
\noindent
However the central extension which we consider here is a subgroup 
of the $\C^*$-extension defined above, obtained by using 
a particular section over $G$. Let us 
write $G=F/R$ as the quotient of the free group $F$ by the normal 
subgroup $R$ generated by the relations.  Then our data 
consists in a homomorphism $\overline{A}:F\to {\rm GL}(V)$ with 
the property that $\overline{A}(r)\in\C^*$, for each 
relation $r\in R$, so that $\overline{A}$ induces $A:G\to {\rm PGL}(V)$. 
This data will  be called 
{\em an almost-linear representation}, in order to distinguish 
it from a  projective representation of $G$. 

\vspace{0.2cm}
\noindent
The {\em central extension $\widehat{G}$ of $G$ associated to 
$\overline{A}$} \index{central extension}
is $\widehat{G}=F/(\ker \overline{A}\cap R)$, namely the smallest  central 
extension  of $S$ resolving the projective representation $A$ to a linear 
representation compatible with $\overline{A}$. Then 
$\widehat{G}$ is a central extension of $G$ by the subgroup 
$\overline{A}(R)\subset \C^*$ and hence it is naturally a subgroup 
of $\widetilde{G}$. In other terms $\overline{A}$ determines a 
projective representation $A$ and a section over $G$ whose 
associated 2-cocycle takes values in $\overline{A}(R)$ and which 
describes the central extension $\widehat{G}$.

\vspace{0.2cm}
\noindent
Now, the intertwiner functor $\mathbf K$ \index{intertwinner}  is actually an almost-linear  
representation (in the obvious sense) of the 
Ptolemy groupoid and thus induces an  
almost-linear representation of the Ptolemy-Thompson group $T$ 
into the unitary group. 
We can extract from \cite{FG} the following results  (see also the equivalent 
construction at the level of Heisenberg  algebras in \cite{BBL}):

\begin{proposition}\label{dilogcar}
The dilogarithmic almost-linear 
representation $\mathbf K$ has  the following properties: 
\begin{enumerate}
\item images of disjoint flips in $\widehat{T}$  commute with each other; 
\item the square of a flip is the identity;
\item the composition of the  lifts of the five flips from the 
pentagon relation below is $\exp(2\pi i h)$ times the identity.

\begin{center}
\includegraphics{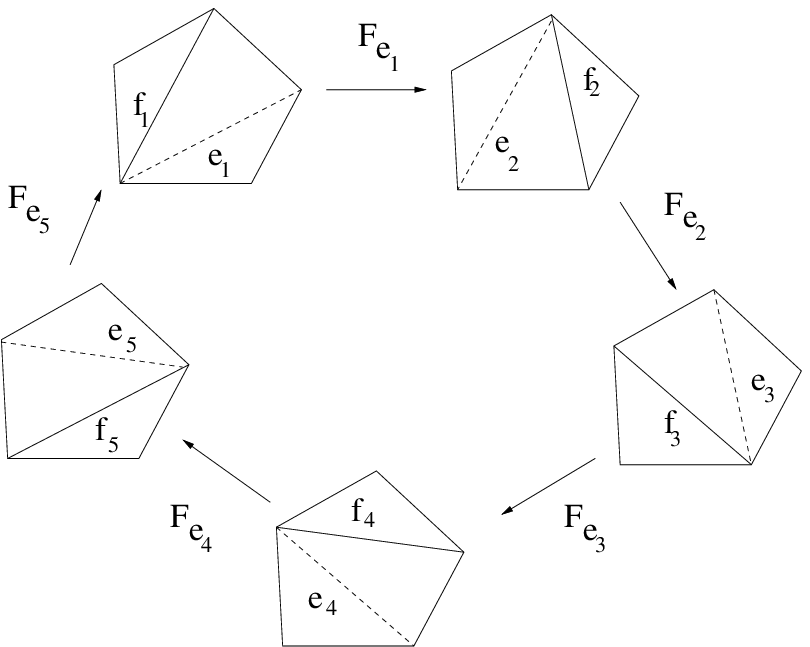}
\end{center}

\end{enumerate}
\end{proposition}
\begin{proof}
The first condition is that images by $K$ of flips on disjoint edges 
should commute. This is obvious by the explicit formula for $K$. 
The second and third conditions are proved in \cite{FG}. 
\end{proof}

\begin{remark}
In \cite{fu-se} we considered the action of labelled flips in the 
pentagon relation. The composition of labelled flips is then the 
transposition of the labels of the two diagonals 
(see Proposition \ref{permut}). 
\end{remark}

\vspace{0.1cm}
\noindent 
Therefore the image by $\mathbf K$ of relations of the Ptolemy groupoid 
into $\C^*$ is the subgroup $U$ generated by $\exp(2\pi i h)$.  
We can view  the pentagon relation in the Ptolemy-Thompson 
group $T$ as a pentagon relation in the Ptolemy groupoid $Pt$. Thus  
the image by $\mathbf K$ of relations of the Ptolemy-Thompson group $T$ 
into $\C^*$  is  also the subgroup $U$. In particular the associated 
2-cocycle takes values in $U$.  
If $h$ is a formal parameter or an irrational 
real number we obtain then a 2-cocycle with values in $\Z$. 

\begin{definition}
The dilogarithmic central extension 
$\widehat{T}$ is the central extension of $T$ by $\Z$  
associated to the dilogarithmic almost-linear representation $\mathbf K$ 
of $T$, or equivalently, to the previous 2-cocycle. 
\end{definition}

\subsection{The relative abelianization of the 
braided Ptolemy-Thompson group $T^*$}
Recall from \cite{fu-ka2,fu-ka4} (see also \ref{ptdef}) that there 
exists a natural surjective homomorphism $T^*\to T$  
which is obtained by {\em forgetting the punctures} once 
these groups are considered as mapping class groups.  
Its kernel is the infinite braid group $B_{\infty}$ 
consisting of those braids in the punctures 
of $D^*$ that move non-trivially only finitely many punctures. 
In other words $B_{\infty}$ is the direct limit of an ascending 
sequence of braid groups associated to an exhaustion of $D^*$ by punctured 
disks.  This yields the  following exact sequence description of $T^*$:
\[ 1 \to B_{\infty} \to \T\to T \to 1. \]

\vspace{0.2cm}
\noindent 
Observe that $H_1(B_{\infty})=\Z$. Thus, the abelianization homomorphism  
$B_{\infty}\to H_1(B_{\infty})=\Z$  
induces a central extension $T^*_{\rm ab}$ of $T$, where one 
replaces $B_{\infty}$ by its abelianization 
$H_1(B_{\infty})$, as in the diagram below: 

\[ \begin{array}{ccccccc}
1  \to & B_{\infty} & \to & \T         & \to &  T        & \to 1 \\
       & \downarrow &     & \downarrow &     & \parallel &   \\ 
1  \to & \Z         & \to & T^*_{\rm ab}          & \to &   T       & \to 1  
\end{array}
\]

\vspace{0.2cm}
\noindent Then $T^*_{\rm ab}$ is the relative abelianization of $T^*$ over $T$. 
We are not only able to make computations in the 
mapping class group $T^*$ and thus in $T^*_{\rm ab}$,  but also 
to interpret the algebraic relations in $T^*_{\rm ab}$ in geometric terms. 
 
\begin{proposition}\label{abelcar}
The group $T^*_{\rm ab}$ has the presentation with three generators 
$\alpha^*_{\rm ab}$, $\beta^*_{\rm ab}$ and $z$ and the relations 
\[ {\alpha^*_{\rm ab}}^4={\beta^*_{\rm ab}}^3=1, (\beta^*_{\rm ab}\alpha^*_{\rm ab})^5=z, 
[\alpha^*_{\rm ab}, z]=1, [\beta^*_{\rm ab}, z]=1\]
\[ [\beta^*_{\rm ab}\alpha^*_{\rm ab}\beta^*_{\rm ab} \, , \, {\alpha^*_{\rm ab}}^2\beta^*_{\rm ab}\alpha^*_{\rm ab}\beta^*_{\rm ab}{\alpha^*_{\rm ab}}^2]=1\]
\[[\beta^*_{\rm ab}\alpha^*_{\rm ab}\beta^*_{\rm ab} \, , \, {\alpha^*_{\rm ab}}^2{\beta^*_{\rm ab}}^2 
{\alpha^*_{\rm ab}}^2\beta^*_{\rm ab}\alpha^*_{\rm ab}\beta^*_{\rm ab}{\alpha^*_{\rm ab}}^2{\beta^*_{\rm ab}}
{\alpha^*_{\rm ab}}^2]=1\]
Moreover the projection map $T^*_{\rm ab}\to T$ sends $\alpha^*_{\rm ab}$ to $\alpha$, 
$\beta^*_{\rm ab}$ to $\beta$ and $z$ to the identity.  
\end{proposition}
\begin{proof}
Recall from \cite{fu-ka2} that $T^*$ is generated by the two elements $\alpha^*$ and 
$\beta^*$ below. 
\begin{itemize} 
\item The support of the element $\beta^{*}$ of $T^{*}$ is the central hexagon.  Furthermore $\beta^{*}$ acts as the counterclockwise rotation of order  
three which permutes cyclically  
the punctures.  One has ${\beta^{*}}^{3}=1$.

\begin{center}
\includegraphics{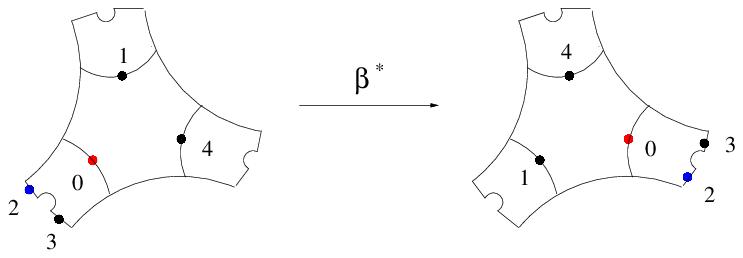}
\end{center}

\item  The support of the element $\alpha^{*}$ of $T^{*}$ is the union 
of two adjacent hexagons, 
one of them being the support of $\beta^{*}$.  
Then $\alpha^{*}$ rotates counterclockwise the support 
of angle $\frac{\pi}{2}$, by  
keeping fixed the central puncture. One has  ${\alpha^{*}}^4=1$. 

\begin{center}
\includegraphics{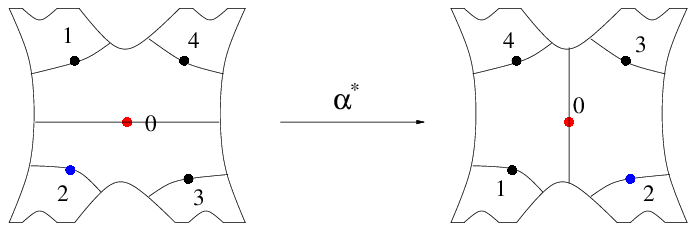}
\end{center}
\end{itemize}

\vspace{0.2cm}
\noindent
Let now $e$ be a simple arc in  $D^{*}$ which connects two punctures. 
We associate a braiding $\sigma_e\in B_{\infty}$ to $e$ by 
considering the homeomorphism
that moves clockwise the punctures at the 
endpoints of the edge $e$ in a small neighborhood of the edge,  
in order to interchange their positions. This means that if $\gamma$ is an arc transverse to $e$, then the braiding $\sigma_{e}$
moves $\gamma$ to the left when it approaches $e$. Such a braiding will be called {\it positive}, while $\sigma_{e}^{-1}$ is
{\it negative}.

\begin{center} 
\includegraphics{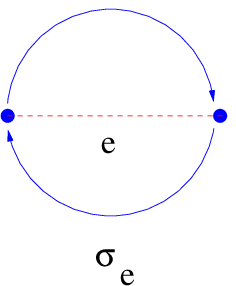}      
%\caption{$\sigma_e$} 
\end{center}

\vspace{0.2cm}
\noindent
It is known that $B_{\infty}$ is generated by the braids $\sigma_e$ 
where $e$ runs over the edges of the binary tree with vertices 
at the punctures. Let $\iota:B_{\infty}\to T^*$ be the inclusion. 
It is proved in \cite{fu-ka2} that 
the braid generator $\sigma_{[02]}$ associated to the edge joining the 
punctures numbered $0$ and $2$ has image 
\[ \iota(\sigma_{[02]})=(\bps\aps)^5\]
because we have 

\begin{center}
\includegraphics{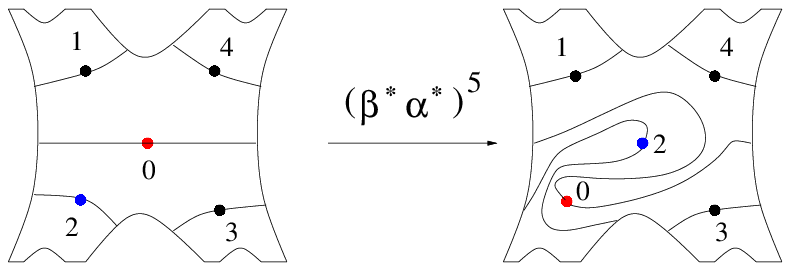}
\end{center}

\vspace{0.2cm}
\noindent
Recall next that all braid generators $\sigma_e$ are conjugate and 
call $z$ their image in $T^*_{\rm ab}$. It follows that $T^*_{\rm ab}$ is an 
extension of $T$ by $\Z$. Moreover, it is simple to check that 
$\aps\sigma_{[02]}\alpha^{-1}$ is also a braid generator, namely 
$\sigma_{[\aps(0)\aps(2)]}$. The same holds for 
$\bps\sigma_{[02]}\bps^{-1}=\sigma_{[\bps(0)\bps(2)]}$. 
This implies that the extension $T^*_{\rm ab}$ is central.  

\vspace{0.2cm}
\noindent
In particular, a presentation of $T^*_{\rm ab}$ can be obtained by  looking at 
the lifts of  relations in $T$, together with those coming from 
the fact that $z$ is central. 

\vspace{0.2cm}
\noindent
The first  set of relations above are obviously satisfied by $T^*_{\rm ab}$. 
Finally recall from \cite{fu-ka2} that $T^*$ splits over the smaller 
Thompson group $F$ and thus the following relations hold true in $T^*$:
\[ [\beta^*\alpha^*\beta^* \, , \, {\alpha^*}^2\beta^*\alpha^*\beta^*{\alpha^*}^2]=
[\beta^*\alpha^*\beta^* \, , \, {\alpha^*}^2{\beta^*}^2 
{\alpha^*}^2\beta^*\alpha^*\beta^*{\alpha^*}^2{\beta^*}
{\alpha^*}^2]=1.\]
Thus  relations from the second set are    
automatically verified in $T^*_{\rm ab}$. 
Since these relations form a complete set of lifts of  relations 
presenting $T$ and since $z$ is central, then they represent a complete 
system of relations in $T^*_{\rm ab}$. This ends the proof. 
\end{proof}

\subsection{Computing the class of $T^*_{\rm ab}$}\label{comput}

\begin{lemma}\label{multiple}
The class $c_{T^*_{\rm ab}}$ is a multiple of the Euler class. \index{extension class} 
\end{lemma}
\begin{proof}
Since $T^*$ splits over the Thompson group $F\subset T$ (see \cite{fu-ka2}) it follows that 
$T^*_{\rm ab}$ also splits over $F$. Therefore the extension class  $c_{T^*_{\rm ab}}$ 
lies in the kernel of the restriction map $H^2(T)\to H^2(F)$. According to \cite{gh-se} 
the kernel is generated by the Euler class. \index{Euler class}
\end{proof}

\vspace{0.2cm}
\noindent
Let us introduce the group $T_{n,p,q}$ presented by  the generators  
$\als, \bls, z$ and the relations: 
\[ (\bls\als)^5=z^n \]  
\[ \als^4=z^p \]
\[ \bls^3=z^q\]
\[ [\bls\als\bls, \als^2\bls\als\bls\als^2]=1\]
\[[\bls\als\bls, \als^2\bls\als^2\bls\als\bls\als^2\bls^2\als^2]=1\]
and
\[ [\als,z]=[\bls,z]=1.\]
Recall from Proposition \ref{abelcar} 
that $T^*_{\rm ab}=T_{1,0,0}$. It is easy to see that $T_{n,p,q}$ are central extensions 
of $T$ by $\Z$.  Because of the last two commutation relations the extension 
$T_{n,p,q}$ splits over the Thompson group $F$. 
Thus the restriction of $c_{T_{n,p,q}}$ to $F$ vanishes and 
a fortiori the restriction to the commutator subgroup  $F'\subset F$.   
According to \cite{gh-se} we have $H^2(F')=\Z\alpha$ where 
$\alpha$ is the discrete  Godbillon-Vey class. Thus the map 
$H^2(T)\to H^2(F')$ is the projection $\Z\alpha\oplus \Z\chi\to \Z\alpha$. 
Since $c_{\widehat{T}}$ belongs to the 
kernel of $H^2(T)\to H^2(F')$ we derive that 
$c_{T_{n,p,q}}\in \Z\chi$. Set $c_{T_{n,p,q}}=\chi(n,p,q) \chi$.

\begin{proposition}[\cite{fu-se}]\label{linearform}
We have $\chi(n,p,q)= 12n-15p-20q$. 
\end{proposition}

\begin{corollary}\label{12}
We have $c_{T^*_{\rm ab}}=12\chi$. 
\end{corollary}

\begin{remark}
The extension $T^*\to T$ splits also 
over the subgroup $\langle \alpha^2,\beta\rangle$ which is isomorphic to ${\rm PSL}(2,\Z)$ 
(see \cite{fu-ka2}). This implies that $c_{T^*_{\rm ab}}$ is a multiple of $6\chi$.  
\end{remark}

\begin{remark}
The following refinement of the above argument shows that 
 $c_{T^*_{\rm ab}}$ is a multiple of $12\chi$.  
Consider the element $\gamma=\alpha\beta\alpha\beta\alpha^2\beta\alpha\beta\alpha^2$ in $T$. Then $\gamma^3=\alpha^2$ so that $\gamma^6=\alpha^4=1$. 
Let then 
$\gamma^*_{\rm ab}=\alpha^*_{\rm ab}\beta^*_{\rm ab}\alpha^*_{\rm ab}\beta^*_{\rm ab}
{\alpha^*_{\rm ab}}^2\beta^*_{\rm ab}\alpha^*_{\rm ab}\beta^*_{\rm ab}{\alpha^*_{\rm ab}}^2$ be a lift of $\gamma$ to $T^*_{\rm ab}$. One can show that 
$(z^{-1}\gamma^*_{\rm ab})^3={\alpha^*_{\rm ab}}^2$. 
The elements $\gamma$ and $\alpha$ of orders 6 and 4 respectively 
determine an embedding of $\Z/6\Z *_{\Z/2\Z}\Z/4\Z$ into $T$. 
The relation from above shows that one can lift this embedding
to an embedding of  $\Z/6\Z *_{\Z/2\Z}\Z/4\Z$ into $T^*_{\rm ab}$ 
by using the lifts $z^{-1}\gamma^*_{\rm ab}$ and $\alpha^*_{\rm ab}$. 
Now we know that $\Z/6\Z *_{\Z/2\Z}\Z/4\Z$ is isomorphic to ${\rm SL}(2,\Z)$ and 
$H^2({\rm SL}(2,\Z),\Z)=\Z/12\Z$. Moreover the pull-back of the Euler 
class on ${\rm SL}(2,\Z) \subset T\subset {\rm Homeo}^+(S^1)$ is the generator 
of $Z/12\Z$. This implies that  $c_{T^*_{\rm ab}}$ is a multiple of $12\chi$.
\end{remark}

\subsection{Identifying the two central extensions of $T$}\label{ident}

\vspace{0.2cm}
\noindent 
The main result of this section is the following: 

\begin{proposition} \index{dilogarithmic extension}
The dilogarithmic extension $\widehat{T}$ is identified to $T^*_{\rm ab}$. 
\end{proposition}
\begin{proof}
The main step is to translate the properties of the dilogarithmic 
representation of the Ptolemy groupoid in terms of the Ptolemy-Thompson group.  
Since $\widehat{T}$ is a central extension 
of $T$ it is generated by the lifts $\al,\be$ of $\alpha$ and $\beta$ 
together with the generator $z$ of the center. 
Let us see what are the relations arising in the 
group $\widehat{T}$.  According to Proposition \ref{dilogcar} 
lifts of disjoint flips should commute. By a simple computation we can 
show that 
the elements $\beta\alpha\beta$, $\alpha^2\beta\alpha\beta\alpha^2$ and 
$\alpha^2\beta\alpha^2\beta\alpha\beta\alpha^2\beta^2\alpha^2$ act as 
disjoint flips on the Farey triangulation. 
In particular we have the relations 
\[  [\be\al\be,\al^2\be\al\be\al^2]=
 [\be\al\be,\al^2\be\al^2\be\al\be\al^2\be^2\al^2]=1\]
satisfied in $\widehat{T}$. Moreover, by construction we also have  
\[\be^3=\al^4=1\]
meaning that the $\al$ is still periodic of order $4$ while $\be$ is not deformed.

\vspace{0.2cm}
\noindent 
Eventually the only non-trivial lift of relations comes from the 
pentagon relation $(\be\al)^5$.  The element $(\be\al)^5$ is actually 
the permutation of the two edges in the pentagon times the composition of the five flips.  
The pentagon equation is not anymore satisfied 
but Proposition \ref{dilogcar}  shows that the dilogarithmic image of  
$(\be\al)^5$ is a scalar operator. 
Since $z$ is the generator of the 
kernel $\Z$ of $\widehat{T}\to T$  
it follows that the lift of the pentagon equation from $T$ to 
$\widehat{T}$ is given by 
\[ (\be\al)^5=z.\]
According to Proposition \ref{abelcar} all relations presenting 
$T^*_{\rm ab}$ are satisfied  in $\widehat{T}$. Since $\widehat{T}$ is a 
nontrivial  central extension of $T$ by $\Z$ 
it follows that the groups are isomorphic.  
\end{proof}

\begin{remark}
The key point in the above proof  is that all pentagon relations in $Pt$ 
are transformed in a single pentagon relation in $T$ and thus 
the scalars associated to the pentagons in $Pt$  should be the same.   
\end{remark}

\begin{remark}
The dilogarithmic representation of $T$ induces a projective representation 
of the smaller Thompson group $F\subset T$. The latter is equivalent to a 
linear representation since the braided Ptolemy-Thompson group splits 
over $F$. It is presently 
unknown whether the dilogarithmic representation can be 
extended to one of the groups $V$ or $\mathcal B$.    
\end{remark}

\subsection{Classification of central extensions of the group $T$}
Our main concern here is to  identify the cohomology classes  
of all central extensions of $T$ in $H^2(T)$. 
Before doing that we consider a series of 
central extensions $T_{n,p,q,r,s}$ of $T$ by $\Z$, having properties 
similar to those of  $\widehat{T}$. \index{classification central extensions of $T$}

\begin{definition}
The group $T_{n,p,q,r,s}$, is presented by  the generators  
$\als, \bls, z$ and the relations: 
\[ (\bls\als)^5=z^n \]  
\[ \als^4=z^p \]
\[ \bls^3=z^q\]
\[ [\bls\als\bls, \als^2\bls\als\bls\als^2]=z^r\]
\[[\bls\als\bls, \als^2\bls\als^2\bls\als\bls\als^2\bls^2\als^2]=z^s\]
\[ [\als,z]=[\bls,z]=1.\]
Let us denote  $T_{n,p,q,r}=T_{n,p,q,r,0}$ and 
$T_{n,p,q}=T_{n,p,q,0,0}$. 
\end{definition}

\vspace{0.2cm}
\noindent 
According to \cite{fu-ka2} we can identify $\widehat{T}$ with 
$T_{1,0,0}$. In fact the group $T^*$ is split over 
 the smaller Thompson group $F\subset T$ and thus $\widehat{T}$ is 
split over $F$. Furthermore $F$ is generated by the elements 
$\beta^2\alpha$ and $\beta\alpha^2$ and thus relations of $F$ are precisely 
given by the above commutation relations. 
Thus the last two relations hold, while 
$z$ is central and thus $\widehat{T}$ is given by the above 
 presentation. 

\begin{remark} 
We considered in \cite{fu-ka2} the twin group $T^{\sharp}$ 
and gave a presentation 
of it.  Then, using a similar procedure there is a group obtained from 
$T^{\sharp}$ 
by abelianizing the kernel $B_{\infty}$, which is identified  
to $T_{3,1,0}$. 
\end{remark}

\begin{theorem}[\cite{fu-se}]\label{exte0}
Every central extension of $T$ by $\Z$ is of the form $T_{n,p,q,r}$. Moreover,   
the class $c_{T_{n,p,q,r}}\in H^2(T)$ of the extension $T_{n,p,q,r}$  
is given by:  
\[ c_{T_{n,p,q,r}}= (12n-15p-20q-60r) \chi + r \alpha. \]
Moreover the central $T_{n,p,q}$ are precisely those central extensions 
whose restrictions on $F\subset T$ splits. 
\end{theorem}

\section{More asymptotically rigid mapping class groups}

\subsection{Other planar surfaces and braided Houghton groups}
\vspace{0.2cm} 
\noindent 
The aim of this section is to use the previous methods  
in order to recover the braided Houghton groups as mapping class groups  
of surfaces of infinite type. In particular the braid group on infinitely 
many strands is realized as the commutator subgroup of an explicit 
finitely presented group. This has been done previously by Dynnikov  
who used the so-called three pages representations 
of braids and links in (\cite{Dy}). Our groups are slightly 
different from those considered by Dynnikov and their presentation is 
of a different nature, because it comes from a geometric description 
in terms of mapping classes. Moreover, we obtain that the word problem of the 
braided Houghton groups is solvable.  
A version of our construction was used by Degenhardt, who  
introduced the braided Houghton groups $BH_n$ in his 
(unpublished) thesis \cite{Deg}.  Then Kai-Uwe Bux described 
a conjectural approach to the finiteness properties of these 
groups in \cite{Bux}.

\vspace{0.2cm} 
\noindent 
In order to define the mapping class group of a surface of infinite type we need 
to fix the behavior of homeomorphisms at infinity. The main 
ingredient used in \cite{fu-ka2}  consists of adjoining rigid structures, 
as defined below in a slightly more general context: 

\begin{definition}
A {\em rigid structure} $d$  on the surface \index{rigid structure}
$\Sigma$ is a decomposition  of $\Sigma$ into 2-disks with disjoint 
interiors, called elementary pieces. We suppose that the closures 
of the elementary pieces are still 2-disks.  

\vspace{0.2cm}
\noindent
We assume that we are given a family $F$ of compact subsurfaces of $\Sigma$ 
such that each member of $F$ is a finite union of elementary pieces, 
and called the family of admissible subsurfaces of $\Sigma$. \index{admissible subsurface}
\end{definition}

\vspace{0.2cm}
\noindent
To the data $(\Sigma, d, F)$ we can associate an asymptotically rigid  
mapping class group \index{asymptotically rigid mapping class group} ${\cal M}(\Sigma, d, F)$ as follows. We first restrict to 
those homeomorphisms that act in the simplest possible way at infinity. 

\begin{definition}
A homeomorphism $\varphi$ between two surfaces endowed 
with rigid structures is {\em rigid} if it sends the rigid structure of one 
surface onto the rigid structure of the other. 

\vspace{0.2cm}
\noindent
The homeomorphism $\varphi:\Sigma\to \Sigma$ is  said to be 
{\em asymptotically rigid}  \index{asymptotically rigid homeomorphism}
if  there exists some admissible subsurface 
$C\subset \Sigma$, called a support for $\varphi$, 
such that $\varphi(C)\subset \Sigma$ is 
also an admissible subsurface of $\Sigma$ and  
the restriction $\varphi|_{\Sigma-C}:\Sigma-C\to \Sigma-\varphi(C)$ 
is rigid. 
\end{definition}

\vspace{0.2cm}
\noindent As it is customary when studying mapping class groups  
we now consider isotopy classes of such homeomorphisms. 

\begin{definition}
The group ${\cal M}(\Sigma, d, F)$ of isotopy classes of asymptotically rigid homeomorphisms is  called the 
{\em asymptotically rigid mapping class group} \index{asymptotically rigid mapping class group}
 of $\Sigma$ corresponding to the  
rigid structure $d$ and family  of admissible subsurfaces $F$. 
\end{definition}

\begin{remark}
Two asymptotically rigid homeomorphisms that are isotopic are  
isotopic among asymptotically rigid homeomorphisms. 
\end{remark}

\vspace{0.2cm}
\noindent 
The ribbon tree $D$ and punctured ribbon trees $D^*$ and $D^{\sharp}$ 
are particular examples of surfaces of infinite type with rigid structures.  
We want to turn to even simpler examples obtained from thickening 
trees in the plane and show that interesting groups could be 
obtained in this way.  
Consider the planar ribbon $Y_n$, which is a 2-dimensional neighborhood  
of the wedge of $n$ half-lines (or rays) 
in the plane that intersect at the origin. 
Assume that every half-line is endowed with a linear coordinates system 
in which the origin corresponds to $0$ and that the rotation of order $n$ 
sends them isometrically one into the other. 
 
\vspace{0.2cm}
\noindent 
Let  $Y_n^*$  (respectively $Y_n^{\sharp}$) be  the 
punctured ribbon obtained from 
$Y_n$  by puncturing it along the set of points of   positive 
(respectively nonnegative) integer coordinates on each 
half-line. Punctures are therefore identified with nonnegative 
integers along each ray.  
The origin has coordinates $0$ on all half-lines and does appear 
only in $Y_n^{\sharp}$.

\begin{center}
\includegraphics{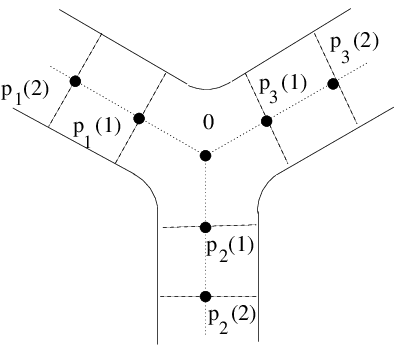}      
%\caption{The punctured surface $Y_3^*$}
\end{center}

\vspace{0.2cm}
\noindent
There is a  family of parallel  arcs associated to each ray, 
obtained by drawing a properly  embedded segment orthogonal to the 
respective half-line and passing through the puncture labelled 
$n$, for every $n\in \Z_+-\{0\}$.   

\vspace{0.2cm}
\noindent
The surface $Y_n$ (respectively $Y_n^*$, $Y_n^{\sharp}$) is then divided by these arcs 
into elementary pieces, which are of two types: 
one central  (respectively punctured for $Y_n^*$) $2n$-gon  containing the  
origin and infinitely many (respectively punctured) 
squares which sit along the half-lines. 
One defines the admissible subsurfaces of $Y_n$ (respectively  $Y_n^*$, 
$Y_n^{\sharp}$)  
to be those (punctured)  $2n$-gons which contain the 
(punctured)  central $2n$-gon and are made of finitely many elementary 
pieces.   

\vspace{0.2cm}
\noindent
Let ${\mathcal M}(Y_n)$ (respectively ${\mathcal M}(Y_n^*)$, 
${\mathcal M}(Y_n^{\sharp})$) 
denote the asymptotically rigid mapping class 
group of $Y_n$ (respectively $Y_n^*$, $Y_n^{\sharp}$) with the above rigid structure. We also suppose that each element $\varphi$ of ${\mathcal M}(Y_n)$ 
(respectively ${\mathcal M}(Y_n^*)$, ${\mathcal M}(Y_n^{\sharp})$) 
is associated with pairs of admissible subsurfaces $C$ and $\varphi(C)$ 
containing {\it the same number} of punctures. 
This additional condition was automatically verified by 
pairs of admissible subsurfaces of the ribbon tree $D^*$ with homeomorphic 
complements.

\vspace{0.2cm}
\noindent
The group ${\mathcal M}(Y_n)$ has a particularly simple form. 
In fact any element of  ${\mathcal M}(Y_n)$ corresponds to a triple  
$((P,Q), r)$, where $P$ and $Q$ are admissible $2n$-gons and 
$r$ is an order $n$ rotation that gives the recipe for identifying the 
boundary arcs of $P$ and $Q$. Moreover, an admissible $2n$-gon $P\subset Y_n$ 
is completely determined by the vector  $v_P\in (\Z_+-\{0\})^n$ recording 
the coordinates of those punctures that lie on the boundary arcs 
of $P$, one coordinate for each ray. 
The cyclic group of rotations $\Z/n\Z$
acts on $\Z^n$ by permuting the coordinates and preserves the 
subgroup  $\Z^{n-1}\subset \Z^n$ of the vectors having the sum of 
their coordinates zero. The map that sends 
the pair $((P,Q),r)$ into $(v_Q-r(v_P), r)\in \Z^n\rtimes \Z/n\Z$ 
induces an isomorphism of  
${\mathcal M}(Y_n)$ onto the subgroup $\Z^{n-1}\rtimes \Z/n\Z$.

\vspace{0.2cm}
\noindent 
One expects ${\mathcal M}(Y_n^*)$ and 
${\mathcal M}(Y_n^{\sharp})$ to be  
extensions of ${\mathcal M}(Y_n)$ by an infinite braid group $B_{\infty}$. 
If ${\mathcal M}(Y_n)$ were abelian then the infinite braid  group 
$B_{\infty}$ would be the commutator  subgroup of the extension group. 
However the semi-direct product $\Z^{n-1}\rtimes \Z/n\Z$ is not direct 
for $n\geq 3$, and hence it is convenient to restrict 
to those  mapping classes in the above groups coming from 
end preserving homeomorphisms. 

\vspace{0.2cm}
\noindent
Consider therefore the 
subgroups  ${\mathcal M}_{\partial}(Y_n)$ (respectively ${\mathcal M}_{\partial}(Y_n^*)$, ${\mathcal M}_{\partial}(Y_n^{\sharp})$) generated by  those homeomorphisms 
which are end preserving i.e. inducing a trivial automorphism of 
the ends of $Y_n$.
Alternatively, the homeomorphisms should send each ray into itself, at 
least outside a large enough compact set. 

\vspace{0.2cm}
\noindent 
It follows from above that ${\mathcal M}_{\partial}(Y_n)$ is isomorphic 
to  $\Z^{n-1}$.

\vspace{0.2cm}
\noindent The groups ${\mathcal M}_{\partial}(Y_n^*)$  are  
isomorphic to the braided Houghton groups considered by \index{braided Houghton group}
Degenhardt (\cite{Deg}) and Dynnikov (\cite{Dy}). 
It is known that these are finitely presented 
groups for all $n\geq 3$. 
 The same result holds for the larger 
related groups ${\mathcal M}_{\partial}(Y_n^{\sharp})$, as it is proved in 
\cite{fu}:

\begin{theorem}(\cite{fu})
The groups ${\mathcal M}(Y_n^{\sharp})$ and 
${\mathcal M}_{\partial}(Y_n^{\sharp})$ are finitely presented for 
$n\geq 3$.  
The commutator subgroup of ${\mathcal M}_{\partial}(Y_n^{\sharp})$ 
is the infinite braid group $B_{\infty}$ in the punctures of $Y_n^{\sharp}$. 
Moreover, the groups ${\mathcal M}_{\partial}(Y_n^{\sharp})$ (and their versions) 
have solvable word problem. 
\end{theorem}
\begin{proof}
Let us outline the proof. 
First, we can express these groups as extensions by the infinite braid group, 
by means of the following exact sequence:  
\[ 1\to B_{\infty}\to {\mathcal M}_{\partial}(Y_n^{\sharp})\to \Z^{n-1}\to 1\]
where $B_{\infty}=\lim_{k\to\infty} B_{kn+1}$ is the limit of the braid 
groups of an exhausting sequence of admissible subsurfaces of $Y_n^{\sharp}$. 
In fact, a mapping class $\varphi\in  {\mathcal M}_{\partial}(Y_n^{\sharp})$ 
sends a  support $2n$-gon into another support $2n$-gon, 
by translating the arc on the half-line 
$l_j$ of $k_j$ units towards the center. 
Since the support hexagons should contain 
the same number of punctures we have $k_1+k_2+\cdots +k_n=0$.  
The map sending $\varphi$ to $(k_1,k_2,\ldots,k_n)$ is a surjection onto 
$\Z^{n-1}$. The claim follows.

\vspace{0.2cm}
\noindent
Let the line $l_j$ be punctured along the points $p_j(i)$ at distance 
$i$ from the origin. Consider the  mapping class of the 
homeomorphism $d_j$ which translates all punctures of the line 
$l_{j} \cup l_{j+1}$ one unit in the counterclockwise direction, 
as in the figure below:

\begin{center}
\includegraphics{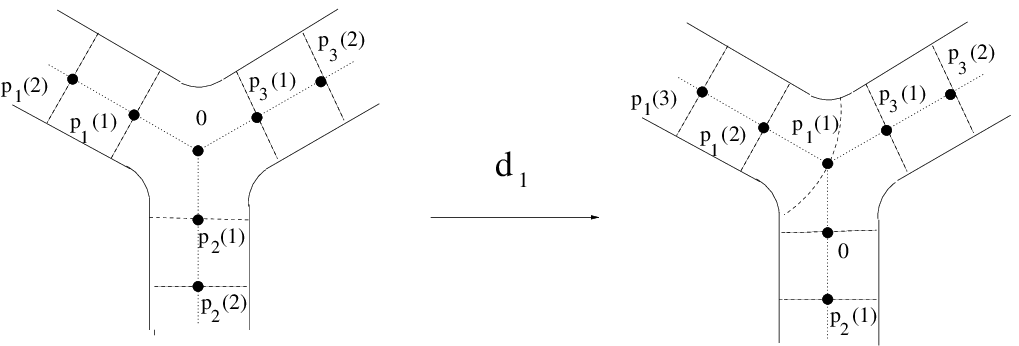}      
%\caption{The mapping class $d_1$}     
\end{center}

\vspace{0.2cm}
\noindent
We use the  convention that the groups  act on the right: thus the 
composition $ab$ denotes $a$ followed by $b$.  
Moreover, the set of subscripts corresponding to the 
rays is $\{1,2,\ldots,n\}$, which is  naturally identified to 
$\Z/n\Z$; let then $<$ denote the cyclic order on  $\Z/n\Z$. 
An explicit presentation is then provided in \cite{fu}:
 
\begin{proposition}[\cite{fu}]
 Set  $u_i=d_{i}d_{i+1}d_{i}^{-1}d_{i+1}^{-1}$. Then the group  
${\mathcal M}_{\partial}(Y_n^{\sharp})$ is generated by the $d_1,d_2,\ldots,d_n$ and admits the presentation: 
\[ d_nd_{n-1}d_{n-2}\cdots d_1=1\]
\[ u_{i_1}u_{i_2}u_{i_3}u_{i_1}=u_{i_2}u_{i_3}u_{i_1}u_{i_2}=
u_{i_3}u_{i_1}u_{i_2}u_{i_3}, {\rm \, if \,}\,  i_1 <i_2 <i_3\]  
\[ d_{i-1}^{-1}u_id_{i-1}=d_{i}u_id_{i}^{-1} \, {\rm for \,\, all }\,\,  i\]
\[ u_iu_ju_i=u_ju_iu_j, \,{\rm for \,\, all }\,\,  i,j \]
\[ d_{i-1}^{-1}u_id_{i-1} u_i  d_{i-1}^{-1}u_id_{i-1}=
u_id_{i-1}^{-1}u_id_{i-1} u_i, \,{\rm for \,\, all }\,\,  i\]
\[ [d_{i}u_id_{i}^{-1}, u_j]= 1, \,{\rm for \,\, all }\,\,  i\neq j \]
\[ [d_{i}u_id_{i}^{-1}, d_j]= 1, \,{\rm for \,\, all }\,\,  i <j< i-1 \]
\[ d_ju_id_j^{-1}=u_iu_ju_j^{-1}, \,{\rm for \,\, all }\,\,  i <j< i-1. \]
\end{proposition}
 
\vspace{0.2cm}
\noindent 
Finally the group ${\mathcal M}_{\partial}(Y_n^{\sharp})$ (and its versions) 
has solvable word problem. In fact, 
for any word $w$ in the generators $d_i$ there exists a 
support of $w$ made of elementary pieces not farther than $|w|+1$ units  
apart from the central $2n$-gon. Then the proof given in \cite{fu-ka2} can be adapted 
to our situation. Observe that we actually use the fact that 
the word problem is solvable in braid groups.  
\end{proof}

\begin{remark}
Let $S_{\infty}$ denote the infinite permutation group of punctures 
of $Y_n^*$ obtained as the direct limit of 
finite permutation groups of punctures in an ascending sequence of 
admissible subsurfaces. 

\vspace{0.2cm}
\noindent 
The Houghton groups $H_n$ considered by Brown (\cite{br1}) are quotients of 
${\mathcal M}_{\partial}(Y_n^*)$ induced from the obvious 
homomorphism $B_{\infty}\to S_{\infty}$ sending braids into the 
associated permutations. This means that we have natural exact sequences 

\[ \begin{array}{ccccccc}
1 \to & B_{\infty} &\to& {\mathcal M}_{\partial}(Y_n^{\sharp})&\to &\Z^{n-1}&\to 1\\
      & \downarrow &   &  \downarrow &  & \downarrow & \\ 
1 \to & S_{\infty} &\to& H_n&\to &\Z^{n-1}&\to 1
\end{array}
\]
\end{remark}

\begin{remark}
The group ${\mathcal M}_{\partial}(Y_2^{\sharp})$ (and its variants) 
is generated by 
two elements, namely $d=d_1=d_2^{-1}$ and $u_1=\sigma_{0p_1(1)}$. 
However,  ${\mathcal M}_{\partial}(Y_2^{\sharp})$ is not finitely presented  
since the commutativity relations coming from the braid group 
are independent, namely we have infinitely many relations  of the form 
$[d^kud^{-k},d^mud^{-m}]=1$, for all integers $m,k$ with $|m-k|\geq 1$. 
Also  ${\mathcal M}_{\partial}(Y_2^{\sharp})$ surjects onto the 
Houghton group $H_2$  which is known to be infinitely presented. 
In some sense  ${\mathcal M}_{\partial}(Y_2^{\sharp})$ is similar 
to the lamplighter groups. 
\end{remark}

\begin{remark}
Since all generators of $B_{\infty}$ are conjugate 
the abelianization of $B_{\infty}$ is $\Z$. The abelianization 
homomorphism $B_{\infty}\to \Z$ induces an extension 
${\mathcal M}_{\partial}(Y_n^{\sharp})^{ab}$ as follows:

\[ \begin{array}{ccccccc}
1 \to & B_{\infty} &\to& {\mathcal M}_{\partial}(Y_n^*)&\to &\Z^{n-1}&\to 1\\
      & \downarrow &   &  \downarrow &  & \downarrow & \\ 
1 \to & \Z &\to& {\mathcal M}_{\partial}(Y_n^*)^{ab} &\to &\Z^{n-1}&\to 1
\end{array}
\]
 
\vspace{0.2cm}
\noindent
For $n=2$ it follows that ${\mathcal M}_{\partial}(Y_2^*)^{ab}$ 
is abelian, generated by the images of $d$ and $u$. In particular, we obtain
 that ${\mathcal M}_{\partial}(Y_2^*)^{ab}\cong H_1({\mathcal M}_{\partial}(Y_2^*))=
\Z^2$.

\vspace{0.2cm}
\noindent
For $n\geq 3$ the group ${\mathcal M}_{\partial}(Y_n^*)^{ab}$ is a nontrivial 
(non-abelian) extension of $\Z^{n-1}$ by $\Z$. 
\end{remark}

\begin{remark}
Given three rays in the binary tree we can associate an embedding 
of $Y_3^{*}$ into $D^*$  that induces injective compatible  homomorphisms
$\Z^2\rtimes \Z/3\Z\to T$ and ${\mathcal M}(Y_3^*)\to T^*$.  
\end{remark}

\subsection{Infinite genus surfaces and mapping class groups}

In \cite{fu-ka1}, we proved that the Teichm\" uller tower of groupoids in  genus zero may be embodied 
in a very concrete group, the universal mapping class group in genus zero $\B$. Not only does 
$\B$ contain the tower, but its remarkable property of being finitely presented also realizes, in the category of groups, the analogous property for the tower of groupoids. 

\vspace{0.1cm}\noindent 
We give a partial solution of the problem of realizing 
the higher genus Teichm\"uller tower in the category of groups. 
We construct a {\it finitely generated} group  ${\cal M}$ 
that contains all the 
pure mapping class groups $P{\cal M}(g,n)$ (with $n>0$). 
The solution is partial as  ${\cal M}$ does not contain all 
the (non-pure) mapping class groups ${\cal M}(g,n)$.

\begin{definition}[following \cite{fu-ka3}]
Let ${\cal T}_{\infty}$ be the graph obtained from the planar dyadic tree by attaching a loop on each edge, based on its middle point. The three-dimensional thickening of ${\cal T}_{\infty}$ is a handlebody, whose boundary is an orientable surface ${\mathscr S}_{\infty}$ 
of infinite genus.

\begin{enumerate}
\item An orientation-preserving homeomorphism $g$ of ${\mathscr S}_{\infty}$ is {\em asymptotically rigid} \index{asymptotically rigid homeomorphism}
if there exist two connected subsurfaces $S_0$ and $S_1$ of ${\mathscr S}_{\infty}$ such that 
 $g$ induces, by restriction on each connected component of ${\cal T}_{\infty}\cap ({\mathscr S}_{\infty}\setminus S_0)$, an isomorphism 
 (of graphs) onto a connected component of ${\cal T}_{\infty}\cap ({\mathscr S}_{\infty}\setminus S_1)$, which respects the local orientation of the edges (coming from the planarity of the dyadic tree).

\item  The {\em asymptotically rigid mapping class 
group of infinite genus} \index{asymptotically rigid mapping class 
group of infinite genus}
${\cal M}$ is the group of mapping classes of isotopies of asymptotically rigid homeomorphisms of the surface $\mathscr{S}_{\infty}$.

\end{enumerate}

%\begin{figure}
\begin{center}

\includegraphics[scale=0.8]{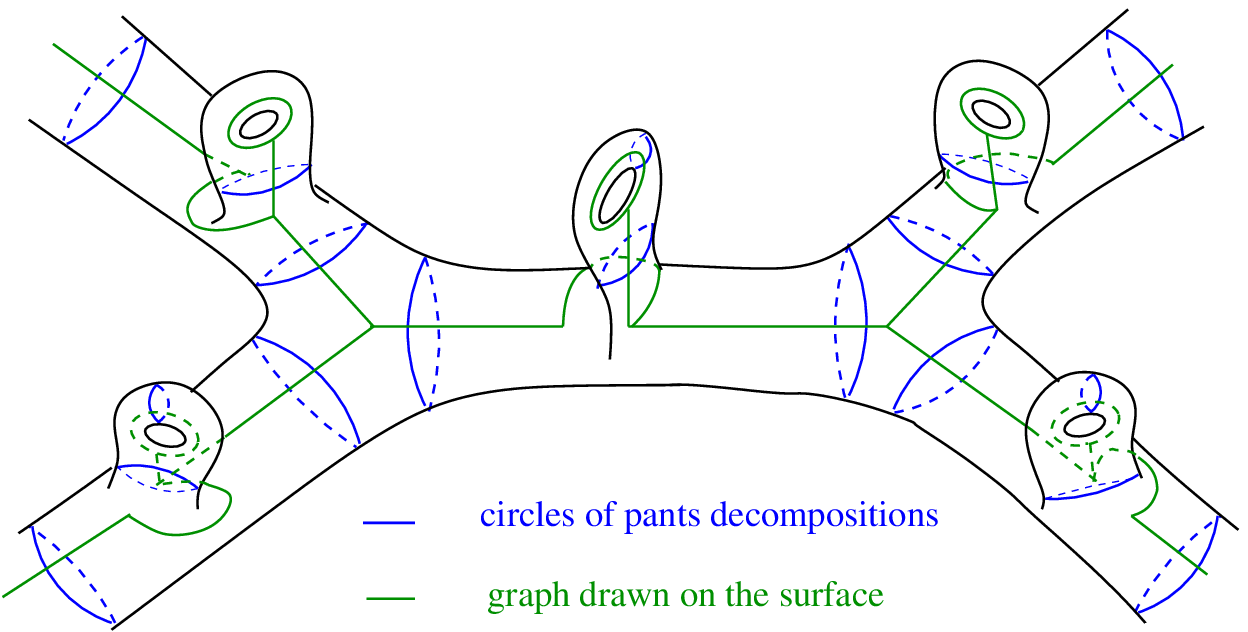}

\end{center}

%\end{figure}

\end{definition}

\vspace{0.1cm}\noindent 
Forgetting the loops of ${\cal T}_{\infty}$, one obtains a morphism from ${\cal M}$ to ${\rm Homeo}(\partial {\cal T})$, whose image is Thompson's group $V$.
It follows easily, as for ${\cal B}$, that ${\cal M}$ is an extension of Thompson's group $V$ by the 
pure mapping class group of the surface:
$$1\rightarrow P{\cal M}\longrightarrow {\cal M}\longrightarrow V\rightarrow 1$$

\vspace{0.1cm}\noindent
The pure mapping class group $P{\cal M}$ is countable, and generated by the Dehn twists 
around the closed simple curves embedded into  ${\cal S}_{\infty}$, and is not finitely generated.

\vspace{0.1cm}\noindent 
The first result of \cite{fu-ka3} is:

\begin{theorem}[ \cite{fu-ka3} Th. 1.1]
 The group ${\cal M}$  is finitely generated.
\end{theorem}

\vspace{0.1cm}\noindent 
The proof is based on a convenient presentation due to  S. Gervais of the mapping class groups $P{\cal M}(g,n)$ (\cite{ge}), from which we deduce that
$P{\cal M}$ is generated by a set of Dehn twists around some curves of ${\mathscr{S}}_{\infty}$ 
forming a certain set $\mathscr{F}$. By collapsing the handles of the surface ${\mathscr{S}}_{\infty}$ onto punctures, 
those curves may be identified with paths on $D^*$ -- the planar surface we have introduced for $T^*$ -- joining two punctures. 
Exploiting the action of $T^*$ on the homotopy classes of those paths, one shows that almost all of them are equivalent modulo $T^*$. 
It is possible to ``lift'' this result to ${\mathscr{S}}_{\infty}$,  and this enables us to prove that 
the family $\mathscr{F}$ is finite modulo the action of a finitely  generated subgroup of ${\cal M}$. It follows 
easily that ${\cal M}$ is generated by a finite number of Dehn twists and by the lifts to ${\cal M}$ of generators 
of $V$.

\vspace{0.1cm}\noindent 
Using the Lyndon-Hochschild-Serre spectral sequence associated to the short exact sequence 
\[1\rightarrow P{\cal M}\longrightarrow {\cal M}\longrightarrow V\rightarrow 1\]
and a theorem of Brown asserting that the rational homology of $V$ is trivial (\cite{br2}), we prove that the group ${\cal M}$ has the same rational homology as 
$P{\cal M}$. Describing $P{\cal M}$ as a direct limit of groups $P{\cal M}(g,n)$, 
it follows from Harer's stability theorem  (\cite{ha}) that the homology of $P{\cal M}$ is the stable homology 
of the mapping class group.

\begin{theorem}[\cite{fu-ka3}, Th. 1.2]
The rational homology of the group ${\cal M}$ is isomorphic to 
the rational stable homology of the mapping class group.

\end{theorem}

\vspace{0.1cm}\noindent 
Our result proves, therefore, that there exists a finitely generated 
group whose rational homology is isomorphic to that of $BU$, the universal classifying space 
of complex fibre bundles. We compute that  $H^2({\cal M},\Z)=\Z$, and show that the generator 
of this second cohomology group may be identified with 
the first universal Chern class.

\vspace{0.1cm}\noindent 
\noindent As a corollary  of the argument of the proof 
the group $\M$ is perfect and 
$H_2(\M,\Z)=\Z$. For a reason that will become clear in what follows, the generator of $H^2(\M,\Z)\cong \Z$ is called 
{\em the first universal Chern class} of $\M$, and is denoted $c_1(\M)$. 
\index{universal Chern class}

\vspace{0.1cm}\noindent 
\noindent Let ${\cal M}_g$ be the mapping class group of a closed surface $\Sigma_g$ of genus $g$. We show that  the standard representation 
$\rho_g: {\cal M}_g\rightarrow {\rm Sp}(2g,\Z)$ in the symplectic group,  
deduced from the 
action of ${\cal M}_g$ on  $H_1(\Sigma_g,\Z)$, extends to 
the infinite genus case, by replacing the finite dimensional setting by concepts of Hilbert analysis. In particular, a key role is played by Shale's {\it restricted 
  symplectic group} ${\rm Sp_{res}}({\cal H}_r)$ on the real Hilbert space ${\cal 
  H}_r$ generated by the homology classes of non-separating closed curves of ${\cal 
  S}_{\infty}$. We then have:  
 
\begin{theorem}[\cite{fu-ka3}]\label{metab}
The action of $\M$ on $H_1({\cal S}_{\infty},\Z)$ induces a representation $\rho:\M\rightarrow {\rm Sp_{res}}({\cal H}_r)$. 
\end{theorem}  

\vspace{0.2cm} \noindent
 The generator $c_1$ of $H^2({\cal M}_g,\Z)$ is called the 
first Chern class, since it may be obtained as follows 
(see, e.g., \cite{mo}). The group ${\rm Sp}(2g,\Z)$ is 
contained in the symplectic group ${\rm Sp}(2g,\R)$, whose maximal compact subgroup 
is the unitary group $U(g)$. Thus, the first Chern class may be viewed in 
$H^2(B {\rm Sp}(2g,\R),\Z)$. It can be  first pulled-back on 
$H^2(B{\rm Sp}(2g,\R)^{\delta},\Z)=H^2({\rm Sp}(2g,\R),\Z)$ and 
then on $H^2({\cal   M}_g,\Z)$ via $\rho_g$. 
This is the generator of $H^2({\cal M}_g,\Z)$. 
Here $B{\rm Sp}(2g,\R)^{\delta}$ 
denotes the classifying space of the group ${\rm Sp}(2g,\R)$ endowed with  
the discrete topology. 

\vspace{0.1cm}\noindent 
 \noindent The restricted symplectic group \index{restricted symplectic group}
${\rm Sp_{res}}({\cal H}_r)$ has a well-known 
2-cocycle, which measures the 
projectivity of the {\it Berezin-Segal-Shale-Weil metaplectic representation} 
in the bosonic Fock space (see \cite{ne}, Chapter 6 and Notes 
p. 171). Unlike the finite dimensional case, this cocycle is not directly related 
to the topology of ${\rm Sp_{res}}({\cal H}_r)$, 
since the latter is a contractible Banach-Lie group. 
However, ${\rm Sp_{res}}({\cal H}_r)$  admits an embedding into the 
restricted linear group of Pressley-Segal ${\rm {\rm GL}^0_{res}}({\cal H})$ (see 
\cite{pr-se}), where ${\cal 
  H}$ is the complexification of ${\cal H}_r$, which possesses a cohomology class of degree 2: the Pressley-Segal 
class $PS\in H^2({\rm {\rm GL}^0_{res}}({\cal H}),\C^*)$. The group ${\rm {\rm GL}^0_{res}}({\cal H})$ 
is a homotopic model of the classifying space $BU$, where 
$U=\displaystyle{\lim_{n\to \infty} U(n,\C)}$, and the class $PS$ does correspond to the 
universal first Chern class. Its restriction on ${\rm Sp_{res}}({\cal H}_r)$ is 
closely related to the Berezin-Segal-Shale-Weil cocycle, and reveals the topological 
origin of the latter. Via the composition of morphisms 
$$\M\longrightarrow {\rm Sp_{res}}({\cal H}_r) \hookrightarrow 
{\rm {\rm GL}^0_{res}}({\cal 
  H}),$$  we then derive from 
$PS$ an {\it integral} cohomology class on $\M$: 
 
\begin{theorem}[\cite{fu-ka3}]\label{chern} 
The Pressley-Segal class $PS\in H^2({\rm {\rm GL}^0_{res}}({\cal H}),\C^*)$ induces the first universal Chern 
class $c_1(\M)\in H^2({\cal M},\Z)$. 
\end{theorem}

\section{ Cosimplicial extensions for  the Thompson group $V$}
 
Various extensions of Thompson's group $V$ have been encountered in what precedes. 
 
\begin{enumerate}
\item The extension $1\rightarrow P_{\infty}\longrightarrow BV\longrightarrow V\rightarrow 1$, where $BV$ is 
the braided Thompson group of Brin-Dehornoy (\cite{bri1}, \cite{bri2}, \cite{de1}, \cite{de2}).
\item The extension $1\rightarrow K^*_{\infty}\longrightarrow \B\longrightarrow V\rightarrow 1$, where 
$\B$ is the universal mapping class group of (\cite{fu-ka1}).
\item The extension  $1\rightarrow P{\cal M}\longrightarrow {\cal M}\longrightarrow V\rightarrow 1$, where
${\cal M}$ is the asymptotically rigid mapping class group of infinite genus (\cite{fu-ka3}).
\end{enumerate}
Each one appears in a specific context. It turns out that it is possible to recover all of them by means of a very general and algebraic formalism. More precisely, one may describe 
a functorial construction which produces this type of extensions for $V$. It is defined on a category 
whose objects are called  {\em cosimplicial $\mathfrak{S}$-extensions}, \index{cosimplicial extensions}
where the letter $\mathfrak{S}$ stands for the  ``symmetric group''. 
This formalism, which is inspired from a non-simplicial construction in 
\cite{FiLo}, seems to be 
quite useful when the appropriate language of the problem is algebraic. This is the case when one wishes to define a convenient profinite completion 
of the groups $BV$ or $\B$.

\subsection{Strand doubling maps}

For $n\geq1$, let $S_n$ denote the set $\{1,\ldots,n\}$, and $\mathfrak S_{n}$ denote the symmetric group acting on $S_n$. 

\begin{definition}
For each integer $n\geq 1$ and each $i=1,\ldots,n$, the $i^{th}$ strand doubling map $\partial^{i}_{n}:\mathfrak{S}_{n}\rightarrow \mathfrak{S}_{n+1}$ is defined as follows. For any $\sigma\in \mathfrak{S}_{n}$, $\partial^{i}_{n}(\sigma)$ is the natural extension of $\sigma$ as a permutation
of $S_{n+1}$ when one  simultaneously duplicates $i$ at the source and the $\sigma(i)$ at the target. More precisely, let $\Omega^{i}_n$ be the set ${\bf n}\setminus\{i\}\cup \{i_l,i_r\}$, whose elements are those of
$S_{n}$ except $i$, which is replaced by two elements, $i_{l}$ and $i_{r}$ (where the index $l$ stands for ``left", and $r$ for ``right"). It is ordered by\\
$$1<2<\ldots<i-1<i_l<i_r<i+1<\ldots<n.$$ 
If $j=\sigma(i)$, let $\tau:\Omega^{i}_n\rightarrow \Omega^{j}_n$ be the bijection which is the natural extension of $\sigma$:
$\tau(k)=k$ if $k\notin \{i_{l}, i_{r}\}$, $\tau(i_{l})=j_{l}$ and $\tau(i_{r})=j_{r}$.

\vspace{0.1cm}\noindent 
The permutation  $\partial^{i}_{n}(\sigma)\in \mathfrak{S}_{n+1}$ is the bijection $f^j_n\circ \tau\circ {(f^{i}_n)}^{-1}$, where $f^{k}_n:\Omega_{n}^k\rightarrow S_{n+1}$, for $k=i$ or $j$, is the unique
isomorphism between the ordered sets $\Omega^{k}_n$ and $ S_{n+1}$.\\
\end{definition}

\begin{remark} The maps $\pa^{i}_{n}$ are not homomorphisms in the category of groups. Nevertheless, they verify the coherence relations:
$$\pa^{i}_{n}(\s\circ \tau)=\pa^{\tau(i)}_{n}(\s)\pa^{i}_{n}(\tau).$$
\end{remark}

\begin{example} Let $\s_{i}\in \mathfrak{S}_{n}$ be the transposition $(i, i+1)$. Then $\partial^{i}_{n}(\s_{i})=\s_{i} \s_{i+1}$, 
$\partial^{i}_{n}(\s_{i-1})=\s_{i} \s_{i-1}$, $\partial^{i}_{n}(\s_{j})=\s_{j+1}$ if $i<j$ and
 $\partial^{i}_{n}(\s_{j})=\s_{j}$ if $j<i-1$.
 \end{example}
 
 \subsection{Cosimplicial $\mathfrak{S}$-extensions}

\begin{definition}\label{cosimpl}
A cosimplicial $\mathfrak{S}$-extension is a family of group extensions
$$1\rightarrow K_{n}\longrightarrow G_{n}\longrightarrow \mathfrak{S}_{n}\rightarrow 1$$
indexed by $n\in \N^*$, such that: 
\begin{enumerate}
\item For all $i=1,\ldots, n$, the strand doubling map $\partial^{i}_{n}:\mathfrak{S}_{n}\rightarrow \mathfrak{S}_{n+1}$ admits a lift, still denoted
$\partial^{i}_{n}: G_{n}\rightarrow G_{n+1}$, which verifies, for all $g,h\in G_n$:
$$\pa^{i}_{n} (gh)=\pa^{\bar{h}(i)}_{n}(g )\pa^{i}_{n}(h),$$
where $\bar{h}$ denotes the image of $h\in G_{n}$ in $\mathfrak{S}_{n}$.

\vspace{0.1cm}\noindent 
In particular, each $\pa^{i}_{n}$ restricts to a morphism of groups $\pa^{i}_{n}:K_{n}\rightarrow K_{n+1}$.

\item There exist morphisms (called codegeneracy morphisms) $\varepsilon^{i}_{n}: K_n\rightarrow K_{n-1}$, for
$i=1,\ldots, n$, such that $(K_{n}, \pa^{i}, \varepsilon^{i},n\geq 1)$ is a cosimplicial group.

\end{enumerate}
\end{definition}

\begin{remark}
Since $\varepsilon^{i}_{n+1}\circ \partial^{i}_{n}=id$, the morphisms $\partial ^{i}_{n}:K_{n}\rightarrow K_{n+1}$
are injective.
\end{remark}

\begin{figure}
\begin{center}
\includegraphics[scale=0.8]{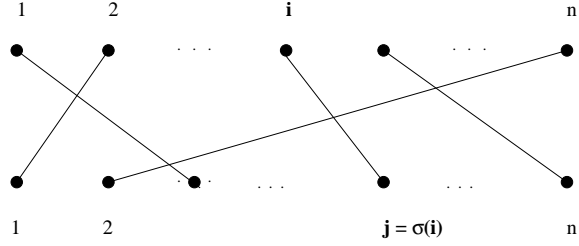}
\caption{Diagram of strands representing a permutation $\sigma$}
\end{center}
\end{figure}

\begin{figure}

\begin{center}
\includegraphics[scale=0.8]{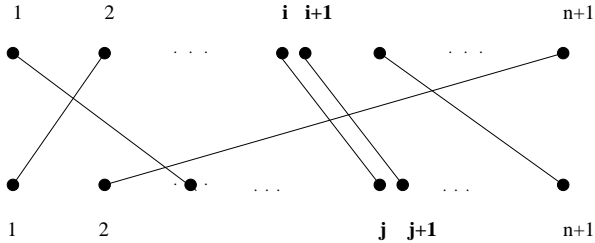}
\caption{Diagram of strands representing the permutation $\partial^{i}_{n}(\sigma)$}
\end{center}
\end{figure}

 \subsection{Dyadic trees and the functor $\bf{K}$}

\begin{definition}
\begin{enumerate}
\item
Let ${\cal T}_0$ be the planar rooted dyadic tree, whose vertices except the root are 3-valent. Let also ${\cal T}$ be the planar unrooted regular tree, whose vertices are all 3-valent. One may view ${\cal T}_0$ as a subtree of ${\cal T}$. The edge $e$ of ${\cal T}$ which is not contained in ${\cal T}_0$ but is incident to the root of ${\cal T}_0$ is called the reference edge of ${\cal T}$.      
\item A planar rooted finite dyadic  $n$-tree is a finite subtree of ${\cal T}_0$ which contains the root, has $n$ leaves, and whose vertices, other than the root and the leaves, are 3-valent. The leaves of such a tree are canonically labelled by $1,\ldots,n$ from left to right, according to the given orientation of the plane.
\item  A planar unrooted finite dyadic  $n$-tree is a finite subtree of ${\cal T}$ which contains the reference 
edge  $e$, has $n$ leaves, and whose vertices, other than leaves, are 3-valent. The leaves of such a tree are canonically labelled by $1,\ldots,n$ from left to right, according to the given orientation of the plane, assuming
that its leftmost leaf belonging to ${\cal T}_0$ is  labelled 1.\\
In short, a planar rooted or unrooted finite dyadic $n$-tree will be called a (rooted or unrooted) labelled tree.
\item If $\tau$ is a rooted or unrooted finite $n$-tree, 
then $\mid \tau\mid$ will denote the number of its leaves.

\end{enumerate}

\end{definition}

\begin{definition}[Category of trees]

 Let $\bf T_0$ (resp. $\bf T$) be the small category of rooted (resp. unrooted) finite dyadic labelled trees defined as follows.\\
\indent $\bullet$ Its objects are the rooted (resp. unrooted) finite dyadic labelled trees.\\
\indent $\bullet$ Let $\tau$ be an object of $\bf T_0$ (resp. $\bf T$) with $n$ leaves.  If $1\leq i\leq n$,
 let $\partial^{i}_{n}(\tau)$ be the dyadic tree obtained from $\tau$ by grafting two edges at its $i$-th leaf (such a pair of grafted edges is
 sometimes called a carret). Since $\partial^{i}_{n}(\tau)$ is planar, it inherits a canonical labelling.
 The tree $\partial^{i}_{n}(\tau)$ is called a simple expansion of $\tau$.\\
  Let now $\tau$ and $\tau'$ be two rooted (resp. unrooted) finite dyadic labelled trees. One says that $\tau'$ is an expansion of $\tau$ if  there exits a chain of
  simple expansions connecting $\tau$ to $\tau'$. This means that $n=\mid \tau\mid \leq m=\mid \tau' \mid$, and either $\tau=\tau'$, or
$\tau'=\partial^{i_{k+1}}_{n+k}\cdots \partial^{i_2}_{n+1} \partial^{i_1}_{n}(\tau)$, for some $i_{1},\ldots, i_{k+1}$,  with $m=n+k+1$.\\
By definition,  ${\rm Hom}(\tau,\tau')$ is nonempty if and only if $\tau'$ is an expansion of $\tau$, in which case it has a single element. Therefore, the set
of all morphisms is the set of pairs $(\tau,\tau')$, where $\tau'$ is an expansion of $\tau$.

\end{definition}

\begin{remark}\label{presque}

\begin{enumerate}

\item $Ob(\bf T_0)$ and $Ob(\bf T)$ are partially ordered sets by setting $\tau\leq \tau'$ if and only if $\tau'$ is an expansion of $\tau$. They are directed ordered sets, since any two trees have a 
common expansion.

\item The categories $\bf T$ and $\bf T_0$ are ``almost'' isomorphic in the following sense. Let ${\bf \tau_3}\in Ob({\bf T})$ be the tripod
 whose three edges are incident to the root of ${\cal T}_0$. Let $\tau_0\in Ob({\bf T_0})$ be one of the two rooted 3-trees in $\bf T_{0}$. Denote by $\bf T_{\tau\geq \tau_3 }$ the full sub-category of $\bf T$ whose objects are the $n$-trees which contain $\tau_3$, and denote similarly by $\bf T_{\tau\geq \tau_0 }$ the full sub-category of $\bf T_0$ whose objects are the $n$-trees which contain $\tau_0$. Plainly, the sub-categories $\bf T_{\tau\geq \tau_3 }$ and ${\bf T_0}_{\tau\geq \tau_0 }$ are isomorphic.
\end{enumerate}

\end{remark}

\begin{proposition-definition}
Let $\bf G$ be the category of groups, and ${\bf T}_*$ stand for the category ${\bf T_0}$ or ${\bf T}$.
Let  $(K_{n},\pa^{i}_{n}, \varepsilon^{i}_{n},n\geq 1)$ be a cosimplicial group.  
The functor $\bf K:\bf T_*\rightarrow \bf G$ is defined as follows:
\begin{itemize}
\item Let $\tau\in Ob(\bf T_*)$. Set ${\bf K}(\tau)=K_{\mid \tau\mid}$.\\
\item Let $\varphi\in {\rm Hom}(\tau,\tau')$. Then  $n=\mid \tau\mid \leq m=\mid \tau' \mid$.
Either $\tau=\tau'$, in which case one sets ${\bf K}(\varphi)=1$, the neutral element of $K_{n}$. 
Or $\tau'=\partial^{i_{k+1}}_{n+k}\cdots \partial^{i_2}_{n+1} \partial^{i_1}_{n}(\tau)$, for some $i_{1},\ldots, i_{k+1}$,
 with $m=n+k+1$. In that case, one sets ${\bf K}(\varphi)=\partial^{i_{k+1}}_{n+k}\circ\cdots \circ\partial^{i_2}_{n+1} \circ\partial^{i_1}_{n}\in {\rm Hom} (K_{n}, K_{m})$,
 and this does not depend on the choice of simple expansions which connect $\tau$ to $\tau'$.
The functor $\bf K$ yields a  group ${\bf K}_{\infty}[\bf T_*]$ which is the colimit
   $$K_{\infty}[\bf T_*]=\displaystyle \lim_{\stackrel{\longrightarrow}{\tau \in {\bf T_*}}} {\bf K}(\tau).$$
\end{itemize}

\vspace{0.1cm}\noindent 
The groups $K_{\infty}[\bf T]$ and $K_{\infty}[\bf T_0]$ are isomorphic. We denote them by the same symbol $K_{\infty}$.
\end{proposition-definition}

\begin{proof}
Two chains of simple expansions between two trees $\tau$ and $\tau'$ such that $\tau\leq \tau'$
may only differ by a repetition of the relation $\partial^j_{N+1}\circ \partial^{i}_{N}(\tau)=
\partial^{i+1}_{N+1}\circ \partial ^{j}_{N}(\tau)$, where $\tau$ is a tree such that $\mid \tau \mid=N$,  and  $1\leq j<i\leq N$. Since
$(K_{n}, \pa^{i}_{n}, \varepsilon^{i}_{n},n\geq 1)$ is a cosimplicial group, one has $\partial^j_{N+1}\circ \partial^{i}_{N}=
\partial^{i+1}_{N+1}\circ \partial ^{j}_{N}$ in ${\rm Hom}(K_{N},K_{N+1})$. This proves
that ${\bf K}(\varphi)$ is well defined.\\
The isomorphism between $K_{\infty}[\bf T]$ and $K_{\infty}[\bf T_0]$ is a consequence of the second remark of \ref{presque}. 
\end{proof}

\subsection{Extensions of Thompson's group $V$}

Let $VExt$ be the category of extensions of Thompson's group $V$. An object is therefore a short exact sequence
of groups
$1\rightarrow K\longrightarrow G\longrightarrow V\rightarrow 1.$ 
 A morphism between to
objects is a commutative diagram of short exact sequences.

\begin{proposition}\label{construction}
There exists a functor from the category $\mathfrak{S}Ext$ to the category $VExt$, 
which maps a cosimplicial $\mathfrak{S}$-extension $(1\rightarrow K_{n}\longrightarrow G_{n}\longrightarrow \mathfrak{S}_{n}\rightarrow 1)$
to an extension of Thompson's group $V$ by $K_{\infty}$:\\

$$1\rightarrow K_{\infty}\longrightarrow G[V]\longrightarrow V\rightarrow 1$$
which is split over Thompson's group $F$.

\vspace{0.1cm}\noindent 
In particular, the correspondence which associates the group $G[V]$ to the cosimplicial $\mathfrak{S}$-extension
is a functor from $\mathfrak{S}Ext$ to the category of groups.

\end{proposition}

\begin{proof}
In the following, ${\bf T}_{*}$ will equally denote the category ${\bf T}_{0}$ or the category ${\bf T}$ . The isomorphism
classes of groups or short exact sequence of groups constructed below will not depend on the choice for ${\bf T}_{*}$.

\vspace{0.1cm}\noindent 
Let $\mathscr{S}(G)$ be the following set, whose elements are called $G$-symbols. A $G$-symbol
is a triple $(\tau_{1},\tau_{0}, g)$, where $\tau_{0}$ and $\tau_{1}$ are two objects of 
$\bf T$ with the same number of leaves, say $n\geq 1$, and $g$ is an element of the group $G_{n}$.
The integer $n$ is called the level of the $G$-symbol  $(\tau_{1},\tau_{0}, g)$.

\vspace{0.1cm}\noindent 
One defines a set of binary relations denoted $\sim_{n,i}$, for each integer $n\geq 1$ and $1\leq i\leq n$. By definition,
 $\sim_{n,i}$ relates a symbol $(\tau_{1},\tau_{0}, g)$ of level $n$ to the symbol 
 $(\partial^{\bar{g}(i)}_{n}(\tau_{1}),\partial^{i}_{n}(\tau_{0}), \partial^{i}_{n}(g))$ of level $n+1$.

\vspace{0.1cm}\noindent 
 Let now $\mathscr{R}$ be the equivalence relation generated by the set of relations $\sim_{n,i}$ on  $\mathscr{S}(G)$.
 An equivalence class is denoted $[\tau_{1},\tau_{0}, g]$.
 
\vspace{0.1cm}\noindent 
 The group $G[V]$ is defined as follows. Its elements are the equivalence classes of symbols
 $\mathscr{S}(G)/\mathscr{R}$. Let $[\tau_{1},\tau_{0},h]$ and $[\tau_{2},\tau_{1}',g]$
 be two elements. At the price of replacing both symbols by equivalent ones, one may assume that
 $\tau_{1}'=\tau_{1}$. Set $$[\tau_{2},\tau_{1},g]\cdot [\tau_{1},\tau_{0},h]=[\tau_{2},\tau_{0}, gh].$$
 The point is that the above  definition of a product of two elements of  $\mathscr{S}(G)/\mathscr{R}$
 only depends on the equivalence classes, not on  the choice of the symbols.\\
 The neutral element is the class of any symbol $(\tau,\tau, 1_{n})$, where $\tau$ is any tree, $n$ is its level, and $1_{n}\in G_{n}$ is
 the neutral element of $G_{n}$.

\vspace{0.1cm}\noindent  
 For each  $\tau\in Ob({\bf T}_{*})$, there is a morphism $\iota_{\tau}:{\bf K}(\tau)\rightarrow G[V]$, $k\mapsto [\tau,\tau, k]$. If
 $\tau\leq \tau'$ and $\varphi$ is the unique morphism from $\tau$ to $\tau'$, then $\iota_{\tau}=\iota_{\tau'}\circ {\bf K}(\varphi)$. This implies
 the existence of a morphism $\iota_{\infty}:K_{\infty}\rightarrow G[V]$ induced by the $\iota_{\tau}$'s. Since the morphisms 
 $\partial^{i}_{n}$ are injective, so is the morphism $\iota_{\infty}$.

 \vspace{0.1cm}\noindent 
When $G_{n}=\mathfrak{S}_{n}$ (hence $K_{n}=\{1\}$), the group $G[V]$ is denoted $V$, and is the
Thompson group acting on the Cantor set. The subgroup of $V$ whose elements are of the form $[\tau_{1},\tau_{0},1_{n}]$, for any pair of $n$-trees $(\tau_{1},\tau_{0})$, and any $n\geq 1$, is
Thompson's group $F$.

\vspace{0.1cm}\noindent 
The morphism $G[V]\rightarrow V$ is defined by $[\tau_{1},\tau_{0},g]\rightarrow [\tau_{1},\tau_{0},\bar{g}]$. It is surjective,
and its kernel is $K_{\infty}$. The splitting over Thompson's group $F$ is the morphism $[\tau_{1},\tau_{0},1_{n}]\in F\mapsto [\tau_{1},\tau_{0},1_{n}]\in G[V]$
(where $1_{n}$ denotes the neutral element of $\mathfrak{S}_{n}$ in the left symbol, and
the neutral element of $G_{n}$ in the right symbol).
The functoriality of this construction can be easily checked.

\end{proof}

\subsection{Main examples}

%\subsubsection*{Thompson's group  $V$}

\subsubsection*{The braided Thompson group $BV$}
\index{braided Thompson group $BV$}

The family of extensions  $1\rightarrow P_n\longrightarrow B_n \longrightarrow {\mathfrak S}_n \rightarrow 1$, $n\geq 1$, 
is a cosimplicial $\mathfrak S$-extension, with $\partial_{n}^{i}$ the obvious geometric strand doubling map, and $\varepsilon_{n}^{i}: P_{n}\rightarrow P_{n-1}$ the morphism obtained by deleting the $i^{th}$ braid.

\begin{figure}
\begin{center}
\includegraphics[scale=0.8]{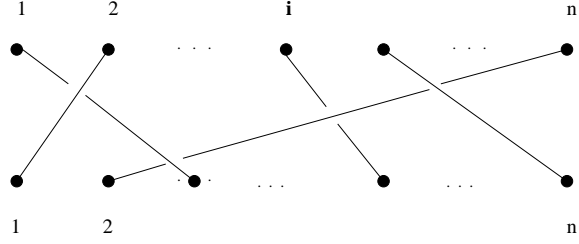}
\caption{Diagram of strands representing a braid $\sigma$}
\end{center}
\end{figure}

\begin{figure}
\begin{center}
\includegraphics[scale=0.8]{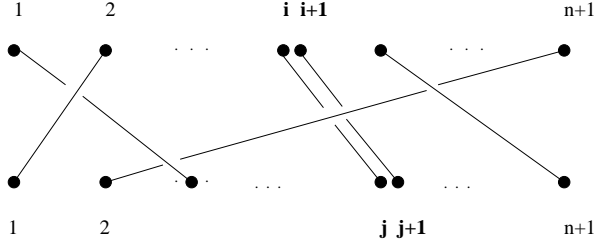}
\caption{Diagram of strands representing the braid $\partial^{i}_{n}(\sigma)$}
\end{center}
\end{figure}

\vspace{0.1cm}\noindent 
The group  $B[V]$ is the braided Thompson group $BV$ discovered independently by Brin and Dehornoy 
(\cite{bri1}, \cite{bri2},\cite{de1}, \cite{de2}). 

\subsubsection*{The universal mapping class group in genus zero ${\cal B}$}

\index{universal mapping class group in genus zero}
In \cite{fu-ka1} we construct the group ${\cal B}$ as a mapping class group of a sphere minus 
a Cantor set. The surface $\mathscr{S}_{0,\infty}$ is the boundary of the 3-dimensional thickening
of a regular (unrooted) dyadic tree. The definition of $\B$ given in $\cite{fu-ka1}$ is therefore 
completely topological. We wish to give an equivalent one using the  formalism
of cosimplicial $\mathfrak{S}$-extensions.\\

\subsubsection*{Mapping class groups $M^*(0,n)$}

\begin{definition}\label{M}
\begin{enumerate}
\item Let $S$ be an $n$-holed sphere, that is, a sphere minus $n$
  disjoint embedded open disks.  Its {\em pure mapping class group} $K^*(S)$  is the group of isotopy classes of
  orientation-preserving homeomorphisms of $S$. The homeomorphisms and the isotopies are assumed
  to fix pointwise the $n$ boundary circles.
\item Suppose we have chosen two  base points on each boundary circle of
  $S$, $p_{+}$ and $p_{-}$. The full mapping class group $M^*(S)$ is the group of isotopy classes of
  orientation-preserving homeomorphisms of the holed sphere $S$ which
  permute the boundary circles and the set of base points, preserving their signs. The
  isotopies are only assumed to fix the base points.\\
If the set of boundary components is labelled by $\{1,\ldots,n\}$, then $K^*(S)$ and $M^*(S)$ are related by a short exact sequence
$$1\rightarrow
  K^*(S)\longrightarrow M^*(S) \longrightarrow \mathfrak{S}_n\rightarrow
  1.$$

\end{enumerate}
 \end{definition}

 \begin{proposition}
 The family of group extensions  $1\rightarrow
  K^*(0,n)\longrightarrow M^*(0,n) \longrightarrow \mathfrak{S}_n\rightarrow
  1$, $n\geq 1$, forms a cosimplicial $\mathfrak{S}$-extension.
\end{proposition}

\vspace{0.1cm}\noindent 
The strand doubling map $\partial^{i}_{n}: \sigma\in M^*(0,n)\mapsto \partial^{i}_{n}(\sigma)\in M^*(0,n+1)$ is deduced from the topological operation consisting in 
gluing a pair of pants $P_{i}$ (resp. $P_{\bar{\sigma}(i)}$) along the $i^{th}$ (resp. $\sigma(i)^{th}$) boundary circle of the $n$-holed  reference 
surface $\Sigma_{0,n}$, then defining the natural extension of $\sigma$ as a 
(mapping class of) homeomorphism  $\Sigma_{0,n}\cup P_{i}\rightarrow \Sigma_{0,n}\cup P_{\sigma(i)}$, and finally identifying  $\Sigma_{0,n}\cup P_{i}$ and  $\Sigma_{0,n}\cup P_{\sigma(i)}$ to $\Sigma_{0,n+1}$ to obtain the expected 
$\partial^{i}_{n}(\sigma)$.

\begin{figure}
\begin{center}
\includegraphics{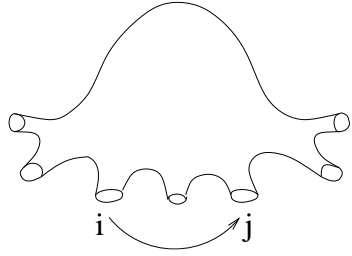}
\caption{The mapping class $\sigma\in M^*(0,n)$ permuting the $i^{th}$ and $j^{th}$ boundary circles}
\end{center}
\end{figure}

\begin{figure}
\begin{center}
\includegraphics[scale=0.9]{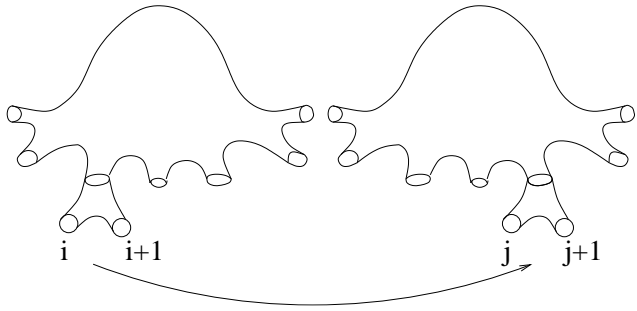}
\caption{The mapping class $\partial^{i}_{n}(\sigma)\in M^*(0,n+1)$}
\end{center}
\end{figure}

\vspace{0.1cm}\noindent 
The codegeneracy morphisms $\varepsilon^{i}:K^*(0,n)\rightarrow K^*(0,n-1)$ are
induced by the topological operation consisting in filling up the $i^{th}$ hole of $\Sigma_{0,n}$ with
a disk, and identifying the resulting surface to $\Sigma_{0,n-1}$.

\begin{proposition}
The  cosimplicial $\mathfrak{S}$-extension $(1\rightarrow
  K^*(0,n)\longrightarrow M^*(0,n) \longrightarrow \mathfrak{S}_n\rightarrow
  1, n\geq 1)$, yields the extension of Thompson's group $V$

$$1\rightarrow K^*_{\infty}\longrightarrow M^*[V]\longrightarrow V\rightarrow 1$$
by the construction of Proposition \ref{construction}. The group $M^*[V]$ is isomorphic to
the universal mapping class group $\B$ defined in $\cite{fu-ka1}$.

\end{proposition}

\subsubsection*{Profinite completions}

\begin{proposition}
Each strand doubling map  $\partial^{i}_{n}: M^*(0,n)\rightarrow M^*(0,n+1)$ 
extends to a map
$\widehat{\partial^{i}_{n}}: \widehat{M^*}(0,n)\rightarrow \widehat{M^*}(0,n+1)$ between the profinite completions of the corresponding groups.
 \end{proposition}

\vspace{0.1cm}\noindent 
See \cite{ka2} for a proof.

\begin{proposition-definition}

The  cosimplicial $\mathfrak{S}$-extension $(1\rightarrow
  \widehat{K^*}(0,n)\longrightarrow \widehat{M^*}(0,n) \longrightarrow \mathfrak{S}_n\rightarrow
  1, n\geq 1)$, yields the extension of Thompson's group $V$

$$1\rightarrow \widehat{K}^*_{\infty}\longrightarrow \widehat{M^*}[V]\longrightarrow V\rightarrow 1$$
by the construction of Proposition \ref{construction}.

\vspace{0.1cm}\noindent 
The group $\widehat{M^*}[V]$, which will be  denoted $\widehat{\B}$ in the sequel,  is called the $V$-profinite
completion of the universal mapping class group of genus zero. The group 
$\widehat{K}^*_{\infty}$ is an inductive limit of the profinite completions of the pure mapping class groups 
$K^*(0,n)$. The group $\widehat{\cal B}$ is related to ${\cal B}$ via the commutative
diagram

\begin{center}
\begin{picture}(0,0)%
\includegraphics{pro.pstex}%
\end{picture}%
\setlength{\unitlength}{4144sp}%
\begingroup\makeatletter\ifx\SetFigFont\undefined%
\gdef\SetFigFont#1#2#3#4#5{%
  \reset@font\fontsize{#1}{#2pt}%
  \fontfamily{#3}\fontseries{#4}\fontshape{#5}%
  \selectfont}%
\fi\endgroup%
\begin{picture}(3105,1009)(451,-830)
\put(451, 29){\makebox(0,0)[lb]{\smash{\SetFigFont{10}{12.0}{\rmdefault}{\mddefault}{\updefault}$1$}}}
\put(2071, 29){\makebox(0,0)[lb]{\smash{\SetFigFont{10}{12.0}{\rmdefault}{\mddefault}{\updefault}$\widehat{\cal B}$     }}}
\put(2971, 29){\makebox(0,0)[lb]{\smash{\SetFigFont{10}{12.0}{\rmdefault}{\mddefault}{\updefault}$V$}}}
\put(3556, 29){\makebox(0,0)[lb]{\smash{\SetFigFont{10}{12.0}{\rmdefault}{\mddefault}{\updefault}$1$}}}
\put(1081, 29){\makebox(0,0)[lb]{\smash{\SetFigFont{10}{12.0}{\rmdefault}{\mddefault}{\updefault}$\widehat{K}^*_{\infty}$ }}}
\put(2971,-781){\makebox(0,0)[lb]{\smash{\SetFigFont{10}{12.0}{\rmdefault}{\mddefault}{\updefault}$V$}}}
\put(3556,-781){\makebox(0,0)[lb]{\smash{\SetFigFont{10}{12.0}{\rmdefault}{\mddefault}{\updefault}$1$}}}
\put(2071,-781){\makebox(0,0)[lb]{\smash{\SetFigFont{10}{12.0}{\rmdefault}{\mddefault}{\updefault}${\cal B}$     }}}
\put(1081,-781){\makebox(0,0)[lb]{\smash{\SetFigFont{10}{12.0}{\rmdefault}{\mddefault}{\updefault}$K^*_{\infty}$ }}}
\put(451,-781){\makebox(0,0)[lb]{\smash{\SetFigFont{10}{12.0}{\rmdefault}{\mddefault}{\updefault}$1$}}}
\end{picture}

\end{center}
The morphism  ${K}^*_{\infty}\rightarrow \widehat{K}^*_{\infty}$ is induced by
the collection of natural morphisms ${K}^*(0,n)\rightarrow
\widehat{K^*}(0,n)$. All vertical arrows are injective.

\end{proposition-definition}

\subsubsection*{Embeddings $BV\subset \B$ and $\widehat{BV}\subset \widehat{\B}$}

The Artin braid group $B_{n}$ embeds into $M^*(0,n+1)$. The family of embeddings 
 $B_{n}\rightarrowtail M^*(0,n+1)$ indexed by $n$ 
induces in turn an embedding 
$BV\subset \B$. After applying the functor of profinite completion, one obtains morphisms
$\widehat{B_{n}}\rightarrow \widehat{M^*}(0,n+1)$ which are still injective. Hence an embedding 
 $\widehat{BV}\subset \widehat{\B}$.

\subsection{The universal mapping class group in genus zero and 
the Grothendieck-Teichm\"uller group}
Not only does the group  ${\cal B}$ of \cite{fu-ka1} contain all the mapping class groups of compact surfaces of genus zero,  but it also encodes their 
mutual relations. 
Otherwise stated, all the information contained in the modular tower of genus zero 
 is encoded in the group ${\cal B}$. 
We wish to give a remarkable illustration of this idea, as well as an example of application of the formalism of cosimplicial ${\mathfrak S}$-extensions.\\

\vspace{0.1cm}\noindent 
In the article \cite{dr}, Drinfeld defined a group 
containing the absolute Galois group $Gal(\overline{\Q},\Q)$, and called it the 
 Grothendieck-Teichm\"uller  group $\widehat{GT}$. \index{Grothendieck-Teichm\"uller group}
 Explicit formulas of the action of 
 the Grothendieck-Teichm\"uller group (in its $k$-pro-unipotent version) 
 on the $k$-pro-unipotent completions of the braid groups were given.
 Similar formulas hold for the profinite version, and the corresponding action of $\widehat{GT}$ 
extends the natural action of $Gal(\overline{\Q},\Q)$ on the profinite completions $\widehat{B_{n}}$  
of the braid groups. A fundamental remark is that the action of $\widehat{GT}$ on the
$\widehat{B_{n}}$'s respects not only the obvious embeddings 
$\widehat{B}_{n}\rightarrow \widehat{B}_{n+1}$ (``adding one strand''), but also 
the strand doubling maps $\widehat{\partial^{i}}: \widehat{B_{n}} \rightarrow \widehat{B}_{n+1}$. 
The action of $\widehat{GT}$ extends to the profinite completions 
 $\widehat{M^*}(0,n)$ of the mapping class groups of holed spheres, preserving the analogous ``strand doubling maps'',
 (topologically corresponding to a pair of pants glueing)  $\widehat{\partial^{i}}$. The action of  
 $\widehat{GT}$ on the Teichm\"uller tower in genus zero 
 has been systematically studied in \cite{bu-fe-lo-sc-vo}.

\vspace{0.1cm}\noindent 
Since $\widehat{GT}$ acts on the $\widehat{M^*}(0,n)$'s respecting 
the maps $\widehat{\partial^{i}}$, the formalism of cosimplicial $\mathfrak S$-extensions 
enables us easily to prove that it extends 
to the completion $\widehat{\cal B}$.  Denoting  $\alpha,\beta, \pi$ and $t$ 
the images in $\widehat{\cal B}$ 
of the generators of ${\cal B}$, one can easily obtain the following :

\begin{theorem}[\cite{ka2}] The Grothendieck-Teichm\"uller group $\widehat{GT}$ acts  
   on the group $\widehat{\B}$. Moreover, denoting by 
   $(\lambda,f)\in \hat{\Z}\times \hat{F_{2}}$ an element of $\widehat{GT}$, 
   the action reads on the generators as follows:
   $$(\lambda,f).(\pi)=\pi^{\lambda}$$
   $$(\lambda,f).(t)=t^{\lambda}$$
   $$(\lambda,f)(\alpha)=\alpha f(t,\alpha t\alpha^{-1})$$
   $$(\lambda,f)(\beta)=\beta.$$
   \end{theorem}

\vspace{0.1cm}\noindent 
 It is quite remarkable that the action of $\widehat{GT}$ 
 turns out to be completely
 defined by 4 formulas (3 would suffice, as the generator $t$ is redundant). Moreover, 
   if one assumes that a pair $(\lambda,f)$ defines an automorphism of $\widehat{\B}$ given by the above formulas on the generators,
   then  $(\lambda,f)$ satisfies the three relations characterizing its belonging to $\widehat{\B}$.
  Therefore, the presentation of the group $\B$ encodes the definition of the Grothendieck-Teichm\"uller group $\widehat{GT}$.

\section{Problems}

\begin{problem}
Are the groups $F,T, V$ or their generalizations $T^*, \B$ automatic? 
More generally are they (synchronously) combable? In the affirmative 
case this would imply that braided Thompson groups are of type $F_{\infty}$. 
\end{problem}

\vspace{0.1cm}\noindent 
A directly related question is the one for the braided Houghton groups. 

\begin{problem}
Find out whether ${\mathcal M}_{\partial}(Y_n^{\sharp})$ 
are $F_{n-1}$ but not $F_n$.
\end{problem}

\vspace{0.1cm}\noindent 
Degenhardt (\cite{Deg}) proved that the braided Houghton groups 
are $F_{n-1}$ but not $F_n$ for $n\leq 3$ and conjectured that this holds for 
all $n$. This would be a parallel to the results obtained by Brown  
(see \cite{br1}) for the usual Houghton groups $H_n$.  
Progress towards the settlement of this conjecture was made 
by Kai-Uwe Bux in \cite{Bux}.

\vspace{0.2cm}
\noindent
This behavior  is in contrast with the case of the Thompson group $T$ (which is ${\rm FP}_{\infty}$) and its braided version $T^*$ (which is at least ${\rm FP}_3$ (see \cite{fu-ka4}) 
and expected to be ${\rm FP}_{\infty}$).  It is therefore likely that 
${\mathcal M}_{\partial}(Y_n^{\sharp})$ are not combable 
(hence not automatic) although the result of \cite{fu-ka4} would suggest 
that they might be asynchronously combable with quadratic Dehn function. 
If the similarity with the braid groups is pushed one step further 
then the braided Houghton groups should have solvable 
conjugacy problem as well.

\vspace{0.1cm}
\noindent
One does not know which other planar graphs yield 
finitely presented asymptotically rigid mapping class groups. 
One may enlarge 
the category of graphs to that of colored graphs, in which automorphisms 
and almost automorphisms are required to preserve the coloring. 

\vspace{0.1cm}
\noindent 
An interesting class of colored planar trees 
comes from universal coverings of ribbon graphs associated to punctured 
surfaces and 2-dimensional orbifolds. The ribbon structure of the graph 
is a cyclic order around each vertex. 
There is a natural coloring of 
vertices and edges of the graph  and this 
induces a coloring on the universal covering tree. Moreover, the tree 
has a natural embedding in the plane which uses the induced cyclic order   
around the vertices. 

\vspace{0.1cm}
\noindent 
However P. Greenberg (\cite{Gr}) 
showed that asymptotically rigid mapping class groups of universal coverings 
of (colored) ribbon graphs (called projective Thompson groups) 
have infinitely many  generators, as soon as the genus of the surface 
is positive. Moreover, if the genus is zero,  Laget \cite{L} proved that 
the asymptotic mapping class groups are finitely presented groups. 
It seems that the finite presentability holds more generally for 
all the groups obtained from the 2-orbifolds of genus zero. 
The basic example in this respect is the Thompson group $T$ which arises 
from the 2-orbifold associated to the group ${\rm PSL}(2,\Z)$, namely 
the sphere with a cusp, one singular point of order 2 and another one of 
order 3.

\vspace{0.1cm}\noindent 
We proved in \cite{fu-ka3} that the universal mapping class group 
$\M$ of infinite genus is finitely generated. 
Also the Greenberg-Sergiescu acyclic extension ${A}_T$ was proved 
in \cite{fu-se} to be finitely generated. The present methods 
do not permit to settle the following

\begin{problem}
Are the groups $\M$ or $A_{T}$ finitely presented or even 
of type $F_{\infty}$ ? 
\end{problem}

\vspace{0.1cm}\noindent 
A fundamental theorem of Tillmann (see \cite{ti1,ti2}) states that 
the plus construction of the classifying space $B\Gamma_{\infty}$ 
of the infinite genus  mapping class group 
$\Gamma_{\infty}=\lim_{g\to \infty}\Gamma_{g,1}$ is an infinite loop space. 
This is a key ingredient in the Madsen-Weiss proof of the Mumford conjecture  
(see \cite{ma-we,mad}). The second proof of Tillmann's theorem (\cite{ti2}) 
uses Segal's  surfaces category whose objects are compact  
oriented 1-manifolds and whose morphisms are Riemann surfaces cobording 
the respective objects. Tillmann actually shows that the operad 
associated to Segal's category detects the infinite loop spaces. 
In this context it would be interesting to know whether we could replace 
$\Gamma_{\infty}$ by $\M$ and we also propose the following: 

\begin{problem}
Find a geometric interpretation of the acyclicity of the group $A$. 
Is the plus construction of the classifying space 
$B\M$ an infinite loop space? 
\end{problem}

\vspace{0.2cm}\noindent 
In \cite{dgh} the authors considered the Teichm\"uller space 
of  quasiconformal asymptotically conformal structures 
on $\Sigma_{0,\infty}$ minus a disk. They showed that $F$ is the automorphism group of this Teichm\"uller space. 

\begin{problem}
Is there a similar 
interpretation for the groups $\B$ and $\M$, for instance? 
\end{problem}

\vspace{0.2cm}\noindent 
Another setting where the group $T$ acts as a mapping class group is that 
of Greenberg's space ${\mathcal Gr}={\rm CPP}_{\Q}/{\rm PSL}(2,\R)$, which is 
sometimes called the Teichm\"uller space associated to $T$ (see \cite{gr,ma}). 
 Here ${\rm CPP}_{\Q}$ is the space of piecewise ${\rm PSL}(2,\R)$ 
functions $f:P_{\R}^1\to P_{\R}^1$ on the projective circle $P_{\R}^1$ 
whose breaking points are rational. 
The space ${\mathcal Gr}$ is contractible and the action of $F$ on 
it is free. Thus ${\mathcal Gr}/F'$ is a $BF'$ and one could use this model 
to build a homology equivalence $BF'\to \Omega S^3$, by making further use of 
James' model of a loop space. 

\begin{problem}
Is it possible to interpret $T$ as the group of automorphisms of the space ${\mathcal Gr}$ equipped with some convenient structure? 
\end{problem}

\begin{small}

\end{small}

\end{document}